\documentclass[10pt]{amsart}
\usepackage{amssymb,latexsym,amscd,amsmath,amsthm}
\usepackage{mathdots}
\usepackage[matrix,arrow]{xy}
\input{amssym.def}
\input{xy}

\xyoption{all}

\theoremstyle{definition}

\numberwithin{equation}{subsection} 
\newtheorem{guess}{theorem}[subsection]

\newtheorem{rem}[guess]{Remark}
\newtheorem{defi}[guess]{Definition}

\newtheorem{thm}[guess]{Theorem}
\newtheorem{lem}[guess]{Lemma}
\newtheorem{prop}[guess]{Proposition}
\newtheorem{Cor}[guess]{Corollary}

\newcommand{\cI}{\mathcal{I}}

\newcommand{\cQ}{\mathcal{Q}}
\newcommand{\cO}{\mathcal{O}}
\newcommand{\cC}{\mathcal{C}}
\newcommand{\cE}{\mathcal{E}}
\newcommand{\cG}{\mathcal{G}}

\newcommand{\cF}{\mathcal{F}}
\newcommand{\cH}{\mathcal{H}}
\newcommand{\cM}{\mathcal{M}}
\newcommand{\cA}{\mathcal{A}}

\newcommand{\cR}{\mathcal{R}}
\newcommand{\cL}{\mathcal{L}}

\newcommand{\cX}{\mathcal{X}}
\newcommand{\cP}{\mathcal{P}}

\newcommand{\bs}{\mathbf{s}}
\newcommand{\bb}{\mathbf{b}}

\newcommand{\Pic}{\mathrm{Pic}}

\newcommand{\Aut}{\mathrm{Aut}}

\newcommand{\Img}{\mathrm{Img}}
\newcommand{\Spec}{\mathrm{Spec}}

\newcommand{\Ker}{\mathrm{Ker}}

\newcommand{\lra}{\longrightarrow}

\newcommand{\hra}{\hookrightarrow}

\newcommand{\ra}{\rightarrow}
\newcommand{\ol}{\overline}

\newcommand{\ms}{\mapsto}

\newcommand{\ul}{\underline}

\newcommand{\RR}{\mathbb{R}}

\newcommand{\PP}{\mathbb{P}}

\newcommand{\ZZ}{\mathbb{Z}}
\newcommand{\GG}{\mathbb{G}}

\newcommand{\XX}{\mathbb{X}}
\newcommand{\CC}{\mathbb{C}}
\newcommand{\bX}{\mathbb{X}}
\newcommand{\bY}{\mathbb{Y}}

\newcommand{\Hom}{\mathrm{Hom}}

\newcommand{\Id}{\mathrm{Id}}

\begin{document}

\title{ Brauer group of moduli of torsors under Bruhat-Tits group scheme $\cG$ over a curve}


\author[Y. Pandey]{Yashonidhi Pandey}
\thanks{The support of Science and Engineering Research Board under Mathematical Research Impact Centric Support File number: MTR/2017/000229 is gratefully acknowledged.}
\address{ 
Indian Institute of Science Education and Research, Mohali Knowledge city, Sector 81, SAS Nagar, Manauli PO 140306, India}
\email{ ypandey@iisermohali.ac.in, yashonidhipandey@yahoo.co.uk}

\begin{abstract} Let $X$ be a smooth projective curve over the complex numbers. We compute the Brauer group of the moduli stack of Bruhat-Tits group scheme $\cG$-torsors. When $g(X) \geq 3$ we compute the Brauer group of the regularly stable locus of the coarse moduli space of semi-stable $\cG$-torsors.

\end{abstract}
\subjclass[2000]{14F22,14D23,14D20}
\keywords{Parahoric bundles, Moduli stack, Brauer group}
\maketitle


\section{Introduction}

Let $X$ be a smooth projective curve over $\CC$ and let $G$ be a semi-simple, simply-connected and almost simple group over $\CC$.  Pappas and Rapoport have introduced in \cite{prq} a global Bruhat-Tits group scheme $\cG \ra X$. We will work with a slightly less general group scheme following Balaji and Seshadri \cite{bs}. Let $\cR$ denote a finite non-empty set of  points on $X$. Let $X^\circ= X \setminus \cR$. For $x \in X$, let $\mathbb{D}_x=\Spec(\hat{\cO_x})$, let $K_x$ be the quotient field of $\hat{\cO_x}$ and let $\mathbb{D}^\circ_x=\Spec(K_x)$. In this paper,  by a Bruhat-Tits group scheme $\cG \ra X$ we shall mean that  $\cG$ restricted to ${X^\circ}$ is isomorphic to $ X^\circ \times G$, and for any closed point $x \in X$,  $\cG$ restricted to $\mathbb{D}_x$ is a parahoric group scheme (cf \S \ref{gpsch}) such that the gluing functions take values in $Mor(\mathbb{D}^\circ_x,G)=G(K_x)$. Our interest in this paper is to compute the Brauer group of the moduli stack and that of the coarse moduli space of regularly stable torsors under a Bruhat-Tits group scheme $\cG \ra X$ (cf \S \ref{pt}). Recall that the Brauer group for a scheme $Y$ is defined to be the group of equivalence classes of Azumaya algebras with group operation induced by tensor product (cf \S \ref{bms}). For a stack $Y$ however, we will work with the cohomological Brauer group defined as the torsion subgroup of $H^2_{\acute{e}t}(Y,\GG_m)$. Let $\cM_X(\cG)$ denote the moduli stack of $\cG$-torsors on $X$. Our first main result computes the Brauer group of $\cM_X(\cG)$. 

Let us fix a maximal torus  and a Borel subgroup $T \subset B \subset G$ over $\CC$. Let $\mathbf{a}_0$ be the alcove determined by $B$ in the apartment $\cA$ of $T$. This determines a set $\mathbf{S}$ of affine simple roots. For $x \in \cR$, let the restriction of $\cG$ to ${\Spec(\hat{\cO_x})}$ correspond to the facet $\sigma \subset \cA_T$. Let us suppose that after translation by the Iwahori-Weyl group, it corresponds to the facet $\sigma_x$ of $\mathbf{a}$. Let $Z^x$ denote the set of simple affine roots of $\mathbf{S}$ not vanishing on $\sigma_x$. Let $\theta$ denote the highest root of $G/\CC$. Let $\alpha_0$ be the affine simple root corresponding to the highest root of $G$. We set $a^\vee_{\alpha_0}=1$ folllowing \cite[Tits, last para page 650]{tits}. For simple roots $\alpha \in \mathbf{S} \setminus \alpha_0$ of $G/\CC$, let $a^\vee_\alpha$ be integers defined by the relation  \begin{equation} \label{corootcoeff}  \theta^\vee=\sum a^\vee_\alpha \alpha^\vee,  \end{equation} where $\alpha^\vee$ denotes the coroot of $\alpha$. 

\begin{thm}  The cohomological Brauer group of $\cM_X(\cG)$ is $\ZZ^{\oplus \cR}$ modulo $(1,\cdots,1)$ and $\{(0,\cdots,a_{\alpha_x}^\vee,\cdots,0) | \alpha_x \in Z^x, x \in \cR \} $.
\end{thm}
So the Brauer group is always trivial when $\cR$ consists of a single point.

 (Semi)-stabilty of $\cG$-torsors, equipped with weights, have been defined in \cite{bs}. Let $M_X^{rs}(\cG)$ denote the moduli space of regularly stable $\cG$ torsors.  Our second main result computes the Brauer group of $M^{rs}_X(\cG)$ when $g(X) \geq 3$. For simplicity, we first state our theorem in the case of one parabolic point $x$ and a minimal facet $\sigma \subset \mathbf{a}$.  Let $\cG^\sigma \ra X$ be the global Bruhat-Tits group scheme so defined. Let $\Omega_x=\{ w_\alpha | \alpha \in Z^x\}$ denote the fundamental weights of $G$ corresponding to $Z^x$ where for $\alpha_0$ we take the trivial weight. We shall view them as characters on $T$. 

\begin{thm} \label{mt} Let $g_X \geq 3$. The Brauer group of $M_X^{rs}(\cG^\sigma)$ is the quotient of $\Hom(Z_G,\GG_m)$ by $  Z_G \hra T \stackrel{\omega_\alpha}{\lra} \GG_m$ where  $\alpha \in \mathbf{S}$ is the unique affine simple root not vanishing at $\sigma$.
\end{thm} 
 
We now return to the general case.   For any facet $\sigma_x$ we define $  l_x= GCD \{ a_\alpha^\vee | \alpha \in Z^{x}  \}$. We define $ f =LCM \{l_x | x \in \cR \}$.

\begin{thm} \label{mt2}  Let $g_X \geq 3$. For every $x \in \cR$, let $I^x$ denote the set of all possible $|Z^x|$-uple integral solutions $(\cdots, n^{x}_\alpha, \cdots) \in \ZZ^{Z^x}$ to \begin{equation} \label{equationf} \sum_{\alpha \in Z^x} n^{x}_\alpha a_\alpha^\vee = f. \end{equation} For each  $e^x:=(\cdots,n^{x}_\alpha,\cdots) \in I^x$, consider the weight $\omega(e_x)= \sum_{\alpha \in Z^x} n^{x}_\alpha \omega_\alpha$. For each $e=(\cdots,e^x,\cdots) \in \prod_{x \in \cR} I^x$ consider the composite  \begin{equation} \label{wt-oneexp} Z_G \stackrel{diag}{\hra} \prod_{x \in \cR} T \stackrel{\prod_{x \in \cR} \omega(e_x)}{\ra} \prod_{x \in \cR} \GG_m \stackrel{\prod}{\lra} \GG_m.  \end{equation} The kernel of $Br(M^{rs}_X(\cG)) \ra Br(\cM_X(\cG))$ is the quotient of $\Hom(Z_G,\GG_m)$ by the subgroup generated by elements as in (\ref{wt-oneexp}) as one varies over elements $e$ in $\prod_{x \in \cR} I^x$.  In particular, when $\cR=\{x\}$, then $Br(M^{rs}_X(\cG))$ is given by this quotient.

\end{thm}

In \S \ref{crosscheck}, we cross-check  our result with \cite{bd}, \cite[Cor 6.5 Thm 4.5]{bh}  and \cite{bbgn}. In Remark \ref{proofstrategydiff}, we explain the difference between the proof strategies of \cite{bbgn} and \cite{bh}. The latter is developed in this paper. In \S \ref{strategy}, we outline another proof strategy following \cite{bd}. This is the main body of the paper. 

We develop the proof strategy of Biswas-Holla \cite{bh}. A key compuational input is from Biswas-Hoffman \cite{BHf1}.  This gives new proofs of \cite{bd} and \cite{bbgn}. In  \cite{bh}, \cite{bbgn} and \cite{bd}, the map $\Pic(M^{rs}_X(\cG)) \ra \Pic(\cM^{rs}_X(\cG))$ of (\ref{mainseq}) is an isomorphism. These do not hold in the general case. For groups other than those of type $A$ and $C$, the cohomological Brauer group of the stack may be non-trivial. The appearance of equation (\ref{equationf}) is a new feature for moduli space. By \cite{bd}  though the regularly stable locus may change as we vary parabolic weights, but the Brauer group only changes with the quasi-parabolic structure. The same phenomena holds in Theorem \ref{mt2}.


Let $L_{X^\circ}(G)$ denote the ind-group of morphisms from $X^\circ \ra G$. Let $B L_{X^\circ}(G)$ denote its classifying space (cf \S \ref{lxg}). If $H^2_{\acute{E}t}(B L_{X^\circ} G,\GG_m)$ had torsion, then it would appear in the Brauer group too (cf proof of Theorem \ref{mtstackalcove}). So in \S \ref{lxg}, we show that this group is torsion-free. We prove this by passing to the analytic site of $L_{X^\circ}G$. For this reason, we need to restrict ourselves to the complex numbers.

We have restricted the group scheme $\cG$ to the context of Balaji-Seshadri for three reasons. Firstly, the restriction allows us to construct a morphism of stacks $\cM_X(\cG) \ra \cM_X(G)$ (cf \S \ref{constructionstacks}). We could have considered gluing functions to take values in $Aut(G)(\mathbb{D}_x^\circ)$ by enlarging $G$ to non-connected groups. But for simiplicity, we have restricted to this case. Secondly, semi-stability conditions in \cite{bs} or \cite{bbp} are known in these cases. Thirdly, to estimate the  codimesion of non-regularly stable locus in the stack of $\cG$-torsors, we make reduction to $G$-bundle theory by using \cite{bs}. We have restricted ourselves to the case of semi-simple, almost simple, simply-connected group $G$ rather than the case of a general reductive group to avoid technical complications. 
 
We have relagated to the appendix results known to experts for which we could not find  suitable references. It formulates a condition for cohomological descent in case of a morphism from an ind-scheme to an algebraic stack.  In Remarks \ref{whybitetale} and \ref{whycomor}, we explain why we have chosen to work with the Big-\'etale site of ind-schemes and comorphisms between sites of ind-schemes and stacks.

\subsection{Conventions and Notations}
We shall abbreviate complex numbers $\CC$ by $k$ whenever the notions and results we use hold over any algebraically closed field of arbitrary characteristic. Set $X^\circ=X \setminus \cR$. We shall fix a maximal split torus $T \subset G/k$.  We shall often need to appeal to results of groups defined over local fields $K$ and apply them to $G_K$, the base change of $G$ to $K$.  
By {\it subscripts} such as $Y_\sigma$ and $\sigma_Y$ we mean the vanishing conditions and by {\it superscripts} such as $Y^\sigma$ and $\sigma^Y$ we mean the non-vanishing conditions. 

\subsection{Acknowledgement} I had envisaged this paper as a joint work with Norbert Hoffmann. I thank him for discussions. This paper would never have been possible without encouragements and support from Vikraman Balaji, Indranil Biswas and C.S.Seshadri. We warmly thank Najmuddin Fakruddin, Rahul Kumar Singh, Yogish Holla and Jochen Heinloth for helpful discussions and for suggesting references. We thank  Tata Institute of Fundamental Research, Mumbai and Mary Immaculate College, Limerick.

\section{Local group theoretical data} \label{lgpthedata}


Our base field will always be $k$. Let $\cO$ or $A:=k[[t]]$ and $K:=k((t))=k[[t]][t^{-1}]$, where $t$ denotes a uniformizing parameter.
Let $G$ be a {\em semisimple simply connected affine algebraic
group} defined over $k$. All the notions in this section hold for a connected reductive group over a local field. This is called the twisted case in \cite{pradv}. But for simplicity, {\it we will specialize these to the untwisted case namely that of $G_k(K)$ over $K$}. We shall fix a maximal torus $T \,\subset\, G$ and let
$Y(T)\,=\, \Hom(\GG_m,\, T)$ denote the group of all
one--parameter subgroups of $T$. For each maximal torus $T$ of $G$, the {\it
standard affine apartment} $\cA_T$  is an affine space under $Y(T)
\otimes_\ZZ \RR$. We now want to consider the group $G(K)$.

In general, there is no origin in an apartment $\cA_T$ (cf. \cite{bt1}). But
for recalling parahoric groups schemes, {\it we may identify $\cA_T$ with $Y(T) \otimes_\ZZ \RR$}
(see \cite[\S~2]{bs}) by choosing a point $v_0 \in \cA_T$. For a root $r$ of $G$ and an integer $n \in \ZZ$, we get an affine functional
\begin{equation} \label{afffunc} \alpha=r + n : \cA_T \ra \RR, x \ms r(x -v_0) +n.
\end{equation}
These are called the {\it affine roots} of $G$. For any point $x \in \cA_T$, let $Y_x$ denote the set of affine roots vanishing on $x$. For an integer $n\geq 0$, define
\begin{equation} \label{facetdefn}
\cH_n=\{x \in \cA | |Y_x|=n \}.
\end{equation}
A facet $\sigma$ of $\cA_T$ is defined to be the connected component of $\cH_n$ for some $n$. The dimension of a facet is its dimension as a real manifold. Since we fixed an origin $v_0 \in \cA_T$, so one can take convex combinations of points in $\cA$ to define new points in $\cA$. Points in any facet $\sigma$ of $\cA_T$ can be expressed as a convex combination of zero-dimensional facets in the closure of $\sigma$. We will call the coefficients {\it barycentric coordinates}. A point is in the interior of a facet if and only if its barycentric coordinates are strictly positive.

Let $R\,=\,R(T,G)$ denote the root system of $G$ (cf. \cite[p. 125]{springer}). Thus for
every $r \,\in\, R$, we have
the root homomorphism $u_r \,: \, \GG_a\,\lra\, G$ \cite[Proposition 8.1.1]{springer}. For any non-empty subset
$\Theta\,\subset\, \cA_T$, the {\it parahoric subgroup} 
$\cP_\Theta\,\subset\, G(K)$ is
defined (\cite[Page 8]{bs}) as 
\begin{equation} \label{defnpara}
\cP_\Theta = <T(A) \,\ u_r(t^{m_r}A) | m_r\,=\,m_r(\Theta) \,=\, - \lfloor {\rm inf}_{\theta \in \Theta} (\theta,r) \rfloor >.
\end{equation}  Moreover,
by \cite[Section 1.7]{bt} we have an affine flat smooth group scheme $\cG_\Theta\,\lra\,
\Spec(A)$ corresponding to $\Theta$. The set of $K$--valued (respectively,
$A$--valued) points of $\cG_\Theta$ is identified with $G(K)$ (respectively,
$\cP_\Theta$). The group scheme $\cG_\Theta$ is uniquely determined by its
$A$--valued points. For simplicity, in this paper $\Theta$ will always be a {\em facet} of $\cA_T$. In this case, we may check (\ref{defnpara}) for any $\theta$ in the interior $\Theta^\circ$ of $\Theta$.  The pro-unipotent radical $\cP_\Theta^u \subset \cP_\Theta$ is defined as 
\begin{equation} \label{defnparau}
\cP_\Theta^u = <T(1+t) \,\ u_r(t^{1-\lceil (\theta,r) \rceil }A) | \theta \in \Theta^\circ>.
\end{equation} 
Let $\cG_\Theta^u \ra \Spec(A)$ denote the affine flat group scheme corresponding to $\cP^u_\Theta$.

We choose a Borel $B$ in $G/k$ containing $T$. For a facet $\sigma \subset \cA$, we shall denote by $\cG_\sigma$ the parahoric group scheme over $\cO$. Let $\mathbf{a}_0$ denote the unique alcove in $\cA$ whose closure  contains $v_0$ and is contained in the finite Weyl chamber determined by our chosen Borel. This determines a set $\mathbf{S}$ of simple {\it affine roots} $\alpha$.  For $Z \subset \mathbf{S}$, let $\sigma_Z \subset \mathbf{a}$ (resp. $\sigma^Z \subset \mathbf{a}$) be the facet where exactly the roots $\alpha_i \in Z$ vanish i.e such that $\alpha_i(\sigma_Z)=0$ (resp. where exactly $\alpha_i \in Z$ do not vanish i.e $\alpha_i(\sigma^Z) \neq 0$) for all $\alpha_i \in  Z$. Similarly, for a facet $\sigma$ let 
\begin{equation} \label{Ysigma}
Y_\sigma = \{ \alpha | \textrm{ affine root vanishing at} \, \sigma \}.
\end{equation}  It is possible to identify $Y_\sigma$ with a closed sub root system of the root system of $G$ as in \cite{borels} and we will do so. Let $Z_\sigma \subset \mathbf{S}$ be the subset corresponding to $Y_\sigma$ and let
\begin{equation} \label{restroots} 
Z^\sigma \subset \mathbf{S}
\end{equation}   denote the complementary set. We refer the reader to \cite[4.6.12]{bt} for

\begin{prop}  \label{raghu} The root system of the reductive quotient $\cG_\sigma/\cG^u_\sigma$ of the special fiber of $\cG_\sigma$ is given by $Y_\sigma$.
\end{prop}


We refer the reader to \cite[Corollary 4.2.2]{pp}
\begin{Cor}  \label{borelpar} Let $\bs$ and $\bb$ be  facets of  $\mathbf{a}$. Suppose that $\bs$ is in the closure of $\bb$. Let $G_{\bs,\bb}  \subset Y_{\bs}$ be defined by
\begin{equation} \label{Gsb}
G_{\bs,\bb}= \{ \alpha \in Y_{\bs} | 0 \leq \alpha(\bb) \}.
\end{equation}
The group $\cG_{\bb}/\cG^u_{\bs}$ is the parabolic  subgroup of $\cG_{\bs}/\cG^u_{\bs}$ given by $G_{\bs,\bb}$.
\end{Cor}

\subsection{Loop groups and their flag varieties} \label{lgpflv}
We need to recall some facts from \cite[\S 2, \S 8]{pradv} about affine Weyl group and Schubert varieties. In loc. cit. these are stated for arbitrary reductive groups that split over a tamely ramified extension $\tilde{K}/K$. Our interest is in applying these results to $G_K$. Since $\tilde{K}=K$ for us, so we shall specialize to this case.

Let $k$ be a field. Let $\cG$ be a flat affine group scheme of finite type over $k[[t]]$.
For a $k$-algebra $R$ let $R[[t]]$ denote the formal power series ring and $R((t))=R[[t]][t^{-1}]$ the field of Laurent polynomials with coefficients in $R$. Recall that the loop group $L\cG$ of $\cG$ represents the functor mapping $R$ to $\cG(R((t)))$. It is represented by an ind-affine scheme. Recall that the jet group $L^+\cG$ represents the functor mapping $R$ to $\cG(R[[t]])$. It is represented by a closed subscheme of $L \cG$ which is affine.  For a facet $\sigma$, the quotient of fpqc-sheaves  
\begin{equation} \label{flagv}
\cF l_\sigma= L\cG_{\sigma}/L^+\cG_{\sigma}
\end{equation}
is the flag variety associated to $\sigma$. It
is represented by an ind-scheme that is ind-projective over $k$.  It represents the functor that to $R$ associates the set of pairs $(\cE,\theta)$ where $\cE$ is a $\cG_{\sigma}$-torsors on $Spec R[[t]]$ together with a trivialization given by $\theta$ over $Spec R((t))$.  We recall

\begin{thm} \label{etlocsec}  \cite[Theorem 1.4]{pradv} Let $R$ be a strictly henselian ring over $k$. for any point $\Spec(R) \ra \cF l_{\sigma}$, we have $\Spec(R) \times_{\Spec(k)} L^+\cG_{\sigma} \simeq \Spec(R) \times_{\cF l_{\sigma}} LG$. 
\end{thm}

\begin{prop} \cite[Prop 10.1]{pradv} \label{PicF} There is an isomorphism
$ Pic(\cF l_{\mathbf{a}}) \simeq \ZZ^{\mathbf{S}}$
 defined by the degrees of the restrictions to $\PP^1_\alpha= L^+\cG_{\sigma^\alpha}/L^+\cG_{\mathbf{a}} \hra \cF l_{\mathbf{a}}$. 
\end{prop}

\subsection{Weyl groups and Schubert varieties} \label{wgsv}

Recall that the Kottwitz homomorphism $\kappa_H$ is defined for a reductive group $H$ over an arbitrary local field $L$ (cf \cite[2.a.2, page 127]{pradv}). Let us specialize to cases of interest to us. 
By (loc. cit. Step 1.), when $H=\mathbb{G}_m$ then writing $f \in LH(k)=k((t))^\times$ in the form $f=t^j u$ where $u \in k[[t]]^\times$, we have $\kappa_H(f)=j$. For the maximal split torus $T_K \subset G_K$,
 $\kappa_T$ is just the sum of $\kappa_{\mathbb{G}_m}$ in each coordinate. Let $$T(K)_1=\ker(\kappa_T).$$  Let $N=N(T)$ be the normaliser of $T$.
One defines the {\it Iwahori-Weyl group associated to $T$} as the quotient
\begin{equation} \label{IW}
\tilde{W} = N(K)/T(K)_1.
\end{equation}
Setting $W_0=N(K)/T(K)$ as the relative Weyl group, one has 
\begin{equation} \label{relativeWgp}
0 \ra X_*(T) \ra \tilde{W} \ra W_0 \ra 0.
\end{equation}
The choice of a point $v_0 \in \cA$ splits the above sequence. We quote

\begin{prop} \label{IwI} \cite[Appendix Prop 8]{pradv} Let $I$ denote an Iwahori subgroup of $G_K(K)$. Then $G_K(K)=I N(K) I$ and the map $In I \ms n \in \tilde{W}$ induces a bijection $I \setminus G(K) / I \simeq \tilde{W}$.
More generally, let $P$ and $P'$ be parahoric subgroups associated to facets $F$ (resp $F'$) contained in the apartment corresponding to $T$. Let $\tilde{W}^P:=(N(K) \cap P)/T(K)_1$ and similarly define $\tilde{W}^{P'}$. Then
$P \setminus G_K(K) /P' \simeq \tilde{W}^P \setminus \tilde{W} / \tilde{W}^{P'}$.
\end{prop}


The {\it affine Weyl group associated to $T$} is the Iwahori-Weyl group $W_a$ of the simply-connected cover $G_{sc,K}$ of the derived subgroup $G_{der,K}$ of $G_K$. Since we have assumed that $G/k$ is simply-connected and semi-simple, so for us
\begin{equation} \label{affineWeyl}
W_a=\tilde{W}.
\end{equation}
 

It is known that $W_a$ is a Coxeter group. Then (cf loc. cit. page 152) $\mathbf{S}$ generate $W_a$. For each $w  \in W_a$, its length $l(w)$ is the minimal number of factors in a product of $s_i$'s representing $w$. Any product realizing the minimum number of factors of $w$ is called a {\it reduced decomposition} of $w$. Let us recall the {\it Bruhat order} (cf loc. cit. page 152). One says that $w' \leq w$ if $w'$ is obtained by replacing some factors of a reduced decomposition of $w$ by $1$. It is a fact that the set of such $w'$ is independent of reduced decomposition of $w$ chosen.

We now recall Schubert cells and varieties (cf loc. cit. Definition 8.3 and \cite[\S 2]{zhu}). Let $Y, Y' \subset \mathbf{S}$ be two subsets.  Then the $L^+ \cG_{\sigma}$-orbit of $\cF l_{\sigma'}$ are parametrized by $W^Y \setminus \tilde{W} /W^{Y'}$ where $W^Y$ is the Weyl group of $\cG_{\sigma_Y} \otimes_{\cO} k$. The {\it Schubert cell} ${}^Y C_w^{Y'}$ (or ${}^\sigma C_w^{\sigma'}$) is the reduced orbit $L^+ \cG_{\sigma} n_w \subset \cF l_{\sigma'}$,
where $n_w \in N(K)$ a representative of $w  \in \tilde{W}^Y \setminus \tilde{W}/ \tilde{W}^{Y'}$ (cf Proposition \ref{IwI}). The {\it Schubert variety} ${}^YS_w^{Y'}$  ( or ${}^\sigma S_w^{\sigma'}$) is the reduced scheme with underlying set the Zariski closure of ${}^Y C_w^{Y'}$. 

{\it When $\sigma_Y=\mathbf{a}$ then we shall simply write $C^{Y'}_w$ (or $C^{\sigma'}_w$) and $S^{Y'}_w$ (or $S^{\sigma'}_w$). Further when $\sigma_Y= \sigma_{Y'}=\mathbf{a}$, then we shall write $C_w$ and  $S_w$.}



\subsection{The group scheme} \label{gpsch} Let $\cR \subset X$ be a non-empty finite set of closed points. For each $x \in \cR$, we choose a facet $\sigma_x \subset \cA_T$. Let $\cG_{\sigma_x} \ra \Spec(\hat{\cO_x})$ be the parahoric group scheme corresponding to $\sigma_x$. Let $X^\circ = X \setminus \cR$. For $x \in X$, let $\mathbb{D}_x=\Spec(\hat{\cO_x})$, let $K_x$ be the quotient field of $\hat{\cO_x}$ and let $\mathbb{D}^\circ_x=\Spec(K_x)$. {\it In this paper,  by a Bruhat-Tits group scheme $\cG \ra X$ we shall mean that  $\cG$ restricted to ${X^\circ}$ is isomorphic to $ X^\circ \times G$, and for any closed point $x \in X$,  $\cG$ restricted to $\mathbb{D}_x$ is a parahoric group scheme $\cG_{\sigma_x}$ such that the gluing functions take values in $Mor(\mathbb{D}^\circ_x,G)=G(K_x)$.} The restriction on gluing functions is useful in the construction in \S \ref{constructionstacks}. Further, this is also the setup of \cite[Balaji-Seshadri Defn 5.2.1]{bs} which is our main reference for (semi)-staiblity conditions. Let us remark the setup of  \cite[Heinloth]{heinloth} and \cite[Zhu]{zhu} are the same. They consider more general group schemes which may not be split over the function field of $X$. The global conditions over $\cG$ demanded in \cite{heinloth} are satisfied in our case (cf. \cite[Introduction]{heinloth}). By \cite[Lemma 5]{heinloth} it is always possible to glue $X^\circ \times G$ with $\{ \cG_x | x \in \cR\}$ along the fpqc cover $\{X^\circ \} \cup \{ \Spec(\hat{\cO_x}) | x \in \cR \}$ of $X$ to get a group scheme $\cG \ra X$. Of course, $\cG$ depends on the gluing data. For instance, if $E \ra X$ is a principal $G$-bundle, then $\Aut(E) \ra X$ is such a group scheme. Its restriction to $X^\circ$ and $\Spec(\hat{\cO_x})$ is always the trivial group scheme, while it may or may not be trivial over $X$. Similarly, let $x_0 \in X$ be a closed point, and let $\cM$ denote the stack of vector bundles of rank $r$ on $X$ and determinant $\cO_X(-d x_0)$ where $0 \leq d <r$. Then for any $V \in \cM$, the adjoint group scheme $Aut(V) \ra X$ is obtained by gluing $X^\circ \times G$ with $\cG_\sigma$ where $\sigma$ is the unique vertex of the alcove $\mathbf{a}$ where only the affine simple root $\alpha_d$ does not vanish.


\subsection{$\Pic(\cF l_\sigma)$ in terms of characters on $T$} Recall (cf \cite[ page 13]{ls} or \cite[page 131]{sorger}) that one has a canonical central extension  
\begin{equation} \label{LGext}
1 \ra \GG_m \ra \tilde{L}G \ra LG \ra 1.
\end{equation} 

Let $\tilde{L}^+\cG_{\sigma}$ denote 
the inverse image of $L^+\cG_{\sigma}$ in $\tilde{L}G$. When $\sigma=\mathbf{a}$ this is called {\it the Iwahori subgroup of Kac-Moody theory} (cf \cite[page 491, 13.2.2.]{sk} and \cite[\S 10]{pradv}). Then the fpqc-quotient $\tilde{L}G/\tilde{L}^+\cG_{\sigma}$ also identifies with $\cF l_{\sigma}$. From the $\tilde{L}^+\cG_{\sigma}$-fibration $\tilde{L}G \ra \cF l_{\sigma}$ using any character $\chi$ of $\tilde{L}^+\cG_{\sigma}$ we can make the line bundle as follows:
\begin{equation} \label{glb}
L_{\chi}:= \tilde{L}G \times_{\chi} \GG_m \ra  \tilde{L}G/\tilde{L}^+ \cG_{\sigma} = \cF l_{\sigma}.
\end{equation}
Conversely, one knows that line bundles on $\cF l_{\mathbf{a}}$ are all homogenous line bundles coming from characters of the Iwahori subgroup in Kac-Moody theory (cf \cite[\S 10]{pradv})). Thus their pull-back to $\tilde{L}G$ is trivial. So
\begin{equation} \label{chartilde} \Pic(\cF l_{\sigma}) \simeq \XX^*(\tilde{L}^+ \cG_{\sigma}). 
\end{equation}
Let us make this more explicit. Then $\Pic(\cF l_{\sigma})$ is a free abelian group on $Z^\sigma$ which is the set of affine simple roots in $\mathbf{S}$ not vanishing on $\sigma$. Let $\alpha \in Z^\sigma$ be a simple affine root. Let $\tilde{\omega}_\alpha$ be the fundamental weight of $\tilde{L}G$ corresponding to $\alpha$. Then $\tilde{\omega}_\alpha$ gives a character on $\tilde{L}^+ \cG_{\sigma}$. In this way, the character group in (\ref{chartilde}) is a free abelian group on the set $\tilde{\omega}_{Z^\sigma}$ of fundamental weights of $\tilde{L}G$ corresponding to $Z^\sigma$.

The extension (\ref{LGext}) splits uniquely over $L^+(G_A)$ (cf \cite{sorger}, \cite[\S 10]{pradv}). Since $L^+\cG_\mathbf{a} \hra L^+(G_A)$, so its restriction to $L^+\cG_\mathbf{a}$ also splits. Let $\tilde{T}$ denote the inverse image of $T$ in $\tilde{L}^+\cG_{\mathbf{a}}$. 
Then sequence $1 \ra \GG_m \ra \tilde{T} \ra T \ra 1$ also splits by the following lemma.

\begin{lem} \label{Tinc} Recall that in \S \ref{lgpthedata} we fixed $T/k \subset B/k \subset G/k$ and $\mathbf{a}$. We have a canonical inclusion of group schemes $T \hra L^+ \cG_{\mathbf{a}}$.
\end{lem} 
\begin{proof} Let $A=k[[t]]$. So we have an inclusion $k \hra A$ giving $p: Spec(A) \ra Spec(k)$.  At the closed fiber, we have an inclusion of $k$-group schemes $T_k \hra G_k$. This gives an inclusion of $A$-group schemes $T_A \hra G_A$ by pulling-back to $Spec(A)$. Consider the reduction at the special fiber: the sections of $T_A$ land in $T_k$ which is contained in the Borel $B_k$ of $G_k$ by hypothesis. Since the Iwarhori group scheme $\cG_{\mathbf{a}}$ is defined as the pull-back of $B_k \subset G(k)$ by the evalution at zero map, so we have a factorization
$
\xymatrix{
T_A \ar@{.>}[r] & \cG_{\mathbf{a}} \ar@{^{(}->}[r] & G_A.
}$
 Since $T_A =p^*T_k$, so by adjunction we have $Mor_A(p^*T_k,\cG_{\mathbf{a}})=Mor_k(T_k,p_*\cG_{\mathbf{a}})$. But $p_*\cG_{\mathbf{a}}=L^+\cG_{\mathbf{a}}$. 
\end{proof}

  So we get in inclusion $T \hra \tilde{T}$ and the following sequence splits 
\begin{equation} \label{centralsplitting}
0 \ra \XX^*(T) \ra \XX^*(\tilde{T}) \stackrel{c}{\lra} \ZZ \ra 0.
\end{equation}

By \cite[(10.4)]{pradv}, we have  \begin{equation} \label{4equi}
\ZZ^{\mathbf{S}}  \simeq \Pic(\cF l_{\mathbf{a}}) \simeq \XX^*(\tilde{L}^+\cG_{\mathbf{a}})  \simeq \XX^*(\tilde{T}).
\end{equation}

Let us make the identification $\ZZ^{\mathbf{S}} \simeq \XX^*(\tilde{T})$ more explicit. Let
\begin{equation} \label{roots}
\mathbf{S}=<\alpha_0=\delta - \theta, \alpha_1, \cdots, \alpha_n>
\end{equation} 
where $\delta$ pairs with the central $\GG_m \hra \tilde{T}$ and is the trivial character on the image of $T$ in $\tilde{T}$ under the splliting (\ref{centralsplitting}), $\theta$ is the highest root and the remaining are the simple roots of $G/\CC$ viewed as affine roots of $LG$. Let us take as coroots
\begin{equation} \label{coroots}
\mathbf{S^\vee}=<c- \theta^\vee,  \alpha_1^\vee, \cdots, \alpha_n^\vee>
\end{equation}
where $c$ is the image of the central $\GG_m$ in $\tilde{T}$.  We normalize the Cartan-Killing form $(|)$ on $\XX^*(T) \otimes_\ZZ \RR$ so that $(\theta|\theta)=2$. The usual Cartan matrix has $(i,j)$-entry $<\alpha_i^\vee,\alpha_j>$ where $1 \leq i,j \leq n$. This is extended by adding a $0$-th row and column and defining $A_{0,0}=2$, $A_{0,j}=-\alpha_j(\theta^\vee)$ and $A_{j,0}=-\theta(\alpha_j^\vee)$ for $1 \leq j \leq l$. The choice of $\mathbf{S}$ and $\mathbf{S^\vee}$ realizes the extended Cartan matrix so defined (cf \cite[XIII, page 484]{sk}). We will take 
\begin{equation} \label{epsilonalpha}
\{\epsilon_\alpha \in \XX^*(\tilde{T}) | \alpha \in \mathbf{S}\} \, \, \text{ dual to the coroots} \, \,  \mathbf{S^\vee}.
\end{equation}  This makes the identification $\ZZ^{\mathbf{S}}=\XX^*(\tilde{T})$ explicit.  

Now we wish to express $\epsilon_\alpha$ in terms of the fundamental weights of $G/\CC$ and $\delta$. Under this identification, 
 let $\{ \omega_\alpha \in \XX^*(T)| \alpha \in \mathbf{S} \setminus{\alpha_0} \}$ be a $\ZZ$-basis of weights of  $G/\CC$. We will extend $\omega_\alpha$ as characters on $\tilde{T}$ by declaring them to be trivial on the central $\GG_m$ and denote this extension by $\omega_\alpha|_{\tilde{T}}$. For $\alpha \in \mathbf{S} \setminus \alpha_0$,  let $a_\alpha^\vee$ be positive integers defined by the relation $\theta^\vee=\sum a_\alpha^{\vee} \alpha^{\vee}$. 
 \begin{prop} \label{resofchar} If $\alpha \in \mathbf{S} \setminus{\alpha_0}$ then  $\epsilon_\alpha|_T=\omega_\alpha$ and  $\epsilon_0|_T$ is trivial.
\end{prop}
Let us remark in passing that this agrees with \cite[(10.5)]{pradv}, in other words we  have $(1,\cdots, a_\alpha^{\vee},\cdots) A=(\cdots, 0,\cdots)$. 
\begin{proof} We check that weights set as follows are dual to the coroots (\ref{coroots}) 
 \begin{equation} \label{epsilon-reln}
 \epsilon_0=\delta \quad, \quad
 \epsilon_\alpha=\omega_\alpha|_{\tilde{T}} + a_\alpha^\vee \delta.
 \end{equation}
 \end{proof}
 
\subsection{Central charge of line bundles on flag varieties}
For each flag variety $\cF l_{\sigma}$, the obstruction to lifting the action of $LG$ to $Pic(\cF l_{\sigma})$ defines a central extension $\tilde{L}G$ of $LG$. The weight of the action of central $\GG_m$ on line bundles defines   {\it central charge} homomorphism (cf \cite[Remark 10.2]{pradv})  \begin{equation} \label{centralcharge} c_\sigma: Pic(\cF l_{\sigma}) \ra \ZZ. 
\end{equation} 
It satisfies the property that $ker(c_\sigma)= \XX^*(\cG_{\sigma} \otimes k)$.
By \cite[page 11, last paragraph]{zhu} the central charge on $\Pic(\cF l_{\sigma})$ can be defined after pull-back to $\Pic(\cF l_{\mathbf{a}})$. On $\cF l_{\mathbf{a}}$ it sends the line bundle $L_{\epsilon_\alpha}$ to $a_\alpha^\vee$ (cf \cite[\S 10]{pradv}).
 
\section{Preliminaries on Global constructions} \label{globalconst}
We begin with a very general cadre to be able to reconcile the references.

Let $C \ra S$ be a  smooth curve. Let $\cG \ra C$ be a smooth {\it affine} group scheme over $C$. Let $\Spec(R) \ra S$ be a $S$-scheme and $y: \Spec(R) \ra C$ be a $R$-point of $C$. Let $C_R:=Spec(R) \times_S C$. Let $\Gamma_y \subset C_R$ denote the graph of $y$ and $\hat{\Gamma}_y$ the completion of $C_R$ along $\Gamma_y$. So we have the closed inclusion $\Gamma_y \hra \hat{\Gamma}_y$. Let $s_R^y: \Spec(R) \ra C_R$ denote the section corresponding to $y$.

The {\it jet group} $\cL^+ \cG$ of $\cG$ is defined as follows.  For any $S$-scheme $R$, 
\begin{equation}
\cL^+ \cG(R)= \{ (y,\beta)| y: Spec R \ra C, \beta \in \cG(\hat{\Gamma}_y) \}.
\end{equation}
It is representable by a $S$-scheme.

The {\it loop group $\cL \cG$} of a global group scheme $\cG \ra C$  is defined as follows. Let $\hat{\Gamma}_y^\circ = \hat{\Gamma}_y \setminus \Gamma_y$. For any $S$-scheme $R$, 
\begin{equation}
\cL \cG(R)= \{ (y,\beta)| y: Spec R \ra C, \beta \in \cG(\hat{\Gamma}_y^\circ) \}.
\end{equation}
It is representable by an ind-scheme over $S$.  It represents the functor that to a $S$-scheme $R$ associates $(y,\cE,\alpha,\beta)$ where
\begin{enumerate}
\item $y: \Spec(R) \ra C$ is a $S$-point
\item  $\cE \ra C_R$ is a $\cG$-torsor
\item $\alpha$ is a trivialization of $\cE$ restricted to ${C_R \setminus s_R^y}$
\item $\beta$ is a trivialization of  $\cE$ restricted to $\hat{\Gamma}_y$.
\end{enumerate}

Let us recall the {\it twisted  affine flag manifold}.  Let $Gr_\cG \ra C$ represent the functor that to every $S$-scheme $\Spec(R)$ associates $(y,\cE,\alpha)$ where 
\begin{enumerate} 
\item $y:Spec(R) \ra C$ is a $S$-point, 
\item $\cE$ is a $\cG$-torsor on $C \times_S \Spec(R)$ and, 
\item $\alpha$ is a trivialization of $\cE$ restricted to ${C_R \setminus s_R^y}$. 
\end{enumerate}
It is a formally smooth ind-scheme over $C$. It can be constructed by sheafifying the quotient $\cL \cG/\cL^+ \cG$ in the fpqc-topology. The natural map $\cL \cG \ra Gr_{\cG}$ forgets the section $\beta$. One also has a natural forgetful map $p: Gr_{\cG} \ra Bun_C(\cG)$. 

Let us specialize to the case $S=Spec(k)$, $\Spec(R)=S$ and $y:Spec(S) \ra C$ is a point $x \in C(k)$. We shall denote the fiber at $x \in C(k)$ of $Gr_{\cG} \ra C$ by 
\begin{equation} \label{grx}
Gr_{\cG,x}
\end{equation} in this special situation. Similarly, let $(\cL \cG)_x$ and $(\cL^+ \cG)_x$ denote the fibers at $x$ of $\cL \cG \ra C$ and $\cL^+ \cG \ra C$. They are isomorphic to $L (\cG|_{\hat{\Gamma}^\circ_x})$ and $L^+ (\cG|_{\hat{\Gamma}_x})$ respectively. Let $(\cL \cG)_x/(\cL^+ \cG)_x$ denote the sheafified quotient in fpqc-topology. Thus we have a natural isomorphism  $Gr_{\cG,x} = (\cL \cG)_x/(\cL^+ \cG)_x$.
These notations agree with \cite[\S 2]{heinloth}.

Let $\sigma_x$ denote the facet defining the parahoric group scheme $\cG_x$  obtained by restricting $\cG \ra C$ to $\hat{\Gamma}_x$. Thus we have a natural isomorphism  
\begin{equation} \label{grxflagv}
Gr_{\cG,x} = (\cL \cG)_x/(\cL^+ \cG)_x \stackrel{(\ref{flagv})}{=} \cF l_{\sigma_x}.
\end{equation}

 We quote
\begin{thm} \label{constantcy} \cite[Zhu, Prop 4.1]{zhu} Let $C$ be a smooth curve over $k$. Let $\cG \ra C$ be a Bruhat-Tits group scheme such that over the generic point it is almost simple, absolutely simple, and simply-connected. Let $L$ be a line bundle on $Gr_{\cG}$. Then the function $c_L$ that associates to every $x \in C(k)$, the central charge (see ( \ref{centralcharge})) of the restriction of $L$ to $Gr_{\cG,x}$ is constant.
\end{thm} 

\subsection{Central charge of line bundles on moduli stack of $\cG$-torsors} \label{centralchargestack}
By \cite[\S  6, 1st paragraph]{heinloth} the central charge of line bundles on $\cM_X(\cG)$ is defined at an arbitrary point $z \in X$ after pulling them back to $Gr_{\cG,z}$ (cf (\ref{grx})) and then applying the central charge homomorphism. We make this more precise as follows. Consider the diagram
 \begin{equation}
 \xymatrix{ & Gr_{\cG} \ar[ld] \ar[rd] & \\
 X & & \cM_X(\cG)
 } 
 \end{equation}
 \begin{enumerate}
 \item By (\ref{grxflagv}), $Gr_{\cG,z}$ is naturally isomorphic to  $\cF l_{\sigma_z}$.  
 \item For computing central charge, it suffices to pull-line bundles to $\cF l_{\mathbf{a}}$.
 \item By \cite[Prop 4.1]{zhu}, this definition is independent of $z \in X$.
 \end{enumerate}

\section{Schubert varieties $S_w$}

\subsection{Bott-Samelson-Demazure-Hansen desingularization of $S_w$.} Let us fix an alcove $\mathbf{a}$. Let $w \in \tilde{W}$ and fix a reduced decomposition $w=s_{i_1}\cdots s_{i_r}$, which we will denote by $\tilde{w}$, in terms of the simple reflections about the walls of $\mathbf{a}$. For a reflection $s_{i_j}$, let $\cG_{i_j}$ denote the corresponding parahoric group scheme. The Demazure variety $D(\tilde{w})$ is defined as 
\begin{equation}
D(\tilde{w})=L^+\cG_{i_1} \times^{L^+\cG_{\mathbf{a}}} \cdots \times^{L^+ \cG_{\mathbf{a}}} L^+\cG_{i_r}/L^+\cG_{\mathbf{a}}.
\end{equation}

\begin{prop} \label{anyw} The Schubert cells  $C_w$ (cf \S  \ref{wgsv}) are affine spaces over $k$.
\end{prop}
\begin{proof}  Recall (cf \cite[Prop 8.8, (8.13)]{pradv}) that the Demazure variety $D(\tilde{w})$ is the quotient of $ \prod_{j=1,\cdots r} L^+ \cG_{i_j}$ by $(L^+\cG_{\mathbf{a}})^r$ by the right action
\begin{equation} \label{rightaction}
(p_1,\cdots,p_r) (b_1, \cdots,b_r) = (p_1b_1, b_1^{-1} p_2 b_2, \cdots , b_{r-1}^{-1} p_r b_r).
\end{equation}
The variety $D(\tilde{w})$ has two open subsets $D(\tilde{w})^\circ$ and $D^\circ(\tilde{w})$ which  are defined by the condition that the last co-ordinate does not belong to $L^+\cG_{\mathbf{a}}$ and the condition that  $(p_1, \cdots,p_r) \in D^\circ(\tilde{w})$ if no $p_i$ belongs to $L^+\cG_{\mathbf{a}}$ respectively. The open subset $D^\circ(\tilde{w})$ maps isomorphically onto the Schubert cell $C_w$ (cf \cite[proof of Prop 9.6]{pradv}). When $w$ consists of a single reflection say $s_{i_r}$ the fpqc-quotient $L^+ \cG_{i_r}/L^+\cG_{\mathbf{a}}$ identifies with $\PP^1_k$ by \cite[Prop 8.7]{pradv} and thus the open subset $C_w$ equals $\mathbb{A}^1_k$. More generally, let us write $w=w' s_{i_r}$. We want to show that $D^\circ(\tilde{w}) \ra D^\circ(\tilde{w}')$ is a line bundle. To this end, let us consider the projection map $\prod_{j=1}^r L^+ \cG_{i_j} \ra \prod_{j=1}^{r-1} L^+ \cG_{i_j}$ which forgets the last coordinate. Quotienting by $L^+\cG_{\mathbf{a}}^{r-1}$ according to the right action (\ref{rightaction}), we get a principal $L^+ \cG_{i_r}$ bundle on $D(\tilde{w}')$. Now going modulo $L^+\cG_{\mathbf{a}}$, we see that $D(\tilde{w}) \ra D(\tilde{w}')$ is a $\PP^1_k$-bundle. Let us restrict this map to the open subset $D(\tilde{w})^\circ \subset D(\tilde{w})$. Then $D(\tilde{w})^\circ \ra D(\tilde{w}')$ becomes a line bundle. Since we have $$D^\circ(\tilde{w})=   D^\circ(\tilde{w}') \times_{D(\tilde{w}')} D(\tilde{w})^\circ,$$
so $D^\circ(\tilde{w}) \ra D^\circ(\tilde{w}')$ is also a line bundle. By induction, we may suppose that $D^\circ(\tilde{w}')$ is an affine space over $k$. So $D^\circ(\tilde{w})$ is the trivial line bundle by the Quillen-Suslin theorem (\cite{quillen}). Thus $C_w= D^\circ(\tilde{w})$ is also an affine space over $k$.

\end{proof}

\begin{prop} \label{demazurevanish} We have $H^n_{et}(D(\tilde{w}),\cO)=0$ for $n \geq 1$.
\end{prop}
\begin{proof} This follows because $D(\tilde{w})$ is an iterated $\PP^1$-fibration. Indeed, one can induct on the reduced length of $w$. The base case of length one is clear because $D(\tilde{w})=\PP^1$. Suppose that length of $w'$ is one less than that of $w$ and $w=w's$ where $s$ is a reflection. Now consider the natural projection $\pi: D(\tilde{w}) \ra D(\tilde{w}')$, which is a $\PP^1$-fibration by the proof of Prop \ref{anyw}. We have $\pi_* \cO_{D(\tilde{w})}=\cO_{D(\tilde{w}')}$ and $R^i (\pi_w)_* \cO_{\tilde{w}}=0$ for $i>0$. Now $H^n(D(\tilde{w}'), R^0(\pi)_* \cO_{D(\tilde{w})})=0$ by induction hypothesis. Thus  by considering the spectral sequence corresponding to $\Gamma_{D(\tilde{w}')} \circ \pi_*=\Gamma_{D(\tilde{w})}$, it follows that $H^n(D(\tilde{w}), \cO_{D(\tilde{w})})=0$.
\end{proof}

\subsection{Results on Schubert varieties}

\begin{prop} \label{et=sing} For Schubert varieties $S_w$ we have $H^n_{\acute{e}t}(S_w,\cO)=0$ for $n \geq 1$. Let $S:=\prod S_w$ denote a finite product of Schubert varieties. We have isomorphisms
\begin{equation} \label{resultschubert}
 H^2_{an}(S, \cO^*)_{tor} \simeq H^3_B(S,\ZZ)_{tor}
\end{equation}

\end{prop}
\begin{proof}  Let $\pi_w : D(\tilde{w}) \ra S_w$ denote the projection map. Since we are in characteristic zero, by \cite[Theorem 8.4]{pradv}, $S_w$ is normal. By \cite[(9.16), Prop 9.7 d)]{pradv} for $i >0$, we have $R^i(\pi_w)_* \cO_{D(\tilde{w})}=0$ and $(\pi_w)_*(\cO_{D(\tilde{w})})=\cO_{S_w}$.
Consider the composition of functors $\Gamma_{S_w} \circ \pi_w = \Gamma_{D(\tilde{w})}$. Let us consider the associated Grothendieck spectral sequence
$$H^{p}_{\acute{e}t}(S_w, R^q(\pi_w)_* (\cO_{D(\tilde{w})}) ) \implies H^{p+q}_{\acute{e}t}(D(\tilde{w}), \cO_{D(\tilde{w})}).$$
 For every $n\geq 1$ we have $E^{n-2,1}_2 \ra E^{n,0}_2 \ra E^{n+2,-1}_2$. Now $E^{n-2,1}_2=0$ because $R^1(\pi_w)_* (\cO_{D(\tilde{w})})=0$. Thus $E^{n,0}_2 =E^{n,0}_3$. Now we have $E^{n-3,2}_3 \ra E^{n,0}_3 \ra E^{n+3,-2}_3$. Since $E^{n-3,2}_2=0$, so $E^{n-3,2}_3$ being a sub-quotient also vanishes. Thus $E^{n,0}_3=E^{n,0}_4$. Reasoning like this, we see that $E^{n,0}_2=E^{n,0}_3=\cdots=E^{n,0}_\infty$. Thus we get the inclusion $H^n(S_w, R^0(\pi_w)_* (\cO_{D(\tilde{w})}))=E^{n,0}_2 =E^{n,0}_\infty \hra H^n(D(\tilde{w}),\cO)$. Now the right term vanishes by Prop \ref{demazurevanish}. This shows $H^n_{\acute{e}t}(S_w,\cO)=0$ for $n \geq 1$.

 Now consider the exponential sequence $0 \ra \ZZ \ra \cO_{an} \ra \cO_{an}^* \ra 0$ on $S_w(\CC)$. The second isomorphism of (\ref{resultschubert}) follows because $H^n_{an}(S(\CC),\cO)=0$ for $n>0$. This can be proven by induction on the number of factors. For the case of one factor, this follows by reasoning as before for $H^n_{\acute{e}t}(S_w,\cO)$ but now in the complex analytic category. The general case follows by inducting on the number of factors and considering the Leray sequence for the fibration that forgets one factor.
\end{proof}

\begin{prop} \label{compSw} Let $S:=\prod S_w$ denote a finite product of Schubert varieties. We have $H_3(S,\ZZ)=0$, $H_1(S,\ZZ)=0$ and $H_2(S,\ZZ)$ is a free abelian group on $2$-cells in $S_w$. We have $H^2_{\acute{e}t}(S,\GG_m)_{tor}=0$.
\end{prop} 
\begin{proof} 

By Prop \ref{anyw} the Schubert cell $C_w$ is an affine space over $k$. Since $k=\CC$, so $S_w$ has the structure of a $CW$ complex with only even dimensional cells. Now the statements on homology groups follow.  
We have $H^3_B( S,\ZZ)=0$. This follows from the universal coefficient theorem, since we have 
$$0 \ra Ext^1_\ZZ(H_2(S,\ZZ),\ZZ) \ra H^3_B(S,\ZZ) \ra Hom(H_3(S,\ZZ),\ZZ) \ra 0.$$ 
By (\ref{resultschubert}) $H^2_{an}(S, \cO^*)_{tor}=0$.

Consider the exact sequence $1 \ra \mu_n \ra \GG_m \ra \GG_m \ra 1$ in etale and analytic topology. It induces multiplication by $n$ on cohomology $H^*(?,\GG_m) \ra H^*(?,\GG_m)$. 
Any torsion class in $H^2(S, \GG_m)$  is represented by a class in $H^2(S,\mu_n)$ for some $n$ in both topologies. Therefore $\varinjlim H^2(?,\mu_n) = H^2(?,\GG_m)_{tor}$ in both topologies.
For the constant sheaf defined by $n$-th roots of unity $\mu_n$ by \cite[Expos\'{e} 16 Thm 4.1]{SGA4}  which holds for not necessarily smooth schemes, we have $$H^2_{\acute{e}t}(S,\mu_n)=H^2_{an}(S,\mu_n).$$

  So  
 $H^2_{\acute{e}t}(S, \GG_m)_{tor} =0$.
\end{proof}

\section{Ind-Grassmannian $\cQ_G$}
We denote the ind-Grassmannian  by
\begin{equation} \label{uniformization}
 \cQ_G= \prod_{x \in \cR} Gr_{\cG,x} = \prod_{x \in \cR} \cF l_{\sigma_x}.
 \end{equation}
It represents the functor that to a scheme $S$ associates families of $\cG$-torsors on $X \times S$ together with a section on $X^\circ \times S$, where $X^\circ = X \setminus \cR$. 

For generalities on ind-schemes (such as big-\'Etale sites) see \S\S \ref{indschemes}.
\begin{prop} \label{cohvan} The cohomological Brauer group $H^2_{\acute{E}t}(\cQ_G, \cO_{\cQ_G}^*)_{tor}=0.$
\end{prop}
\begin{proof} Let us first  handle the case of one parabolic point. So $\cQ_G$ is simply $\cF l_{\sigma}$ for a certain facet $\sigma$ in an alcove $\mathbf{a}$. 


Since in our case, the Iwahori-Weyl group $\tilde{W}$ equals the affine Weyl group $W_a$, so $\tilde{W}$ is a Coxeter group. Hence it has the Bruhat partial order. This partial ordering induces a partial ordering on $\tilde{W}/W_{\sigma}$ as well as follows. For $\ol{u} = u \, mod \, W_\sigma $ and $\ol{v}=v \, mod \, W_\sigma$, we have $\ol{u} \leq \ol{v}$ if there exists a $w \in W_\sigma$ such that $u \leq vw$. 
Recall (\cite{pradv}) that the Schubert variety $S^\sigma_w$  has the structure of a finite dimensional projective, but not necessarily smooth, variety over the complex numbers. Also by \cite[Prop 9.9]{pradv},
  we have
\begin{equation} \label{dirlim}
\cF l_{\sigma}= \varinjlim_{w \in \tilde{W}/W^{\sigma}} S^{\sigma}_w.
\end{equation}

Let us first prove the proposition for the special case $\sigma=\mathbf{a}$. So $W^\mathbf{a}$ is trivial.

The group $\tilde{W}$ together with the Bruhat-order is a partially ordered set. We consider it as a category. So the identity element $e \in \tilde{W}$ is the initial object. Let $Func(\tilde{W},Ab)$ denote the category of contravariant functors from $\tilde{W}$ to the category $Ab$ of abelian groups. Consider the functor $\Gamma(\tilde{W}): Ab(\cF l_{\mathbf{a}}) \ra Func(\tilde{W},Ab)$ defined by $\cF \ms (w \ms \Gamma(S_w,\cF|_{S_w}))$. We also have the functor $\varprojlim: Func(\tilde{W},Ab) \ra Ab$. Further we have 
an equality of functors $\Gamma(\cF l_{\mathbf{a}}, \bullet)= \varprojlim \circ \Gamma(\tilde{W})$. 

We claim that $\Gamma_{\tilde{W}}$ sends injectives to injectives. To prove this, it suffices to show that this functor has an exact left-adjoint. Let us take an inverse system of abelian groups $\{G_w, \phi_{w,w'} , w \leq w'\}$. Each $G_w$ defines  the sheaf of constant sections $\ul{G_w}$ associated to $G_w$ over $S_w$. By the definition of Big \'etale site of an ind-scheme (cf \S \ref{sitesofindscheme}), the inverse system of $\ul{G_w}$ defines  an abelian sheaf on $\cF l_{\mathbf{a}}$. This construction is functorial, exact and left-adjoint to $\Gamma(\tilde{W})$ i.e $$Hom_{\cF l_{\mathbf{a}}}(\{\ul{G_w}\},\cF)=Hom_{Func(\tilde{W},Ab)}(\{G_w\},\{\Gamma(S_w,\cF|_{S_w}) \}).$$

Thus we have the Grothendieck spectral sequence \begin{equation} 
\label{schubertoflagvariety} R^p \varprojlim_w \{ H^q_{\acute{E}t}(S_w, \ul{\GG_m})_w \} \implies H^*_{\acute{E}t}(\cF l_{\mathbf{a}},\ul{\GG_m}).
\end{equation}
Since for an abelian sheaf the cohomology on the big and small \'etale sites agree (cf \cite[Prop III.3.1]{milne}), so we will use them interchangeably for Schubert varieties $S_w$.

Now for all $w \in \tilde{W}$ we have $H^0_{\acute{e}t}(S_w, \ul{\GG_m})=\CC^*$. Thus for $w \leq w'$, the restriction maps are identity.  So the inverse system of abelian groups $\{H^0_{\acute{e}t}(S_w, \ul{\GG_m}) \}$ arises as $c: Ab \ra Func(\tilde{W},Ab)$ where $c$ is the functor that to a group $A$ associates the consant functor $\tilde{W} \ra Ab$. Since by definition, colimit is left-adjoint to the constant functor, so we have an adjunction $Hom_{Ab}(colim(\{B_w\}_{w \in \tilde{W}}),A)=Hom(\{B_w\},c(A))$. Further since $e$ is the initial object, so $colim(\{B_w\})=B_e$. Thus taking colimits in our case is an exact functor. Thus $c$ has an exact left-adjoint. So it takes injectives to injectives. Now $\CC^*$ is an injective abelian group. Thus $c(\CC^*)$ is an injective object in $Func(\tilde{W},Ab)$. Thus $R^2 \varprojlim_w \{ H^0_{\acute{e}t}(S_w,\ul{\GG_m})_w \}$ vanishes.

We have $H^1_{\acute{e}t}(S_w,\ul{\GG_m})=Pic(S_w)$. By GAGA, algebraic line bundles correspond functorially and bijectively to homomorphic line bundles. Now by the exponential exact sequence and the vanishing of $H^i_{cl}(S_w,\cO_{an})$ for $i=\{1,2\}$, we see that $H^1_{cl}(S_w,\cO^*_{an})=H^2_B(S_w,\ZZ)$. By the universal coefficient, theorem it follows that $H^2_B(S_w,\ZZ)=Hom(H_2(S_w,\ZZ),\ZZ)$ since $H_1(S_w,\ZZ)=0$. Recall by Prop \ref{compSw}, $H_2(S_w,\ZZ)$ is a free-abelian group on the $2$-cells. Let us explicitly describe them. Let $\mathbf{a}$ be the alcove defining the Iwahori group scheme. Let $s_i$ be the simple reflections about the walls of $\mathbf{a}$. Let $w=s_{i_1}s_{i_2}\cdots s_{i_r}$ be a reduced expression of $w$ in terms of $s_i$. The $2$-cells of $S_w$ correspond bijectively to the $\PP^1=L^+\cG_{\sigma_i}/L^+\cG_{\mathbf{a}}$ corresponding to the reflections occuring in the reduced expression of $w$ (cf also Prop \ref{PicF}). Thus for $w \leq w'$, the map $H_2(S_w,\ZZ) \ra H_2(S_{w'},\ZZ)$ are injective and its cokernel is free over $\ZZ$. Thus the restriction of line bundle map are surjective. Thus the inverse system $\{Pic(S_w) \}$ satisfies the Mittag-Leffler condition. So $R^1 \varprojlim_w \{ H^1_{\acute{e}t}(S_w,\ul{\GG_m})_w \}=0$.

Thus the natural morphism $H^2_{\acute{E}t}(\cF l_{\mathbf{a}},\ul{\GG_m}) \ra \varprojlim_w \{ H^2_{\acute{e}t}(S_w,\ul{\GG_m})_w \}$ is injective since it identifies with the edge morphism $E^2 \ra E^{0,2}_2$ of a spectral sequence. Any torsion class in $H^2_{\acute{E}t}(Q_G,\ul{\GG_m})$ must map to a torsion classs in each of $H^2_{\acute{e}t}(S_w,\ul{\GG_m})$ for $w \in \tilde{W}$. There it must be trivial by Prop \ref{compSw}. Since it is trivial in each group, so it must be zero in the inverse limit. So it must already be trivial in $H^2_{\acute{E}t}(\cF l_{\mathbf{a}},\ul{\GG_m})$.

This finishes the proof when $\cQ_G=\cF l_{\mathbf{a}}$. In the case of $\mathbf{a}$ is replaced by a facet $\sigma$, the preceeding arguments carry word-for-word. In the case of more than one point, we have $\cQ_G = \prod_{x \in \cR} \cF l_{\sigma_x}$. In this case, instead of $\tilde{W}$ we consider $Maps(\cR,\tilde{W})$ as a category with Bruhat-order as follows: say $\ol{w} \leq \ol{w'}$ if for each $x \in \cR$, we have $w_x \leq w'_x$. We again have $H^0_{\acute{e}t}(S_{\ol{w}},\ul{\GG_m})=\CC^*$ and by the see-saw theorem, we have $H^1_{\acute{e}t}(S_{\ol{w}},\ul{\GG_m})=Pic(S_{\ol{w}})=\prod_{x \in \cR} Pic(S_{w_x})$. We can conclude again that the natural morphism $H^2_{\acute{E}t}(Q_G,\ul{\GG_m}) \ra \varprojlim_w \{ H^2_{\acute{e}t}(S_{\ol{w}},\ul{\GG_m})_w \}$ is injective. Now the claim follows as before by Prop \ref{compSw}.
\end{proof}

\section{$L_{X^\circ}(G)$} \label{lxg} 

Set $X^\circ=X \setminus \cR$. Consider the presheaf of sets on the category of $\CC$-algebras which to a $\CC$-algebra $R$ associates
$$G(Spec(R) \times_k X^\circ)=Hom(Spec(R) \times_k X^\circ,G).$$ 
Let $L_{X^\circ}(G)$ denote associated sheaf of sets. It is represented by an ind-scheme (cf \cite[Lemma 20]{heinloth},  \cite[proof of Lemma 2.1]{kumar}).

The main purpose of  this section is to prove Corollary \ref{H2BLXG} which states that $H^2_{\acute{E}t}(BL_{X^\circ}(G),\GG_m)$ is a free abelian group. We quote 
\begin{thm} \cite[Prop 2.3,Prop 2.4]{te} For $k=0,1,\cdots, \infty$, let $C^{k}(X^\circ,G)$ the space of $k$-differentiable maps with compact-open topology. The natural inclusions $$L_{X^\circ} (G) \subset Hol(X^\circ,G) \subset C^{\infty} ( X^\circ, G) \subset \cdots C^{k} ( X^\circ, G) \subset \cdots C^{0}( X^\circ, G)$$ are homotopy equivalences. 
\end{thm}

Let $G_{an}$ denote the analytic space underlying $G/\CC$.
The curve $X^\circ$ is a complex {\it affine curve} ($\Sigma$ in notation of \cite{te}) {\it smoothly deformable to a bouquet of 
$N:=2g_X + |\cR|-1$ loops}. It follows \cite[cf Page 12, \S II]{te} that the homotopy type of the analytic space underlying $L_{X^\circ}(G)$ equals that of $G_{an} \times \Omega G^{\times N}_{an}$. Now the following corollary follows immediately.

\begin{Cor} \label{consimp} The ind-group $L_{X^\circ}(G)_{an}$ is connected and simply-connected. 
\end{Cor}

\subsection{Bar construction}
We  now recall some generalities that we will need in this section and the next. Let $H$ be a topological group with identity $e$. Let $S$ be a topological space with left $H$-action.  The Bar construction 
\begin{equation} \label{barcons} EH(S)_\bullet
\end{equation} gives a simplicial topological space  as follows. For $n \geq 1$, its zero simplicies are $s \in S$. Its $n$-simpilices are written suggestively as $h_0 | h_1 |\cdots | h_{n-1}|s$ where $h_i \in H$ and $s \in S$. The $i$-th degeneracy removes the $i$-th bar to make $h_{i}$ act on the successive element.  The $j$-th face operator inserts "$e|$" at the $j$-th place. Let $*$ be a point with trivial $H$-action. The simplicial group $EH_\bullet:=EH(*)_\bullet $ is contractible. We have a natural $H$-action on $EH_\bullet$ on the leftmost factor. The Borel construction $BH_\bullet$ is the quotient $ EH_\bullet/H$ with the induced simplicial structure and the degrees are shifted by $-1$. So its $-1$ simplices are $*/H$. Its $n$-simplices are $H^n$ where $H^0=e$. Explicitly, its face and degeneracy maps are given by 
\begin{eqnarray} \label{facedeg}
s_i(h_1,\cdots,h_n) & =& (h_1,\cdots,h_i,1,h_{i+1},\cdots,h_n) \\
d_0(h_1,\cdots,h_n) & =& (h_2,\cdots,h_n) \\
d_i(h_1,\cdots,h_n) & =& (h_1,\cdots,h_i h_{i+1},\cdots,h_n) \\
d_n(h_1,\cdots,h_n)& =& (h_1,\cdots,h_{n-1}).
\end{eqnarray}


Let $BL_{X^\circ}(G)_\bullet$ be the Borel construction of $L_{X^\circ}(G)$. We view it as a simplicial ind-scheme. In particular, $e \in L_{X^\circ}(G)$ is $\Spec(\CC)$.

\begin{prop} \label{vanBL} 
Let $G$ be semi-simple and simply-connected.  Then 
\begin{enumerate}
\item Let $\ZZ$ be the trivial $L_{X^\circ}(G)$ module. Then $H^1_{\acute{E}t}(BL_{X^\circ} (G)_\bullet,\ZZ)=0$.
\item We have $H^1_{\acute{E}t} (B L_{X^\circ} (G),\GG_m)=e$.
\item $H^2_{\acute{E}t}(BL_{X^\circ} (G)_\bullet, \GG_m)=H^1_{\acute{E}t}(L_{X^\circ} (G),\GG_m)$.
\end{enumerate}
\end{prop}
\begin{proof} For any $q \geq 0$, applying the functor $H^q_{\acute{E}t}(?,\ZZ)$ to the simplicial space $BL_{X^\circ} (G)_\bullet$ we get the cosimplicial group $$H^q_{\acute{E}t}(BL_{X^\circ} (G)_\bullet,\ZZ)_\bullet.$$ We have the Atiyah-Hirzebruch spectral sequence
\begin{equation}
E^{p,q}_2:=\pi_p(H^q_{\acute{E}t}(BL_{X^\circ} (G)_\bullet,\mathbb{Z})_\bullet) \implies H^*_{\acute{E}t}(BL_{X^\circ} (G)_\bullet,\mathbb{Z})
\end{equation}
Now $H^1_{\acute{E}t}(e,\mathbb{Z})=H^1_{\acute{e}t}(e,\mathbb{Z})$ vanishes. So $\pi_0( \cdots \xymatrix{H^1_{\acute{E}t}(L_{X^\circ} (G),\mathbb{Z}) & H^1_{\acute{E}t}(e,\mathbb{Z}) \ar@<1ex>[l]\ar[l]) }$ vanishes. Now consider
\begin{equation} \pi_1(\cdots  \xymatrix{
  H^0_{\acute{E}t}(L_{X^\circ} (G)^2,\mathbb{Z}) & H^0_{\acute{E}t}(L_{X^\circ} (G),\mathbb{Z}) \ar@<1ex>[l]\ar[l]\ar@<-1ex>[l] & H^0_{\acute{E}t}(e,\mathbb{Z})
  \ar@<1ex>[l]\ar[l])
 }
 \end{equation}
 Since $L_{X^\circ} (G)$ is connected, this simplifies to 
 \begin{equation}
 \pi_1(\xymatrix{ \mathbb{Z} && \mathbb{Z} \ar[ll]_{d_0-d_1+d_2} && \mathbb{Z} \ar[ll]_{d_0-d_1}})
 \end{equation}
 By the face-degeneracy formulas (cf \ref{facedeg}) each differential above is the identity on $\ZZ$. So the above group vanishes. This proves the first assertion.
 For the following assertions, consider the spectral sequence
\begin{equation}
E^{p,q}_2:=\pi_p(H^q_{\acute{E}t}(BL_{X^\circ} (G)_\bullet,\GG_m)_\bullet) \implies H^*_{\acute{E}t}(BL_{X^\circ} (G)_\bullet,\GG_m)
\end{equation}
Since $H^1_{\acute{E}t}(e,\GG_m)$ vanishes, so $\pi_0(H^1_{\acute{E}t}(BL_{X^\circ} (G)_\bullet,\GG_m))$ also vanishes. So $E^{0,1}_2=0$. Let us abbreviate $L_{X^\circ} G$ as $L$. Consider
\begin{equation*}  \xymatrix{
 H^0_{\acute{E}t}(L^3,\mathbb{G}_m) & H^0_{\acute{E}t}(L^2,\mathbb{G}_m) \ar@<2ex>[l]\ar@<1ex>[l]\ar[l]\ar@<-1ex>[l] & H^0_{\acute{E}t}(L,\mathbb{G}_m) \ar@<1ex>[l]\ar[l]\ar@<-1ex>[l] & H^0_{\acute{E}t}(e,\mathbb{G}_m)
  \ar@<1ex>[l]\ar[l]
 }
 \end{equation*}


 Since by Corollary \ref{consimp}, $L_{X^\circ} (G)_{an}$ is connected and simply-connected, so by \cite[Corollary 2.4]{kumar}, there is no non-constant regular map $\lambda:L_{X^\circ} (G) \ra \CC^*$. On the other hand, the subgroup of constant maps $L_{X^\circ} (G)^n \ra \GG_m$ equals $\CC^*$ since our base field is $\CC$. So all the groups above identify with $\CC^*$. By face-degeneracy relations (cf \ref{facedeg}) the differentials identify with the identity on $\CC^*$. So $\pi_1$ of the above cosimplicial group vanishes. Thus $E^{1,0}_2=0$. This shows the second assertion.
 
 We now consider the third assertion. Arguing as before, we see that $\pi_2$ of the above cosimplicial group vanishes. So $E^{2,0}_2=0$ and $E^{3,0}_2=0$. Since $H^2_{\acute{E}t}(e,\GG_m)$ vanishes, so $\pi_0(H^2_{\acute{E}t}(BL_{X^\circ} (G)_\bullet,\GG_m))$ also vanishes. So $E^{0,2}_2=0$. So $E^{2}=E^{1,1}_\infty$. 


Now consider
\begin{equation} \pi_1(\cdots  \xymatrix{
  H^1_{\acute{E}t}(L_{X^\circ}(G)^2,\mathbb{G}_m) & H^1_{\acute{E}t}(L_{X^\circ}(G),\mathbb{G}_m) \ar@<1ex>[l]\ar[l]\ar@<-1ex>[l] & H^1_{\acute{E}t}(e,\mathbb{G}_m)
  \ar@<1ex>[l]\ar[l])
 }
 \end{equation}
We have $H^1_{\acute{E}t}(e,\mathbb{G}_m)=0$. An element $l \in H^1_{\acute{E}t}(L_{X^\circ}(G),\mathbb{G}_m)$ is mapped to $(l,0)$, $(l,l)$ and $(0,l)$ in $H^1_{\acute{E}t}(L_{X^\circ}(G)^2,\mathbb{G}_m)$ under the various face maps. So $\pi_1(\cdots)=E^{1,1}_2=H^1_{\acute{E}t}(L_{X^\circ} G,\GG_m)$. Now by $E^{-1,2}_2 \ra E^{1,1}_2 \ra E^{3,0}_2$, we get $E^{1,1}_3=E^{1,1}_2$. By $E^{-2,3}_3 \ra E^{1,1}_3 \ra E^{4,-1}_3$, we have $E^{1,1}_3=E^{1,1}_4$. Hence, $E^{1,1}_2=E^{1,1}_\infty$. This shows the claim.

\end{proof}

\subsection{The analytic site of $L_{X^\circ}(G)$}
\begin{prop} \label{colimconn} We have $L_{X^\circ}(G)= \varinjlim Y_n$ for $n \in \mathbb{N}$ where  $Y_n$ are affine schemes on $\CC$ and $Y_n \ra Y_{n+1}$ is a closed immersion.
\end{prop}
\begin{proof}
Let $L(d)$ denote the geometric line bundle on $X$ corresponding to $\cO_X(d \cR)$. Let us fix once for all a closed immersion $G \hra M(k \times k)$ in matrices for some $k \geq 1$. Let $Z_d$ denote the subfunctor of $Mor(X \setminus \cR,M(k \times k))$ that to a ring $R$ associates matrices each of whose entries has a pole of order at most $d$ along $\cR \times \Spec(R)$ and is regular on $X \setminus \cR \times \Spec(R)$.  Thus $Z_d$ identifies with the functor that to a ring $R$ associates $k^2$-many sections $Hom_X(X \times Spec(R),L(d))$ of $L(d) \ra X$. Since the morphism $L(d) \ra X$ is quasi-projective, so $Z_d$ is representable. Remark that $Z_0$ is representable by $M(k \times k)$. More generally, if $l(d)=dim_\CC H^0(X,\cO_X(d \cR))$ then 
$Z_d$ is representable by $ \mathbb{A}^{k^2 \times l(d)}$. With $a_i \in \mathbb{A}^{k^2}$, the closed inclusion $Z_d \hra Z_{d+1}$ corresponds to the inclusion of affine spaces $$(a_1,\cdots,a_{l(d)}) \ms (a_1,\cdots,a_{l(d)},0_{l(d)+1},\cdots,0_{l(d+1)}).$$ So $\varinjlim Z_d=Mor(X \setminus \cR,M(k \times k))$. Set 
\begin{equation} \label{Yexplicit} Y_d:= Mor(X \setminus \cR,G) \times_{Mor(X \setminus \cR,M(k \times k))} Z_d.
\end{equation}  
Notice further that each $Y_d$ is a closed subscheme of $Z_d$. So it is affine. Further, we have $L_{X^\circ}(G)= \varinjlim Y_d$. 

\end{proof}

We define $L_{X^\circ}(G)_{an}$ as a colimit of analytic spaces by setting
\begin{equation} \label{analyticcolim} L_{X^\circ}(G)_{an}= \varinjlim Y_{n,an}.
\end{equation}

By the analytic site of $L_{X^\circ}(G)_{an}$ we shall mean the following
\begin{enumerate}
\item objects are morphisms $u: U \ra L_{X^\circ}(G)_{an}$ factoring through an analytic morphism $u:U \ra Y_{n,an}$ for some $n$. 
\item a morphism from $u: U \ra L_{X^\circ}(G)_{an}$ to $u': U' \ra L_{X^\circ}(G)_{an}$ is an analytic morphism $f: U \ra U'$ of analytic spaces such that $u' \circ f=u$.
\item a covering of $u: U \ra L_{X^\circ}(G)_{an}$ is just an analytically \'etale covering $U' \ra U$ of $U$.
\end{enumerate}

\begin{prop} \label{analyticcoh} We have 
\begin{enumerate}
\item $H^n_{an}(L_{X^\circ}(G),\cO)=0$ for all $n \geq 1$.
\item the group $H^1_{an}(L_{X^\circ}(G),\GG_m)$ is a finitely generated free $\ZZ$-module.
\item $H^1_{an}(L_{X^{\circ}}(G), \mu_n)=0$ for all $n$.

\end{enumerate}
\end{prop}
\begin{proof} (1)
 By (\ref{analyticcolim}), $L_{X^\circ}(G)_{an}= \varinjlim Y_{n,an}$.  For any sheaf $\cF$ on the analytic site of $L_{X^\circ}(G)$ we have $$\Gamma_{an}(L_{X^\circ}(G),\cF)=\varprojlim_n \{ \Gamma_{an}(Y_n,\cF)\}_n.$$ 
By repeating the arguments of Proposition \ref{cohvan}, we can establish the Grothendieck spectral sequence in the analytic topology:
\begin{equation}
 R^p \varprojlim_n \{ H^q_{an}(Y_n,\cF |_{Y_n}) \}_n \implies H^*_{an}(L_{X^\circ} (G),\cF).
\end{equation}

Let us take $\cF$ as the coherent analytic sheaf $\cO_{an}$. Then since $Y_n$ are affine schemes over $\CC$, so $Y_{n,an}$ are closed analytic subspaces of $\CC^n$ for some $n$. So they are Stein spaces.
Thus for $q \geq 1$,  $H^q(Y_n,\cO)$ vanishes. So their $R^p \varprojlim$ is zero. If $q=0$, then the inverse system $\cdots H^0(Y_n,\cO) \leftarrow H^0(Y_{n+1},\cO) \leftarrow \cdots$ is surjective on each arrow. So it satisfies the Mittag-Leffler condition. So its higher $\varprojlim$ is zero.

(2)  Consider the exact sequence $0 \ra \ZZ \ra \cO_{an} \stackrel{exp}{\lra} \cO^*_{an} \ra 0$ of sheaves on the analytic site of $L_{X^\circ}(G)$. So by (1), $H^1_{an}(L_{X^\circ}(G),\GG_m) \ra H^2(L_{X^\circ}(G),\ZZ)$ is an isomorphism. Now since $L_{X^\circ}(G)$ has the homotopy type of $G \times \Omega G^{N}$, so it follows that it is $(2-1)$-connected. So $H_1(L_{X^\circ}(G),\ZZ)=0$. Thus by the universal coefficent theorem $H^2(L_{X^\circ}(G),\ZZ) \ra \Hom(H_2(L_{X^\circ}(G),\ZZ),\ZZ)$ is an isomorphism. This shows that it is finitely generated and free over $\ZZ$.
 
 (3) Consider the short exact sequence of sheaves $1 \ra \mu_n \ra \GG_m \stackrel{z \ms z^n}{\lra} \GG_m \ra 1$ on the analytic site of $L_{X^\circ}(G)_{an}$. Taking the long exact sequence, we see that $H^1_{an}(L_{X^\circ}(G),\mu_n)$ injects into $H^1_{an}(L_{X^\circ}(G),\GG_m)$ as follows. 
  Since by Corollary \ref{consimp}, $L_{X^\circ} (G)_{an}$ is connected and simply-connected, so by \cite[Corollary 2.4]{kumar}, there is no non-constant regular map $\lambda:L_{X^\circ} (G) \ra \CC^*$. So we have $H^0_{\acute{E}t}(L_{X^\circ}(G), \GG_m)=\CC^*$. Thus the first three terms reduce to $1 \ra \mu_n \ra \CC^* \stackrel{z \ms z^n}{\lra} \CC^* \ra \cdots$.
 This shows the injection. Now $H^1_{an}(L_{X^\circ}(G),\mu_n)$ is torsion. So it vanishes by the second assertion.

\end{proof}

\subsection{First cohomology on sites and torsors}
Let us begin by recalling the notion of a torsor on a Grothendieck site $\cC$. Let $\cG$ be a sheaf of groups on $\cC$. A $\cG$-torsor $\cF$ on $\cC$ is a sheaf of sets on $\cC$ endowed with an action $\cG \times \cF \ra \cF$ such that 
\begin{enumerate}
\item whenever $\cF(U)$ is non-empty, the action $\cG(U) \times \cF(U) \ra \cF(U)$ is simply-transitive.
\item for every $U \in Ob(\cC)$, there exists a covering $\{U_i \ra U\}_{i \in I}$ of $U$ such that $\cF(U_i)$ is non-empty for all $i \in I$.
 \end{enumerate}
A trivial $\cG$-torsor is the sheaf $\cG$ endowed with the natural left-action. We have a contravariant functor $\cC \ra Ab$ given by $U \ms \Gamma(U,\cF)$. By definition, 
\begin{equation} \label{globalsection}
\Gamma(\cC,\cF)=\varprojlim_{ \cC} \Gamma(U,\cF).
\end{equation}
A $\cG$-torsor is trivial if and only if $\Gamma(\cC,\cF) \neq \emptyset$.
We quote
 
\begin{lem} \cite[Lemma 5.3]{sites} \label{torsor1coh} Let $\cC$ be a site. Let $\cH$ be an abelian sheaf on $\cC$. There is a canonical bijection between the set of isomorphism classes of $\cH$-torsors and $H^1(\cC,\cH)$. 
\end{lem} 
 
We need to show a techinical result which says that the canonical bijection above behaves well under change of sites.
Let $\epsilon: \cC_1 \ra \cC_2$ be a morphism of sites. Let $\epsilon^s$ denote the sheaf pull-back on sites (cf \cite[Chapter 1]{tamme}). So $(\epsilon_s,\epsilon^s)$ form an adjoint-pair and $\epsilon_s$ is exact. Let us recall the definition of $\epsilon_p$. For $U_2 \in \cC_2$, let $(U_2 \downarrow \epsilon)$ denote the category whose
\begin{enumerate}
\item  objects are $\{(u,U_1)|u: U_2 \ra \epsilon(U_1) \in Mor(\cC_2) \}$.
\item morphisms from $(u,U_1) \ra (u',U_1')$ are $\{f:U_1 \ra U_1'| u'=\epsilon(f) u\}$.
\end{enumerate} 
We have a natural forgetful functor $(U_2 \downarrow \epsilon) \ra \cC_1$. Given a presheaf $\cG_1$ on $\cC_1$, we view it as a presheaf on $(U_2 \downarrow \epsilon)$. Now, we define
\begin{equation} \label{presheafpushforwardsites}
\epsilon_p(\cG_1)(U_2)= \varinjlim_{ (U_2 \downarrow \epsilon) } \cG_1(U_1).
\end{equation}
Let $\cH$ be an abelian sheaf on $\cC_2$. Consider the edge morphism $e: H^1(\cC_1,\epsilon^s \cH) \ra H^1(\cC_2,\cH)$ of the Leray spectral sequence. 

\begin{Cor} \label{torsorextn} Let $\cF_1$ be  a $\epsilon^s \cH$-torsor. Then $e$ maps $\cF_1$  to 
the $\cH$-torsor obtained by extending structure group by $\epsilon_s \epsilon^s \cH \ra \cH$ on $\epsilon_s \cF_1$.
\end{Cor}
\begin{proof} 
From the proof of \cite[Lemma 5.3]{sites}, let us recall the correspondence from $H^1(\cC,\cH)$ to torsors. Let $\xi \in H^1(\cC,\cH)$ be given. Choose an embedding of $\cH$ into an injective sheaf $\cI$ and let $\cQ=\cI/\cH$. Since $H^1(\cC,\cI)=0$, so $\xi$ lifts to some $q \in H^0(\cC,\cQ)$. Let $\cF \subset \cI$ be the subsheaf {\it of sets} defined by the following condition: its local sections over any open $U$ of $\cC$ map to $q|_U$. Then $\cF$ is a $\cH$-torsor and the canonical bijection maps $\xi \ms \cF$.

With notations as above, consider the exact sequence of sheaves $0 \ra \cH \ra \cI \ra \cQ \ra 0$ on $\cC_2$. Now $\epsilon^s$ is left-exact and $\epsilon^s \cI$ is an injective sheaf. Let $\cQ_1=\epsilon^s \cI/\epsilon^s \cH$ denote the sheaf quotient on $\cC_1$. So we have a natural map $\cQ_1 \ra \epsilon^s \cQ$ or equivalently $\epsilon_s \cQ_1 \ra \cQ$ by adjunction.  Now consider the diagram
\begin{equation}
\xymatrix{
0 \ar[r] & \epsilon_s \epsilon^s \cH \ar[r] \ar[d] &  \epsilon_s \epsilon^s \cI \ar[r] \ar[d] & \epsilon_s \cQ_1 \ar[r] \ar[d] & 0 \\
0 \ar[r] & \cH \ar[r] & \cI \ar[r] & \cQ \ar[r] & 0
}
\end{equation}
The left-square commutes because the left two vertical arrows are of adjunction. So the right square also commutes because of exactness of rows. The following diagram is also commutative
\begin{equation}
\xymatrix{
H^0(\cC_1,\cQ_1) \ar[r] \ar[d] & H^0(\cC_2,\epsilon_s \cQ_1) \ar[r] \ar[d] & H^1(\cC_2, \epsilon_s \epsilon^s \cH) \ar[d] \\
H^0(\cC_1,\epsilon^s \cQ) \ar[r] & H^0(\cC_2, \cQ) \ar[r] & H^1(\cC_2,\cH)
}
\end{equation}
The  rightmost vertical arrow extends the structure group, the leftmost comes from $\cQ_1 \ra \epsilon^s \cQ$ and the middle comes from adjunction of $\epsilon^s$. 

Let $q_1 \in H^0(\cC_1,\cQ_1)$ map to $\xi_1 \in H^1(\cC_1,\epsilon^s \cH)$. The edge morphism $e$ is defined by composing the left vertical arrow with the horizontal arrows to get $H^0(\cC_1,\cQ_1) \ra H^1(\cC_2,\cH)$ and then observing that $H^0(\cC_1,\epsilon^s \cQ) \ra H^1(\cC_2,\cH)$ factors through $H^1(\cC_1,\epsilon^s \cH)$. This describes the edge morphism. We may instead chase the diagram under the top horizontal row followed by the last vertical arrow.

 Then $q_1$ determines a subsheaf $\cF_1 \subset \epsilon^s \cI$ by the correspondence described above.  Now by exactness of $\epsilon_s$, we have the exact sequence of sheaves $0 \ra \epsilon_s \epsilon^s \cH \ra \epsilon_s \epsilon^s \cI \ra \epsilon_s \cQ_1 \ra 0$ on $\cC_2$. Consider the subsheaf $\epsilon_s \cF_1 \subset \epsilon_s \epsilon^s \cI$. Under the correspondence between cohomology classes and torsors recalled above, we claim that $\epsilon_s \cF_1$ corresponds to the image $q'_1$ of $q_1$ under $H^0(\cC_1,\cQ_1) \ra H^0(\cC_2,\epsilon_s \cQ_1)$. By the definition of presheaf pushforward (\ref{presheafpushforwardsites}), local sections of $\epsilon_p \cF_1$ via $\epsilon_p \cF_1 \ra \epsilon_p \epsilon^s \cI \ra \epsilon_p \cQ_1 \ra \epsilon_s \cQ_1$ map to restrictions of  $q_1'$ in $\epsilon_s \cQ_1$. Therefore local sections of $\epsilon_s \cF_1$ via $\epsilon_s \cF_1 \subset \epsilon_s \epsilon^s \cI \ra \epsilon_s \cQ_1$ map to $q_1'$. So $\epsilon_s \cF_1$ is contained in the $\epsilon_s \epsilon^s \cH$-torsor corresponding to $q'_1$. But it is itself such a torsor. So the claim follows.

Since $\epsilon_s \cF_1$ is a $\epsilon_s \epsilon^s \cH$-torsor corresponding to the image of $q_1'$ under $H^0(\cC_2,\epsilon_s \cQ_1) \ra H^1(\cC_2,\epsilon_s \epsilon^s \cH)$, so the claim follows.

\end{proof}

\subsection{Change of sites from Big-\'etale to Analytic of sheaves}
The Leray spectral sequence (cf \cite[Theorem I(3.7.6)]{tamme}) is defined for a {\it continous} morphism of sites. For us, $L_{X^{\circ}}G$ is only a sheaf on the big-\'etale site of $\Spec(\CC)$ and similarly $(L_{X^\circ}G)_{an}$ is only a sheaf on the analytic site of $\CC$. We use the setup in \cite[Giraud]{giraudbook} to define a continous morphism $F^*$ from the {\it site of sheaves} on $\acute{E}t(\Spec(\CC))$ to the {\it site of sheaves} on the analytic site of $\CC$. Then we apply the Leray spectral sequence to $F^*, L_{X^\circ}G$ and $(L_{X^\circ}G)_{an}$ to deduce the results of this subsection.

Consider the functor \begin{equation} F^{-1}: \acute{E}t(\Spec(\CC)) \ra an(\CC_{an})
\end{equation}  on underlying categories given by mapping $u: U \ra Y_n$ to the  analytic morphism $u_{an}: U_{an} \ra Y_{n,an}$. Recall \cite[Definition 0.3.1]{giraudbook}, if $X$ and $Y$ are two sites, then a functor $f^{-1}:Y \ra X$ is said to be continous if for every sheaf $G$ on $X$, the presheaf
\begin{equation} \label{pfsites} f_*(G)(y)=G(f^{-1}(y)), y \in Ob(Y)
\end{equation}
is a sheaf. Then $F^{-1}$ is a continuous morphism of sites by \cite[I(1.2.2)]{tamme} because it satisfies the following conditions:
\begin{enumerate}
\item if $\{U_i \ra U\}_{i \in I}$ is a covering of $U$ in the big-\'etale site, then $\{U_{i,an} \ra U_{an}\}_{i \in I}$ is a covering of $U_{an}$ in the analytic site.
\item for a covering $U' \ra U$ and for any $V \ra U$, we have $F^{-1}(U' \times_U V)=(U' \times_U V)_{an} = U'_{an} \times_{U,an} V_{an}= F^{-1}(U') \times_{F^{-1}(U)} F^{-1}(V)$.
\end{enumerate}

Let us proceed to define a continous morphism of sites $f^{-1}_L: \acute{E}t(L_{X^\circ}(G)) \ra an(L_{X^\circ}(G)_{an})$ which is a restriction of $F^{-1}$.

\begin{prop} \cite[J.Giraud Chapitre 0 3.1.4]{giraudbook} Let $E$ be a site and $\hat{E}$ be the category of presheaves on $E$. Let $\eta: E \ra \hat{E}$ be given by $\eta(S)(T)=Hom(T,S)$. Let $P \in Ob(\hat{E})$. Let $E/P= \hat{E}/P \times_{\hat{E}} E$. We equip $E/P$ with the topology that makes the functor $E/P \ra E$ continous. Consider an object ${\bf S}=(S,s: \eta(S) \ra P)$ in the comma category  $E/P$.
There is a natural isomorphism of comma categories
$$(E/P)_{\bf{S}} =E/S,$$
which induces a bijection between the set of refinements of $\bf{S}$ for the topology induced on $E/P$ and the set of refinements of $S$ for the topology of $E$.
\end{prop}

By definition of sites $\acute{E}t(L_{X^\circ}(G))$ and $an(L_{X^\circ}(G))$, we have an equivalence of the following categories \begin{eqnarray}
\acute{E}t(L_{X^\circ}(G))=\acute{E}t(\Spec(\CC))/L_{X^\circ}(G) \\
an(L_{X^\circ}(G))=an(\CC_{an})/L_{X^\circ}(G)_{an},
\end{eqnarray}
By the above proposition, it 
is also an equivalence of sites. Further $f^{-1}_L$ is the restriction of $F^{-1}$. We now want to define a continous morphism $F^*$ from the {\it site of sheaves} on $\acute{E}t(\Spec(\CC))$ to the {\it site of sheaves} on the analytic site of $\CC$.

For a site $E$, let $\tilde{E}$ denote the category of sheaves on $E$. We put the canonical topology on $\tilde{E}$ which is the strongest topology so that every object is representable. Let $\epsilon= a \eta$ where $a: \hat{E} \ra \tilde{E}$ is the sheafification functor. Further for any continuous morphism $f^{-1} : Y \ra X$ of sites, the functor $f_*:\tilde{X} \ra \tilde{Y}$ (\ref{pfsites}) admits a left-adjoint $f^*: \tilde{Y} \ra \tilde{X}$ by \cite[Prop 3.2]{giraudbook}. Further the diagram (cf \cite[0(3.3.2)]{giraudbook}) commutes
\begin{equation} \label{(3.3.2)}
\xymatrix{
X \ar[r]^{\epsilon} & \tilde{X} \\
Y \ar[r]^{\epsilon'} \ar[u]^{f^{-1}} & \tilde{Y} \ar[u]_{f^*}
}
\end{equation}
We quote
\begin{prop} \label{EEtilde} \cite[Proposition 3.6]{giraudbook} The functor $\epsilon: E \ra \tilde{E}$ defines a morphism of sites $\tilde{E} \ra E$ if we equip $\tilde{E}$ with the canonical topology. Consider the direct image functor on sheaves which by definition is induced by composition with $\epsilon$. It induces an equivalence between the category of sheaves of sets on $E$ and $\tilde{E}$.
\end{prop}
Therefore $\tilde{\epsilon}$ and $\tilde{\epsilon}'$ are equivalence of categories in the commutative diagram 
\begin{equation} \label{(3.3.2)1}
\xymatrix{
X \ar[r]^{\epsilon} & \tilde{X} \ar[r]^{\tilde{\epsilon}} & \tilde{\tilde{X}} \\
Y \ar[r]^{\epsilon'} \ar[u]^{f^{-1}} & \tilde{Y} \ar[u]_{f^*} \ar[r]^{\tilde{\epsilon}'} & \tilde{\tilde{Y}} \ar[u]_{(f^*)^*}
}
\end{equation}
Therefore $ \tilde{\epsilon}' f_* \tilde{\epsilon}^{-1}: \tilde{\tilde{X}} \ra \tilde{\tilde{Y}}$ identifies with $(f^*)_*$ because they are both right-adjoint to $(f^*)^*$. So $(f^*)_*$ maps sheaves to sheaves because $f_*$ does. So $f^*$ is continous morphism of sites. We apply this result setting 
\begin{equation} Y=\acute{E}t(\Spec(\CC)), \quad X=an(\CC_{an}), \quad f^{-1}=F^{-1}.
\end{equation} Since $F^{-1}$ is a morphism of sites, so the pair $(F^*,F_*): \tilde{X} \ra \tilde{Y}$ define a morphism of topos. In particular, $F^*$ commutes with finite projective limits. So 
\begin{equation} F^*: \tilde{Y} \ra \tilde{X}
\end{equation} is a continous morphism of sites. 

Let us view $L_{X^\circ}(G) \in Ob(\tilde{Y})$ and $L_{X^\circ}(G)_{an} \in Ob(\tilde{X})$. Since $L_{X^\circ}(G)=\varinjlim Y_n$ by Prop \ref{colimconn}, and $L_{X^\circ}(G)_{an}=\varinjlim Y_{n,an}$ by (\ref{analyticcolim}), since $F^*$ commutes with arbitrary colimits, we have 
\begin{equation} F^*(L_{X^\circ}(G))=L_{X^\circ}(G)_{an}.
\end{equation} Let $(F^*)^s (=(F^*)^*)$ denote the sheaf pull-back by $F^*$ for the sheaf defined by $\mu_n$ on $\tilde{X}$. By \cite[Theorem I(3.7.6)]{tamme} we have the Leray spectral sequence 
\begin{equation*}
E^{p,q}_2=H^p_{\acute{E}t}(\tilde{Y}, L_{X^\circ}(G), R^q (F^*)^s(\mu_n)) \implies E^{p+q}=H^{p+q}_{an}(\tilde{X}, F^*(L_{X^\circ}(G)),\mu_n)
\end{equation*}
 This gives the edge morphism \begin{equation} \label{edge} e: H^1(\tilde{Y}, L_{X^\circ}(G), (F^*)^s \mu_n) \ra H^1(\tilde{X}, L_{X^\circ}(G)_{an},\mu_n).
\end{equation}

Let us relate the above cohomologies on sites $\tilde{Y}$ and $\tilde{X}$ with those on $\acute{E}t(L_{X^\circ}(G))$ and $an(L_{X^\circ}(G))$ to simplify (\ref{edge}).

\begin{prop} Consider the sheaf defined by the group $\mu_n$ on $\tilde{X}$. The pull-back sheaf $(F^*)^s(\mu_n)$ on $\tilde{Y}$ is the sheaf defined by $\mu_n$ itself.
\end{prop}
\begin{proof} Consider the presheaf $\ul{\mu_n}$ defined by $\mu_n$ on the sites $\tilde{X}$ and $\tilde{Y}$. By the last proposition, we may assume that our sites are $X$ and $Y$ instead of $\tilde{X}$ and $\tilde{Y}$ and our functor is $F^{-1}$ instead of $F^*$. Let us observe that this presheaf is already a sheaf. Since $\mu_n$ is discrete, so any local section of $\ul{\mu_n}$ is locally constant on any open of these sites. To verify the sheaf condition, we may  restrict ourselves to Zariski open covers (or open covers in the analytic topology) and a single \'etale morphism $V \ra U$. In both cases, we may further assume that $U$ is connected. Let $V_i, i \leq n$ be the connected components of $V$ or of an open cover of $U$.  Consider $pr_1^*, pr_2^*: \ul{\mu_n}(V) \ra \ul{\mu_n}(V \times_U V)$. Let $s \in \ul{\mu_n}(V)$ and $s_i$ denote its restriction to $V_i$. 
Let us say that $V_i$ is related to $V_j$ if $V_i \times_U V_j \neq \emptyset$. Observe that if $V_i$ is related to $V_j$, then the value in the group $\mu_n$ of $s_i$ and $s_j$ are equal. This property holds more generally for a pair $V_i$ and $V_j$ which are in the same equivalence class generated by this relation. Now observe that over a conneced $U$ any two $V_i$ and $V_j$ are related.

 By definition $(F^*)^s= \# (F^*)^p i$, where $i: Sh \ra Prsh$ is the inclusion functor of sheaves into presheaves, $(F^*)^p$ is the presheaf pull-back defined by $(F^*)^p(\cF)(U)=\cF(f(U))$, and $\#$ is the sheafification functor. We have $i(\mu_n)$ is the presheaf defined by $\mu_n$. Now it suffices to check that $(F^*)^p (\mu_n)$ is the presheaf defined by $\mu_n$. This holds. So $(F^*)^s(\mu_n)=\mu_n$.

\end{proof}

The equivalence of sheaves of sets on $E$ and $\tilde{E}$ in Proposition \ref{EEtilde} restricts to an equivalence of abelian sheaves on $E$ and $\tilde{E}$. It preserves exactness. Further injective sheaves restrict to injective sheaves. Finally for any sheaf $\cF$ on $\tilde{E}$, we have the equality of global sections by definition (cf (\ref{globalsection}))
\begin{equation}
\Gamma(\tilde{Y}, L_{X^\circ}(G),\cF)=\Gamma_{\acute{E}t}(L_{X^\circ}(G),\epsilon_*(\cF)).
\end{equation}
Therefore their derived functors are isomorphic. Hence for any $n$ we have \begin{eqnarray}
H^n(\tilde{Y},L_{X^\circ}(G),\cF)=H^n_{\acute{E}t}(L_{X^\circ}(G),\epsilon_*(\cF)) \\
H^n(\tilde{X},L_{X^\circ}(G)_{an},\cF)=H^n_{an}(L_{X^\circ}(G)_{an},\epsilon_*(\cF))
\end{eqnarray}
So (\ref{edge}) becomes \begin{equation} \label{edge1} e: H^1_{\acute{E}t}( L_{X^\circ}(G), \mu_n) \ra H^1_{an}( L_{X^\circ}(G)_{an},\mu_n).
\end{equation}

\begin{prop} \label{edgeinj} The edge morphism (\ref{edge1}) (or \ref{edge}) is injective.
\end{prop}
\begin{proof} Let $c \in ker(e)$.   Let $L$ be a $\mu_n$ torsor on $\acute{E}t(L_{X^\circ}(G))$ representing $c$ by Lemma \ref{torsor1coh}. Let $L_{an}$ represent $e(c)$. We have a natural map $(F^*)_s L \ra L_{an}$ by Corollary \ref{torsorextn}. So for any open $U \in \acute{E}t(L_{X^\circ}(G))$ we have 
\begin{equation} \label{localsections}
H^0(U,L)=H^0(F^*(U),(F^*)_sL) \ra H^0(F^*(U),L_{an}).
\end{equation}
Since $e(c)=0$, so $L_{an}$ admits a global section $s^{an}$ on the analytic site. Since $L_{X^\circ}(G)_{an}$ is connected, and $s^{an}$ takes values in $\mu_n$, so it must be a  constant global section. So its restriction $s_u^{an} \in L_{an}(U)$ to any open $u: U_{an} \ra L_{X^\circ}(G)_{an}$ in the analytic site is a constant section.  Consider an arbitrary open $u: U \ra L_{X^\circ}(G)$ of the big-\'etale site. Since $s^{an}_u$ is a constant section, so in (\ref{localsections}) it may be viewed as a section $s_u \in H^0(F^*(U),(F^*)_sL)$, and thereby of $L(U)$ through $u$. So the sections $s_u$ are constant for any $u \in \acute{E}t(L_{X^\circ}G)$. So the collection $\{s_u\}$ define a  global section $s$ of $L$. So $L$ is trivial.
\end{proof}

 
\begin{Cor} \label{H2BLXG}  The group $H^2_{\acute{E}t}(B L_{X^\circ} G, \GG_m)$ is finitely-generated and free abelian.
\end{Cor}
\begin{proof} We have $H^2_{\acute{E}t}(B L_{X^\circ} G, \GG_m) = H^1_{\acute{E}t}(L_{X^\circ} G, \GG_m)$ by Proposition \ref{vanBL}. Any torsion-class is represented by an element in $H^1_{\acute{E}t}(L_{X^\circ} G, \mu_n)$ for some $n$. The second group vanishes by Propositions \ref{edgeinj} and \ref{analyticcoh}.
\end{proof}

{\it For notational convenience, we will denote $H^1_{\acute{E}t}(L_{X^\circ}G,\GG_m)$ by $\Pic(L_{X^\circ} G)$.}

\section{Brauer group of the moduli stack: $G$ is simply connected} \label{bgpms}
\subsection{Parahoric torsors} \label{pt}
Let $\cG \ra X$ be a group scheme as in \S \ref{gpsch}. A {\it quasi-parahoric} torsor $\cE$ is a $\cG$--torsor on $X$. This means that $\cE \times_X \cE \simeq \cE \times_X \cG$ and we have an action map $a: \cE \times_X \cG \ra \cE$ which satisfies the usual axioms of $G$-bundles. By weights we mean elements ${\boldsymbol\theta}\,=\, \{\theta_x| x \in \cR \} \in (Y(T) \otimes \RR)^m$ where $\theta_x$ lies in the interior of the facet $\sigma_x$ (cf \S \ref{gpsch}) and $m=|\cR|$. A {\it parahoric torsor} is a pair $(\cE\, , {\boldsymbol\theta})$ consisting of  a quasi-parahoric torsor and weights.

\subsection{Uniformization}
Let $\cM_X(\cG)$ be the moduli stack of parahoric $\cG$-torsors on $X$. By \cite{heinloth} it is an algebraic stack. Recall that $\cQ_G= \prod_{x \in \cR} \cF l_{\sigma_x}$ and $X^\circ = X \setminus \cR$. A $R$-point of $\cQ_G$ classifies $\cG$-torsors on $X \times \Spec(R)$ together with a section on $X^\circ \times \Spec(R)$. The map $\cQ \ra \cM$ forgets the section and the ind-scheme $L_{X^\circ}(G)$ acts on $\cQ_G$ by changing the section.  By the Uniformization theorem (cf \cite[Heinloth]{heinloth}) we have an isomorphism of stacks
\begin{eqnarray} \label{unif}
\cQ_G/L_{X^\circ}(G) = \cM_X(\cG) \\
\cQ \times L_{X^\circ}(G) \ra \cQ \times_\cM \cQ.
\end{eqnarray}
For a stack $\cX$, $H^2_{et}(\cX,\GG_m)_{tor}$ is called the {\it cohomological Brauer group}. We wish to compute it when $\cX=\cM_X(\cG)$.

\subsection{Cohomology of sheaves on $\cM_X(\cG)$}
We begin with a generality. Let $\bY_\bullet$ be a simplicial ind-scheme and $\cX$ be an Artin stack. Let $a: \bY_\bullet \ra \cX$ be a morphism universally of cohomological descent (cf \S \ref{cohdessec}). Let us denote by $\GG_m$ the abelian sheaf defined by $\GG_m$ on the big \'etale sites of $\bY_\bullet$ and $\cX$. We have $a^*\GG_m=\GG_m$ since $a^*$ is just restriction to $\acute{E}t(\bY)$ of $\acute{E}t(\cX)$. Then by Theorem \ref{ss} we have an equality of abutments
\begin{equation} \label{eqabutment} H^{p+q}_{\acute{E}t}(\bY_\bullet,\GG_m)=H^{p+q}_{\acute{E}t}(\cX,\GG_m).
\end{equation}

For our purposes, we now specialize to the case  $p:\cQ_G \ra \cM_X(\cG)$. 


\begin{prop} The morphism $p$ is universally of cohomological descent.
\end{prop}
\begin{proof}  Recall that $Gr_{\cG,x}$ parametrizes $\cG$-torsors together with a  section on $X \setminus \{x\}$. By \cite[Thm 4]{heinloth}, for any $S$-family $\cP \in \cM_X(\cG)(S)$, there exists an \'etale covering $S' \ra S$ such that $\cP|_{X \setminus \{x\} \times S'}$ is trivial.   
Recall by (\ref{uniformization})   that
$
 \cQ_G= \prod_{x \in \cR} Gr_{\cG,x} = \prod_{x \in \cR} \cF l_{\sigma_x}$.
So $p$ admits \'etale local sections in the sense of Definition \ref{etlocalsection}. So by Theorem \ref{et-sec-indsch-stack}, $p$ is universally of cohomological descent.
\end{proof}

Let us abbreviate $\cQ_G$ as $\bY$ and $\cM_X(\cG)$ as $Z$.  Associated to any such projection $p$, by the coskeleton construction one has a simplicial ind-scheme augmented by $(a,Z)$. It is isomorphic to the quotient of the Bar construction (see (\ref{barcons}))
\begin{equation} EL_{X^\circ}(G)(\bY)/L_{X^\circ}(G).
\end{equation} We will abbreviate it as $\bY_\bullet$. Its set of $n$-simplices is the quotient by the left-action of $L_{X^\circ}(G)$ on $L_{X^\circ}(G)^{\times n} \times \bY$ for $n \geq 0$. We have a simplicial map
\begin{equation}
p_\bullet: \bY_\bullet \ra BL_{X^\circ}(G)_\bullet,
\end{equation}
 $p_\bullet$ is defined by projecting on the $L_{X^\circ}(G)$-part.
The Grothendieck spectral sequence for $\Gamma_{\bY}=\Gamma_{BL_{X^\circ}(G)} \circ p_{\bullet,*}$ with values in a simplicial sheaf $\cF^\bullet$ on $\bY_\bullet$ has the form
\begin{equation}
E^{p,q}_2 = H^p_{\acute{E}t}(BL_{X^\circ}(G)_\bullet,R^qp_{\bullet,*} (\cF^\bullet)) \implies H^{*}_{\acute{E}t}(\bY_\bullet,\cF^\bullet),
\end{equation}
(cf \cite[page 10 (1.9) and page 27 (5.5)]{te} for similar result). 
Combining with (\ref{eqabutment}) and taking $\cF^\bullet$ as the sheaf defined by $\GG_m$, we get 
\begin{equation} \label{sstele}
E^{p,q}_2 = H^p_{\acute{E}t}(BL_{X^\circ}(G)_\bullet,R^qp_{\bullet,*}(\GG_m)) \implies H^{*}_{\acute{E}t}(\cM_X(\cG),\GG_m).
\end{equation}

\subsection{Big-\'etale site and line bundles on sites and ind-schemes}
 For the case of schemes, by \cite[Prop III.3.1]{milne} for an abelian sheaf the cohomology on the big and small \'etale sites agree. Let $A \ra \cM$ be  an atlas of an algebraic stack $\cM$. Let $A^{\times_{\cM} p}$ denote the $p$-fold fiber product of $A$ over $\cM$. We have spectral sequences \begin{equation} \label{twospec} \xymatrix{ E^{p,q}_1=H^q_{\acute{e}t}(A^{\times_{\cM} p},\GG_m ) \ar[d] & \implies & H^n_{\acute{e}t}(\cM,\GG_m) \ar[d] \\ E^{p,q}_1=H^q_{\acute{E}t}(A^{\times _{\cM} p},\GG_m) & \implies & H^n_{\acute{E}t}(\cM,\GG_m) }
\end{equation}
which fit vertically by  the natural homomorphism from the small to big \'etale groups.
So cohomologies on the big-\'etale and small sites agree for algebraic stacks too. The following proposition proves a similar result for ind-projective varieties like $\cF l_{\sigma}$.
\begin{prop} \label{bigetalelinebundles} Let $\sigma$ be a facet of $\mathbf{a}$. We have $H^1_{\acute{E}t}(\cF l_{\sigma}, \GG_m)=\Pic(\cF l_{\sigma})$.
\end{prop}
\begin{proof} For simplicity, we explain the case of $\sigma=\mathbf{a}$. The more general case can be proven by a very similar argument. Recall that $\cF l_{\mathbf{a}}=\varinjlim_{w \in \tilde{W}} S_w$ where $\tilde{W}$ is the Iwahori-Weyl group. We revisit the setup of the proof of Prop \ref{cohvan}. By the Grothendieck spectral sequence \ref{schubertoflagvariety}, we get
$$0 \ra R^1 \varprojlim H^0_{\acute{E}t}(S_w, \GG_m) \ra H^1_{\acute{E}t}(\cF l_{\mathbf{a}},\GG_m) \ra \varprojlim H^1_{\acute{E}t} (S_w, \GG_m) \ra R^2 \varprojlim H^0_{\acute{E}t}(S_w, \GG_m)
$$
on the big-\'etale site.
Now since $H^0_{\acute{E}t}(S_w, \GG_m)=H^0(S_w,\cO_{S_w}^\times)=\CC^\times$, so the first and last terms vanish. So the middle arrow is an isomorphism.  Now by \cite[II Theorem(4.3.1)]{tamme}, there is a canonical isomorphism $H^1_{\acute{e}t}(S_w, \GG_m) \simeq \Pic(S_w)$ where $\Pic(S_w)$ is the usual Picard group $H^1_{Zar}(S_w,\cO_{S_w}^*)$. Since the Iwahori-Weyl group $\tilde{W}$ equals the affine Weyl group $W_a$ for us, so $\varprojlim \Pic(S_w)= \oplus_{i \in \mathbf{S}} \Pic(S_{w_i})$ where $w_i$ is the simple reflection corresponding to the affine simple root $\alpha_i \in \mathbf{S}$. So we can conclude by \cite[Prop 10.1]{pradv} according to which $\Pic(\cF l_{\mathbf{a}})=\oplus_{i \in \mathbf{S}} \Pic(S_{w_i})$.
\end{proof}

\subsection{The case of one parabolic point and line bundles} \label{gagsigma}
In this subsection, let $x$ be the unique parabolic point, let $\sigma^\alpha$ for $\alpha \in \mathbf{S}$ be the facet of $\mathbf{a}$ where {\it only $\alpha$ does not vanish}. Thus the case $\sigma^{\alpha_0}=v_0$ corresponds to the parahoric group $G(\cO)$. For simplicity we will just write $\sigma$ for a zero-dimensional facet. 

\subsubsection{Construction of $\cG^{\mathbf{a}} \ra X$.} \label{constructiongpscheme} For any $\sigma$, we choose an alcove  $\mathbf{a}^x$ such that $\sigma$ lies in its closure. So we have a natural map $\cG_{\mathbf{a}^x} \ra \cG_{\sigma}$ over $\mathbb{D}_x$. Recall that we have assumed in (\S \ref{gpsch}) that the gluing functions $\{f_x\}$ lie in $Mor(\mathbb{D}^\circ_x,G)=G(K_x)$ where $K_x$ is the quotient field of $\hat{\cO_x}$.

As in the introduction, let us agree to denote by $\cG^\sigma$  the Bruhat-Tits group scheme on $X$ which restricts to  $\cG_\sigma \ra \mathbb{D}_x$  and $X^\circ \times G$  respectively and is obtained by gluing through $f_x$. Let $\cG^{\mathbf{a}}$ denote the group scheme which restricts to $\cG_{\mathbf{a}^x} \ra \mathbb{D}_x$ and $X^\circ \times G$  respectively and is constructed by gluing through $f_x$. {\it  So the natural map $\cG_{\mathbf{a}^x} \hra \cG_\sigma$ over $\Spec(\hat{\cO_x})$ extends to $\cG^{\mathbf{a}} \hra \cG^{\sigma}$ over $X$.} The group scheme $\cG^{\mathbf{a}}$ depends on our choices of $f_x$ and $\mathbf{a}^x$, but in this paper we will only need its existence. We have a morphism of stacks $\cM_X(\cG^{\mathbf{a}}) \ra \cM_X(\cG)$. 

\subsubsection{Construction of $\cM_X(\cG^{\mathbf{a}}) \ra \cM_X(G)$} \label{constructionstacks} While there is not necessarily a morphism of group schemes $\cG^{\mathbf{a}} \ra X \times G$, let us construct a morphism of algebraic stacks $\cM_X(\cG^{\mathbf{a}}) \ra \cM_X(G)$ whose fibers are full flag varieties $G/B$. We will need only its existence in the rest of the paper.

Let us construct a morphism of algebraic stacks $\cM_X(\cG^{\mathbf{a}}) \ra \cM_X(G)$. For each $x \in \cR$, let $w_x$ be an element in the affine Weyl group $W_a$ which maps $\mathbf{a}_0$ (cf \S \ref{lgpthedata}) to $\mathbf{a}^x$ . Set $v_x = w_x v_0$. Let $N$ denote the normalizer of $T$ in $G$. We choose an element $n_x \in N(K_x)$ which maps to $w_x$ where $K_x$ is the quotient field of $\hat{\cO_{X,x}}$. We shall view $n_x$ as an element in $G(K_x)$.

 Let $E \ra X$ be a principal $G$-bundle obtained by gluing the trivial bundles on $X^\circ$ and $\{\mathbb{D}_x\}_{x \in \cR}$ by $\{n_x f_x\}_{x \in \cR}$. Let $Ad(E) \ra X$ denote the adjoint group scheme of $E$. Then $Ad(E)$ is obtained by gluing the constant group scheme $X^\circ \times G$ with $\cG_{v_0}$ by $\{n_x f_x\}$. Thus $Ad(E)$ is obtained by gluing $X^\circ \times G$ with $\cG_{v_x}$ via $\{f_x\}$.
Since the group scheme $\cG^\mathbf{a}$ is obtained by gluing $X^\circ \times G$ and $\mathcal{G}_{\mathbf{a}^x}$ via $\{f_x\}$, and we have natural morphisms $\cG_{\mathbf{a}^x} \hra \cG_{v_x}$, so we obtain a natural map of group schemes $\cG^{\mathbf{a}} \ra Ad(E)$. This also furnishes \begin{equation} \label{phif} \phi_f: \cM_X(\cG^{\mathbf{a}}) \ra \cM_X(\Aut(E)).
\end{equation} 

 By Proposition \ref{redtoIwa}, its fibers are isomorphic to $G/B$. The principal bundle $E$ is a left $\Aut(E)$-torsor and right $G$-torsor on $X$. We have an isomorphism of stacks 
\begin{equation} \label{mapofstacksbasepoint} \mu_E: \cM_X(\Aut(E)) \ra \cM_X(G)
\end{equation} which sends a right $\Aut(E)$-torsor $\cF$ to the principal $G$-bundle $\cF \times_{\Aut(E)} E$. Here $\cF \times_{\Aut(E)} E$ denotes the space where   for local sections $f, g, e$ of $\cF$, $\Aut(E)$ and $E$ respectively we identify $(e g, g^{-1} f)$ with $(e,f)$. Its inverse is given by sending $F$ to $F \times_G E^{op}$. Here $E^{op}$ has the same underlying space as $E$ but for local section $e$ of $E^{op}$, $e.g$ is defined to be $e g^{-1}$ after viewing it as a local section of $E$. Thus we obtain a morphism of stacks with desired properties:
\begin{equation} \label{mapfromiwahoritoG}
\mu_E \circ \phi_f: \cM_X(\cG^{\mathbf{a}}) \ra \cM_X(\Aut(E)) \ra \cM_X(G).
\end{equation}

To prove our main theorem, we begin relating line bundles on $\cM_X(\cG^\sigma)$ with those on $\cM_X(\cG^{\mathbf{a}})$. To this end, we begin by reconciling our references 
\cite{faltings} and \cite{heinloth} on the one hand and \cite{BHf1} on the other because  we want to use  formulation \cite[Theorem 4.2.1]{BHf1} of \cite[Theorem 17]{faltings} and a similar formulation of \cite[Thm 7]{heinloth}.

Let us recall that the references \cite[Faltings]{faltings} and \cite[Heinloth]{heinloth} work with rigidified line bundles i.e line bundles whose restriction to the trivial torsor for moduli stacks and on the trivial coset for affine grassmannians becomes trivial. Let us make this more precise. Let $L$ be a line bundle on a stack $\cX \ra S$ together with a section $s: S \ra \cX$ and $L_s$ denotes the restriction of $L$ to $S$ through $s$. Then a rigidification of $L$ is a choice of a trivialization $\alpha: \cO_S \simeq L_s$. These references work over a connected noetherian base scheme $S$ and the curve $C \ra S$ is smooth projective and absolutely irreducible.
In particular, the moduli stack $\cM_{C \ra S}(G)$ of principal $G$-bundles is fibered over $S$. Therefore let us note that line bundles on $\cM$ are not rigidifiable automatically.

Now we recall the set up of \cite{BHf1}. Recall that for us (cf \S \ref{redcokerwt}) $\Pic(?)$ denotes the group of isomorphism classes of line bundles.
Recall if $\cX/k$ is an algebraic stack, then the Picard functor (cf \cite[Definition 2.1.1]{BHf1}) $\ul{\Pic}(\cX)$ is the functor that to a scheme $T$ of finite type over a field $k$ associates the group $\Pic(\cX \times T)/pr_2^*\Pic(T)$. If $\ul{\Pic}(\cX)$ is the constant presheaf given by an abelian group $A$, then following \cite{BHf1} we shall say that $\ul{\Pic}(\cX)$ is {\it discrete} and simply denote $\ul{\Pic}(\cX) \simeq A$.
These definitions naturally generalize to the case of ind-schemes also. For the various stacks $\cX$ and ind-schemes $\bX$ of interest to us, the Picard functor $\ul{\Pic}(\cX)$ (or $\ul{\Pic}(\bX)$) will be a finitely generated free abelian group hence discrete, the base space $S$ will be $\Spec(\CC)$ and the section $s: S \ra \cX$ will be given by the trivial torsor or the identity coset. Since $S=\Spec(\CC)$ so all line bundles on $\cX/S$ are rigidifiable. So when $\ul{\Pic}(?)$ is discrete, then the constant presheaf defined by $\Pic(?)$ on the category of schemes and $\ul{\Pic}(?)$ are isomorphic. There is a natural forgetful map from rigidified line bundles to isomorphism classes of line bundles and further, in this case,  once we choose arbitrary rigidifications for any set of generators of $\ul{\Pic}(\cX)$, then they determine a compatible choice of rigidifications on every line bundle in $\ul{\Pic}(\cX)$. Henceforth we choose once for all arbitrary rigidifications for a set of generators of $\ul{\Pic}(\cX)$. Moreover, we will consider line bundles upto rigidifications in the following sense. 

Let us choose a uniformizer $z \in \hat{\cO}_{X,x}$. Recall that we have a map $glue_{x,z}: \cF l_{\sigma} \ra \cM_X(\cG^\sigma)$ that on a coset $f L^+ \cG_{\sigma}$ glues the trivial $G$-torsor on $X \setminus \{x\}$ with the trivial $\cG_{\sigma}$-torsor on $\Spec(\hat{\cO}_{X,x})$. 

Let $Gr_G$ denote $LG/L^+G$. Let $G$ be simply-connected and simple. We will use \cite[Theorem 4.2.1]{BHf1} formulation of \cite[Theorem 17]{faltings}: we have  $\ul{\Pic}(Gr_G) = \ZZ$ and $glue_{x,z}^*: \ul{\Pic}(\cM_G) \ra \ul{\Pic}(Gr_G)$ is an isomorphism of functor.
Similarly, denoting isomorphism classes of line bundles by $\Pic(?)$, by $c$ the homomorphism given by central charge (cf \ref{centralcharge}), by $\cG_z$ the reduction of the group scheme to the closed fiber at $z$, by \cite[Thm 7]{heinloth} we have the exact sequence 
\begin{equation} \label{picseq}
0 \ra \prod_{z \in \cR} \XX^*(\cG_z) \ra \Pic(\cM_X(\cG)) \stackrel{c}{\ra} \ZZ \ra 0.
\end{equation}
 
Let us emphasize that in the reference \cite[Thm 7]{heinloth}, the middle term above is the Picard group of rigidifiable bundles.

\begin{prop} \label{pullbackalcove} Pull-back under $ \cF l_{\mathbf{a}} \stackrel{q_{\mathbf{a}}}{\lra} \cM_X(\cG^{\mathbf{a}})$ establishes an isomorphism 
\begin{equation} \label{pullbackiso} q_{\mathbf{a}}^*: \Pic(\cM_X(\cG^{\mathbf{a}}) \ra \Pic(\cF l_{\mathbf{a}})=\oplus_{ \alpha \in \mathbf{S}} \ZZ L_{\epsilon_{\alpha}}.
\end{equation} 
\end{prop}

\begin{proof}  For $q_{\mathbf{a}}$ from (\ref{sstele}) the terms 
$ 0 \ra E^{1,0}_2 \ra E^1_\infty \ra E^{0,1}_2$ work out to 
\begin{eqnarray*}
0 \ra H^1(BL_{X^\circ} G_\bullet, \GG_m) \ra H^1(\cM_X(\cG^{\mathbf{a}}), \GG_m) \ra  H^0(BL_{X^\circ} G_\bullet, H^1(\cF_{\mathbf{a}}, \GG_m))  .
\end{eqnarray*}
where all cohomology groups are computed in the big-\'etale topology. By Proposition \ref{vanBL}, we have $H^1(BL_{X^\circ} G_\bullet, \GG_m)=0$. Since $L_{X^\circ} G$ is connected, so the third group is just $H^1(\cF_{\mathbf{a}},\GG_m)$. So we have an injective map $H^1(\cM_X(\cG^{\mathbf{a}}), \GG_m) \ra H^1(\cF_{\mathbf{a}},\GG_m)$.

 The ample generator  of $\Pic(\cM_X(G))$ has central charge one (cf \cite[Theorem 17]{faltings}). It may be written as $L_{\sigma^{\alpha_0}}$ in our notations. Let $L_{\sigma^{\alpha_0}}(\mathbf{a})$ denote its pull-back under $\cM_X(\cG^\mathbf{a}) \rightarrow \cM_X(G)$ (\ref{mapfromiwahoritoG}). The central charge does not change under $\mu_E$ (cf \ref{mapofstacksbasepoint}) because it is an isomorphism of stacks. It also does not change under $\phi_f$ (cf \ref{phif})
 by \S \ref{centralchargestack}. So $L_{\sigma^{\alpha_0}}(\mathbf{a})$  also has central charge one. Taking $\cG=\cG^\mathbf{a}$, the sequence (\ref{picseq}) is split by mapping $1 \in \ZZ$ to $L_{\sigma^{\alpha_0}}(\mathbf{a})$. 
 
Consider $\cG_{\mathbf{a}} \ra G_A$. At the closed fiber, the image of $\cG_{\mathbf{a}} \otimes k$ in $G$ is $B$. So we may view a weight $\omega$ of $G$ as a character on $\cG^{\mathbf{a}}_x$ via 
\begin{equation} \cG^{\mathbf{a}}_x = \cG_{\mathbf{a}} \otimes k \ra B \ra T \stackrel{\omega}{\ra} \GG_m.
\end{equation} Let $L_\omega$ be the corresponding line bundle on $\cM_X(\cG^\mathbf{a})$ via (\ref{picseq}). Let $L_{\tilde{\omega}}$  denote the pull-back of $L_\omega$ via $\cF l_{\mathbf{a}} \ra \cM_X(\cG^{\mathbf{a}})$. By \cite[Theorem 4.2.1]{BHf1} we see that $L_{\sigma^{\alpha_0}} \ra \cM_X(G)$ pulls back to the generator of $Gr_G=\cF l_{\sigma^{\alpha_0}}$. This line bundle  pulls back to $L_{\epsilon_0} \ra \cF l_{\mathbf{a}}$ (cf \cite[(2.2.6)]{zhu}). By (\ref{epsilon-reln}) it follows that on $\cF l_{\mathbf{a}}$ we have an isomorphism 
 \begin{equation} \label{decomplineb}
L_{\epsilon_\alpha} \simeq L_{\tilde{\omega}_\alpha} \otimes L_{\epsilon_0}^{ a_\alpha^\vee}.
\end{equation} 
 This isomorphism together with the splitting of (\ref{picseq}) imply that an isomorphism is induced in (\ref{pullbackiso}) by the pull-back map that sends: $L_\omega \ms L_{\tilde{\omega}}$ and $L_{\sigma^{\alpha_0}}(\mathbf{a}) \ms L_{\epsilon_0}$.

\end{proof}

\begin{prop}  \label{redtoIwa} Let $\sigma$ be any facet. Let $\mathbf{a}$ be an alcove such that $\sigma$ lies in its closure.  Consider the diagram
\begin{equation} 
\xymatrix{
\cF l_\mathbf{a} \ar[r]_p \ar[d]^{q_\mathbf{a}} & \cF l_\sigma \ar[d]^{q_{\sigma}} \\
\cM_X(\cG^\mathbf{a}) \ar[r]^\pi & \cM_X(\cG^\sigma)
}
\end{equation} 
Then the diagram
\begin{equation} \label{cartesiansquare}
\xymatrix{
\Pic(\cF l_\mathbf{a})   & \Pic(\cF l_\sigma) \ar[l]^{p^*} \\
\Pic(\cM_X(\cG^\mathbf{a}))  \ar[u]^{q_{\mathbf{a}}^*} & \Pic(\cM_X(\cG^\sigma) \ar[u]^{q^*_\sigma} \ar[l]_{\pi^*})
}
\end{equation}
has all arrows injective and is a pull-back square. Further, we have $p_* \GG_m=\GG_m$ and $\pi_* \GG_m=\GG_m$. Lastly, let $G^\sigma$ denote the reductive quotient of $\cG_\sigma \otimes k$ and $F^\sigma$ its full flag variety. The sheaves $R^1p_* \GG_m$ and $R^1 \pi_* \GG_m$ are the trivial local systems with fibers isomorphic to $\Pic(F^\sigma)$.
\end{prop}
\begin{proof} By aruging exactly as in Proposition \ref{pullbackalcove}, we get that $q^*_\sigma$ is injective. By the uniformization theorem (cf (\ref{unif}) \cite{heinloth}, it follows that (\ref{redtoIwa}) is cartesian. 

Here the fibers of the horizontal maps are  isomorphic to the $k$-scheme given by $\cG_\sigma \otimes k/ \Img (\cG_\mathbf{a} \otimes k \rightarrow \cG_\sigma \otimes k)$. By Corollary \ref{borelpar}, this is a flag variety $F^\sigma$ of  $G^\sigma:=\cG_\sigma/\cG_\sigma^u$. One can see that $F^\sigma$ is the full flag variety by Corollary \ref{borelpar} and its proof as follows. Firstly, no affine root $\alpha$ takes an integral value on $\mathbf{a}$; so all inequalities in (\ref{Gsb}) are strict because, by definition, for $\alpha \in Y_\sigma$ we have $\alpha(\sigma)=0$ and no root vanishes on $\mathbf{a}$. Thus viewing $Y_\sigma$ as the root-system of $\cG_\sigma/\cG^u_\sigma$, the set $G_{\sigma,\mathbf{a}}$ (cf (\ref{Gsb})) has exactly one of $\alpha$ or $-\alpha$. So $\cG_{\mathbf{a}}/\cG_{\sigma}^u$ is a Borel subgroup of $G^\sigma$. 

Now $LG \ra \cF l_\sigma=LG/L^+\cG_{\sigma}$ has \'etale local sections by Theorem \ref{etlocsec}. So \'etale locally, the morphisms $p$ and $\pi$ are $F^\sigma$-fibrations. In particular, $\pi$ and $p$ are faithfully flat.
Let us show that the diagram (\ref{cartesiansquare}) is a pull-back square. Take any line bundle $L \in \Pic(\cF l_{\sigma})$. By Proposition \ref{pullbackalcove},  it suffices to show that the natural morphism $q_{\sigma}^* q_{\sigma,*} L \ra L$ of sheaves on $\cF l_{\sigma}$ is an isomorphism. This can be checked on any faithfully flat covering. So we will check that
\begin{equation} \label{check} \theta: p^* q_{\sigma}^* q_{\sigma,*} L \ra p^* L  
\end{equation} is an isomorphism. We first prove the following lemma.

\begin{lem} \label{pushforwardflatbasechangeindscheme} Let $\bX$ be an ind-scheme, say $\varinjlim_{n \in I} X_n$. Consider $a: \bX \ra \cX$ as in the diagram (\ref{pushforward}).  Let $\cG \in \bX_{\acute{E}t}$ be a quasi-coherent sheaves on $\bX$. Let $\cG_n$ be the restriction of $\cG$ on $X_n$. Suppose that there is a finite subcategory $I' \subset I$ such that $\cG= \varprojlim_{n' \in I'} \cG_{n'}$.  Then for any flat $u: U \ra \cX$, consider the cartesian square
\begin{equation}
\xymatrix{
\bX_u \ar[r]_{u'} \ar[d]^{a'} & \bX \ar[d]^{a} \\
U \ar[r]^u & \cX
}
\end{equation}
The canonical map $u^* a_* \cG \ra a'_* u'^* \cG$ is an isomorphism.
\end{lem}
\begin{proof}  Consider the cartesian squares
\begin{equation}
\xymatrix{
X_{n,u} \ar[r]_{u'_n} \ar[d]^{a'} & X_n \ar[d]^a \\
U \ar[r]^{u} & \cX
}
\end{equation}
So $\bX_u=\varinjlim_{n \in I} X_{n,u}$.  By definition of the site of ind-schemes and sheaves on it, $\cG$  is determined by its restrictions $\{ \cG_n|n \in I\}$ on $X_n$ for $n \in I$. Further we have
\begin{equation*}
(u'^* \cG)_n = u'^*_n (\cG_n),
\quad
a_* \cG= \varprojlim a_* \cG_n, \quad a'_* (u'^*\cG)= \varprojlim a'_* (u'^* \cG)_n = \varprojlim a'_* (u'^*_n \cG_n).
\end{equation*}
Since pushforward commutes with flat base-changes, we have the isomorphism 
\begin{equation} \label{pushforwardflat} u^* a_* \cG_n \ra a'_* u'^*_n \cG_n
\end{equation}  Therefore $u^* a_* \cG \ra  u^* a_* \varprojlim_{I'} \cG_n \ra u^* (\varprojlim_{I'} a_* \cG_n)  \ra \varprojlim_{I'} u^* a_* \cG_n \ra \varprojlim_{I'} a'_* u'^*_n \cG_n$
\begin{equation*}
   \ra a'_* \varprojlim_{I'} u'^*_n \cG_n \ra a'_* \varprojlim_{I'} (u'^* \cG)_n \ra a'_* u'^* \cG.
\end{equation*}
All arrows above are isomorphisms: the first and last by hypothesis,
the fourth  by (\ref{pushforwardflat}) and the remaining because pushforward commutes with arbitrary projective limits and pull-back by finite projective limits. The last two assertions follow since $(a^*,a_*): \bX_{\acute{E}t} \ra \cX_{\acute{E}t}$ (cf (\ref{mortopoiindschemestack})) and $(u^*,u_*): U_{\acute{E}t} \ra \cX_{\acute{E}t}$ are morphisms of topoi. So pushforward commutes with projective limits and pull-backs with finite projective limits.

\end{proof} 
By the proof of Proposition \ref{bigetalelinebundles}, the hypothesis of Lemma \ref{pushforwardflatbasechangeindscheme}  hold for line bundles on $\cF l_{\sigma}$. Indeed, we proved the following isomorphisms  $$H^1_{\acute{E}t}(\cF l_{\mathbf{a}},\GG_m) \ra \varprojlim H^1_{\acute{E}t}(S_w,\GG_m) \ra \varprojlim \Pic(S_w) = \oplus_{i \in \mathbf{S}} \Pic(S_{w_i}),$$
and therefore any line bundle on $\cF l_{\sigma}$ is determined by its restriction to any Schubert variety that contains all $S_{w_i}$ for $w_i \in \mathbf{S}$. So since $\pi$ is flat, we have by Lemma \ref{pushforwardflatbasechangeindscheme}
\begin{equation}
\pi^* q_{\sigma,*} L \stackrel{\simeq}{\ra} q_{\mathbf{a},*} p^* L.
\end{equation}
an isomorphism of sheaves. Applying $q_{\mathbf{a}}^*$, this gives $
p^* q_{\sigma}^* q_{\sigma,*} L = q_{\mathbf{a}}^* \pi^* q_{\sigma,*} L  \ra q_{\mathbf{a}}^* q_{\mathbf{a},*} p^* L \stackrel{adj}{\ra} p^*L$. The adjunction arrow $adj: q_{\mathbf{a}}^* q_{\mathbf{a},*} \ra \Id$ is an isomorphism because  all line bundles on $\cF l_{\mathbf{a}}$ descend to $\cM_X(\cG^{\mathbf{a}})$ by Proposition \ref{pullbackalcove}. So we have checked that $\theta$ in (\ref{check}) is an isomorphism.

Consider the \'etale local $F^\sigma$ fibration $p: \cF l_{\mathbf{a}} \ra \cF l_{\sigma}$ and an arbitrary open $u:U \ra \cF l_{\sigma}$. Let $(\cF l_{\mathbf{a}})_u:= U \times_{\cF l_{\sigma}} \cF l_{\mathbf{a}}$. Then $(\cF l_{\mathbf{a}})_u\ra U$ is a $F^\sigma$ fibration represented by a scheme. Further it is the final object of the category:
\begin{equation}
\xymatrix{
V \ar[r] \ar[d] & \cF l_{\mathbf{a}} \ar[d]^{p} \\
U \ar[r]^{u} & \cF l_{\sigma}
}
\end{equation}
of objects over $u$ together with a morphism to $\cF l_{\mathbf{a}}$ as in the above diagram. By Definition (\ref{pushforwardindschmes}) we have $(p_* \cO)_u=\varprojlim p_* \cO_V$. The inverse system reduces to $p_* \cO_{(\cF l_{\mathbf{a}})_u}$. So it identifies with $\cO_U$. Thus $p_* \cO=\cO$ on the big-\'etale site. Similarly $\pi_* \cO=\cO$ because $\pi$ is an \'etale local $F^\sigma$ fibration.

So $\pi_* \GG_m=\GG_m$ and $p_* \GG_m=\GG_m$ {\it on the big-\'etale sites}. Thus at closed points the  fibers of $R^1 \pi_* \GG_m$ and $R^1 p_* \GG_m$  are  isomorphic to $\Pic(F^\sigma)$. This isomorphism can be made canonical by the following observation. Let $y: \Spec(\CC) \ra \cM_X(\cG^{\sigma})$ be an arbitrary closed point. Choose any isomorphism $\theta$ of the fiber $\pi^{-1}(y)$  with $F^\sigma$ as left-$G^\sigma$ homogenous spaces. Since $G^\sigma$ is connected and $\Pic(F^\sigma)$ is discrete, so $\Pic(\theta): \Pic(\pi^{-1}(y) \ra \Pic(F^\sigma)$ is indpendent of  $\theta$. Since $p$ and $\pi$ are \'etale locally $F^\sigma$-fibrations and $\CC$-points are dense, so this shows that $R^1 p_* \GG_m$ and $R^1 \pi_* \GG_m$ are the trivial local systems.
\end{proof}

\begin{Cor} \label{pullbacktogen} When $\sigma$ is a vertex of $\mathbf{a}$, we have $\Pic(\cM_X(\cG^\sigma))=\ZZ$.
 Under $\pi: \cM_X(\cG^\mathbf{a}) \ra \cM_X(\cG^\sigma)$, the pull-back map corresponds to $\ZZ L_{\epsilon_\alpha} \hra \oplus_{ \alpha \in \mathbf{S}} \ZZ L_{\epsilon_{\alpha}}$.
\end{Cor}

Recall that $G$ is a simply-connected and semi-simple group. 

\subsection{The case $\cR=\{x\}$ and the facet $\sigma$ is alcove $\mathbf{a}$}

\begin{thm} \label{mtstackalcove}  The cohomological Brauer group $H^2_{\acute{E}t}(\cM_X(\cG^{\mathbf{a}}),\GG_m)_{tor}=0$.
\end{thm}
\begin{proof}    For $\cF l_{\mathbf{a}} \stackrel{q_{\mathbf{a}}}{\ra} \cM_X(\cG^{\mathbf{a}})$ from the spectral sequence (\ref{sstele}) we deduce 
$$ 0 \ra E^{1,0}_2 \ra E^1_\infty \ra E^{0,1}_2 \ra E^{2,0}_2 \ra ker(E^2_\infty \ra E^{0,2}_2) \ra E^{1,1}_2 \ra \cdots$$
(cf \cite[page 371, Cor 3.2]{merkur}).  Thus we have
\begin{eqnarray*}
0 \ra H^1(BL_{X^\circ} G_\bullet, \GG_m) \ra H^1(\cM_X(\cG^{\mathbf{a}}), \GG_m) \stackrel{\theta}{\ra}  H^0(BL_{X^\circ} G_\bullet, H^1(\cF l_{\mathbf{a}}, \GG_m))  \\ \ra H^2(BL_{X^\circ} G_\bullet, \GG_m) \ra \ker[H^2(\cM_X(\cG^{\mathbf{a}}), \GG_m)  \stackrel{\alpha}{\ra}  H^0(BL_{X^\circ} G_\bullet, H^2(\cF l_{\mathbf{a}},\GG_m))] \\ \ra H^1(BL_{X^\circ} G_\bullet, H^1(\cF l_{\mathbf{a}},\GG_m)) \ra \dots .
\end{eqnarray*}
where all cohomology groups are computed in the big-\'etale topology. Let us mention some simplifications. Now $H^1(BL_{X^\circ} G_\bullet, \GG_m)=0$ by Proposition \ref{vanBL}. Since $L_{X^\circ} G$ is connected and $H^i(\cF l_{\mathbf{a}},\GG_m)$ is discrete for $i=1,2$, so $L_{X^\circ} G$ acts trivially. Also from Prop \ref{vanBL}(1) we get $ H^1(BL_{X^\circ} G_\bullet, H^1(\cF l_{\mathbf{a}}, \GG_m))=0$. Using Corollary \ref{H2BLXG} we denote $H^2(BL_{X^\circ} G_\bullet, \GG_m)$ by $\Pic(L_{X^\circ}G)$.
So the sequence simplifies to 
\begin{equation} \label{Mercurevsimplified}
0 \ra \Pic(\cM_X(\cG^{\mathbf{a}})) \stackrel{\theta}{\ra} \Pic(\cF l_{\mathbf{a}}) \ra \Pic(L_{X^\circ} G) \ra \ker(H^2(q_{\mathbf{a}})) \ra 0
\end{equation}
where we denote $H^2(q_{\mathbf{a}}): H^2(\cM_X(\cG^{\mathbf{a}}),\GG_m) \ra H^2(\cF l_{\mathbf{a}},\GG_m)$. We have the inclusion $ H^2_{\acute{E}t}(\cM_X(\cG^{\mathbf{a}}), \GG_m)_{tor} \subset \Ker(H^2(q_{\mathbf{a}}))_{tor}$ since  $H^2(\cF l_{\mathbf{a}},\GG_m)_{tor}=0$ by Proposition \ref{cohvan}. So to prove the proposition, it suffices to show that $\ker(H^2(q_{\mathbf{a}}))_{tor}=0$. The morphism $\theta$ is the pull-back $q_{\mathbf{a}}^*$ of Prop \ref{pullbackalcove}. So $\theta$ is an isomorphism. 
 Now $\Pic(L_{X^\circ} G)$ equals $H^2(BL_{X^\circ} G_\bullet,\GG_m)$ which is torsion-free by Corollary \ref{H2BLXG}.   

\end{proof}

\subsection{The case $\cR=\{x\}$ and arbitrary facet $\sigma$}
\begin{thm} \label{mtstack}  The cohomological Brauer group $H^2_{\acute{E}t}(\cM_X(\cG^{\sigma}),\GG_m)_{tor}=0$.
\end{thm}
\begin{proof} On the big-\'etale site, the Leray spectral sequence for $p$  with values in $\GG_m$  gives the exact sequence $0 \ra H^1(\cF l_\sigma,p_* \GG_m) \ra H^1(\cF l_{\mathbf{a}}, \GG_m) \ra H^0(\cF l_{\sigma}, R^1 p_* \GG_m) \ra H^2(\cF l_{\sigma}, \GG_m) \ra H^2(\cF l_{\mathbf{a}}, \GG_m)$. Since $p_* \GG_m=\GG_m$  by Prop \ref{redtoIwa}, the above sequence reduces to
$$0 \ra \Pic(\cF l_\sigma) \stackrel{p^*}{\ra} \Pic(\cF l_\mathbf{a}) \ra H^0(\cF l_{\sigma}, R^1 p_* \GG_m) \ra H^2(\cF l_{\sigma}, \GG_m) \stackrel{H^2(p)}{\ra} H^2(\cF l_{\mathbf{a}},\GG_m) \ra $$ 
Now $H^2(\cF l_{\sigma}, \GG_m)_{tor}=0$ by Prop (\ref{cohvan}). So $ \ker(H^2(p))$ is torsion-free.  The above sequence reduces to 
$$0 \ra \Pic(\cF l_\sigma) \ra \Pic(\cF l_\mathbf{a}) \ra H^0(\cF l_{\sigma}, R^1 p_* \GG_m) \ra  \ker(H^2(p)) \ra 0 $$ 
 We will abbreviate $\cM_X(\cG^\sigma)$ and $\cM_X(\cG^{\mathbf{a}})$ as $\cM^\sigma$ and $\cM^\mathbf{a}$. We have a similar sequence for $\pi$ as well because $H^2(\cM^{\mathbf{a}},\GG_m)_{tor}=0$ by Theorem \ref{mtstackalcove}. 
 
Therefore the two sequences can be put in exact sequences
\begin{equation*}
\xymatrix{
 \Pic(\cF l_\sigma) \ar@{^{(}->}[r]  & \Pic(\cF l_\mathbf{a}) \ar[r]^{\alpha_{\cF}}  & H^0(\cF l_{\sigma}, R^1 p_* \GG_m) \ar[r] & \ker(H^2(p)) \ar[r] & 0 \\
 \Pic(\cM^\sigma) \ar@{^{(}->}[r] \ar@{^{(}->}[u]^{q_\sigma^*} & \Pic(\cM^\mathbf{a}) \ar[r]^{\alpha_{\cM}} \ar@{^{(}->}[u]^{q_{\mathbf{a}}^*} & H^0(\cM^\sigma, R^1 \pi_* \GG_m)  \ar[u] \ar[r] & \ker(H^2(\pi)) \ar[r] \ar[u] & 0 \ar[u]
}
\end{equation*}
Since the square (\ref{cartesiansquare}) is cartesian so the images of $\alpha_{\cF}$ and $\alpha_{\cM}$ are naturally isomorphic. 
Further $R^1 p_* \GG_m$ and $R^1 \pi_* \GG_m$  are the trivial local systems by Prop \ref{redtoIwa}. So the middle arrow identifies with identity on $\Pic(F^\sigma)$. Thus $\ker(H^2(\pi)) \ra \ker(H^2(p))$ is an isomorphism.  So $\ker(H^2(\pi))$ is torsion-freew.  Now $H^2(\cM^\sigma,\GG_m)_{tor} \subset \ker(H^2(\pi))$ because  by Theorem \ref{mtstackalcove} we have $H^2(\cM^{\mathbf{a}},\GG_m)_{tor}=0$. So it follows that the cohomological Brauer group $H^2_{\acute{E}t}(\cM^\sigma,\GG_m)_{tor}=0$.
\end{proof}

\subsection{Several points with facet alcove at each of the points}
\begin{prop} \label{sevptsalcove} For each $x \in \cR$, let us choose alcoves $\mathbf{a}^x$. Let $\cG^{\mathbf{a}}$ be a group scheme that restricts to $\cG_{\mathbf{a}^x} \ra \mathbb{D}_x$ at each $x \in \cR$. Let $\cM$ denote the moduli stack with facet $\sigma_x=\mathbf{a}^x$ at each $x \in \cR$. The cohomological Brauer group of $\cM$ vanishes.
\end{prop}
\begin{proof} By repeating the construction in \S \ref{constructionstacks} for multiple points, we get a morphism $\cM_X(\cG^{\mathbf{a}}) \ra \cM_X(G)$.
Let $\sigma_0^x$ be the unique vertex of $\mathbf{a}^x$ corresponding to the vertex of $\mathbf{a}$ where only $\alpha_0$ does not vanish. We have the cartesian square
\begin{equation}
\xymatrix{
\prod_{x \in \cR} \cF l_{\mathbf{a}^x} \ar[r]^{p} \ar[d]^{q} & \prod_{x \in \cR} \cF l_{\sigma_0^x} \ar[d]^{q_0} \\
\cM \ar[r]^{\pi} & \cM_X(G)
}
\end{equation} 
Reasoning exactly as in the proof of Theorem \ref{mtstackalcove}, 
for $q$ and $q_0$ the sequences analogus to (\ref{Mercurevsimplified}) fit into exact sequences
\begin{equation*}
\xymatrix{
\Pic(\cM) \ar@{^{(}->}[r]^{q^*} & \Pic(\prod_{x \in \cR} \cF l_{\mathbf{a}^x}) \ar[r] & \Pic(L_{X^\circ} G) \ar[r] & \ker(H^2(q)) \ar[r] & 0 \\
\Pic(\cM_X(G)) \ar@{^{(}->}[r]^{q_0^*} \ar@{^{(}->}[u] & \Pic(\prod_{x \in \cR} \cF l_{\sigma_0^x}) \ar[r] \ar@{^{(}->}[u] & \Pic(L_{X^\circ} G) \ar[r] \ar[u]^{\Id} & \ker(H^2(q_0)) \ar[r] \ar[u] & 0 \ar[u]
}
\end{equation*}
The leftmost square is a pull-back square by exactly the same reasoning as (\ref{cartesiansquare}). So by definition of pull-back square we have $coker(q_0^*) \hra coker(q^*)$. By Theorem \ref{constantcy}, and the existence of line bundle of central charge one on $\cM_X(G)$, both these cokernels are equal to $coker(\ZZ \stackrel{diag}{\lra} \ZZ^{\oplus \cR})$. Therefore the two images in $\Pic(L_{X^\circ} G)$ are equal. Therefore $\ker(H^2(q_0)) \ra \ker(H^2(q))$ is an isomorphism. Since $H^2(\cM_X(G),\GG_m)$ is torsion-free by the one point case, so $\ker(H^2(q_0))$ is also torsion-free. Thus $\ker(H^2(q))$ is torsion-free. On the other hand, $ H^2(\cM,\GG_m)_{tor} \subset \ker(H^2(q))$ because $H^2(\prod_{x \in \cR} \cF l_{\mathbf{a}^x},\GG_m)_{tor}=0$ by Proposition \ref{cohvan}. Therefore $ H^2(\cM,\GG_m)_{tor}=0$. This proves the result.
\end{proof}
\subsection{General case}
For each $x \in \cR$, for a facet $\sigma_x$ let us choose an alcove $\mathbf{a}^x$ such that $\sigma_x$ lies in its closure. Let $Z^x \subset \mathbf{S}$ be the set of affine simple roots corresponding to the affine roots bordering $\mathbf{a}^x$ but not vanishing at $\sigma_x$. Let $\cG^{\mathbf{a}}$ be the group scheme obtained from $\cG$ as in \S \ref{constructiongpscheme} using the $\{\mathbf{a}^x\}_{x \in \cR}$. So we have a morphism $\cG^{\mathbf{a}} \ra \cG$. We shall abbreviate $\cM_X(\cG)$ as $\cM$ and $\cM_X(\cG^{\mathbf{a}})$ as $\cM^{\mathbf{a}}$.
\begin{thm} \label{genstacks} The cohomological Brauer group of $\cM_X(\cG)$ is $\ZZ^{\oplus \cR}$ modulo $(1,\cdots,1)$ and $\{(0,\cdots,a_{\alpha_x}^\vee,\cdots,0) | \alpha_x \in Z^x, x \in \cR \} $.
\end{thm}
\begin{proof}  We have the cartesian square
\begin{equation}
\xymatrix{
\prod_{x \in \cR} \cF l_{\mathbf{a}^x} \ar[r]^{p} \ar[d]^{q} & \prod_{x \in \cR} \cF l_{\sigma_x} \ar[d]^{q_0} \\
\cM^{\mathbf{a}} \ar[r]^{\pi} & \cM
}
\end{equation}
For $q$ and $q_0$ the sequences analogus to (\ref{Mercurevsimplified}) fit into exact sequences
\begin{equation*}
\xymatrix{
\Pic(\cM^{\mathbf{a}}) \ar@{^{(}->}[r]^{q^*} & \Pic(\prod_{x \in \cR} \cF l_{\mathbf{a}^x}) \ar[r] & \Pic(L_{X^\circ} G) \ar[r] & \ker(H^2(q)) \ar[r] & 0 \\
\Pic(\cM) \ar@{^{(}->}[r]^{q_0^*} \ar@{^{(}->}[u] & \Pic(\prod_{x \in \cR} \cF l_{\sigma_x}) \ar[r] \ar@{^{(}->}[u] & \Pic(L_{X^\circ} G) \ar[r] \ar[u]^{\Id} & \ker(H^2(q_0)) \ar[r] \ar[u] & 0 \ar[u]
}
\end{equation*}
The leftmost square is a pull-back square by exactly the same reasoning as (\ref{cartesiansquare}). So by definition of pull-back square we have $i: coker(q_0^*) \hra coker(q^*)$. They are subgroups of $\Pic(L_{X^\circ}G)$ which is torsion-free by Prop \ref{H2BLXG}. Further they have the same ranks namely $|\cR|-1$ over $\ZZ$ by (\ref{picseq}). So $coker(i)$ is torsion. Consider
\begin{equation}
\xymatrix{
0 \ar[r] & coker(q^*) \ar[r] & \Pic(L_{X^\circ} G) \ar[r] & \ker(H^2(q)) \ar[r] & 0 \\
0 \ar[r] & coker(q_0^*) \ar[r] \ar[u]^{i} & \Pic(L_{X^\circ} G) \ar[r] \ar[u]^{\Id} & \ker(H^2(q_0)) \ar[r] \ar[u]^k & 0 \\
}
\end{equation}
So by snake lemma, we have $coker(i)=ker(k)$. Thus $\ker(k)$ is torsion.  On the other hand,
 $\ker(H^2(q))$ is torsion-free because $Br(\cM^{\mathbf{a}})=0$ by Proposition \ref{sevptsalcove}. Considering the $\ZZ$-ranks above, we find that $$\ker(k)=\ker(H^2(q_0))_{tor}.$$

Now $H^2(\cM,\GG_m)_{tor} \subset \ker(H^2(q_0))$ because $H^2(\prod_{x \in \cR} \cF l_{\sigma_x},\GG_m)_{tor}=0$ by Prop \ref{cohvan}. Therefore $$Br(\cM)=H^2(\cM,\GG_m)_{tor}=\ker(H^2(q_0))_{tor}=\ker(k)=coker(i).$$ We have 
\begin{equation}
coker(i)=coker( \Pic(\prod_{x \in \cR} \cF l_{\sigma_x}) \oplus \Pic(\cM^{\mathbf{a}}) \ra \Pic(\prod_{x \in \cR} \cF l_{\mathbf{a}^x}))
\end{equation}
Consider the central charge morphism $\Pic(\prod_{x \in \cR} \cF l_{\mathbf{a}^x}) \ra \ZZ^{\oplus \cR}$. It is surjective and its kernel is contained in $\Pic(\cM^{\mathbf{a}})$.
So $coker(i)$ is $\ZZ^{\oplus \cR}$ modulo the image of $\Pic(\prod_{x \in \cR} \cF l_{\sigma_x}) \oplus \Pic(\cM^{\mathbf{a}})$. This image works out to $(1,\cdots,1)$ and $$\{(0,\cdots,a_{\alpha_x}^\vee,\cdots,0)| \alpha_x \in Z^x , x \in \cR \}.$$ 
\end{proof}
So the Brauer group of the moduli stack is trivial for groups of type $A$ and $C$, but it may not be trivial in general. For instance one can take the case of two points with facets vertices corresponding to affine roots whose highest co-root coefficient is two.

\section{Brauer group of moduli space $M_X^{rs}(\cG)$} \label{bms}

From \cite[Prop 6.1.1]{pp} and \cite[Prop 6.3.1]{pp} we quote
\begin{prop} \label{isagerbe}  Let $G$ be a semi-simple simply-connected group. We have
\begin{enumerate}
\item[a)]  The open substack $\cM_X^{rs}(\cG)$ of regularly stable torsors, in characteristic zero, and when $g_X \geq 3$, has complement of co-dimension at least two.
\item[b)] The morphism $\cM_X^{rs}(\cG) \ra  M_X^{rs}(\cG)$ is a gerbe banded by $Z_G$.

\end{enumerate}
\end{prop}

As a corollary we get
\begin{Cor} \label{resiso} The restriction map $res: \Pic(\cM_X(\cG)) \ra \Pic(\cM_X^{rs}(\cG))$ induces an isomorphism. 
\end{Cor}
\begin{proof} By \cite[Prop 1]{heinloth}, $\cM_X(\cG)$ is a smooth algebraic stack which is locally of finite type. Now $\cM^{rs}_X(\cG)$ is a open substack whose complement has codimension at least two. So by \cite[Lemma 7.3 ii)]{BHf1}, the restriction map is an isomorphism.
\end{proof}

\begin{prop} \label{codim=2} The restriction map of cohomological Brauer group  $Br(\cM) \ra Br(\cM^{rs}_X(\cG))$ to the regularly stable locus is an isomorphism.
\end{prop}
\begin{proof} Let us abbreviate $\cM_X(\cG)$ as $\cM$ and $\cM_X^{rs}(\cG)$ as $\cM^\circ$. Recall that these stacks are smooth noetherian stacks over $\CC$. Since $\cM$ is smooth and noetherian, so $H^2_{\acute{e}t}(\cM,\GG_m)$ and $H^2_{\acute{e}t}(\cM^\circ,\GG_m)$ are torsion by a direct generalization of the standard proof for schemes (cf \cite[Theorem 9.1.5]{tamme}).


Let $A \rightarrow \cM$ be a smooth atlas of $\cM$ and let $A'$ denote the inverse image of $\cM^{\circ}$ in $A$. Let $A_p$ (resp. $A'_p$) denote the fiber product of $A$ (resp. $A'$) with itself $(p+1)$-times over $\cM$ (resp. $\cM^{rs}$). 
\begin{lem} We have spectral sequences \begin{equation} \label{twospec1} \xymatrix{ E^{p,q}_1=H^q_{\acute{e}t}(A_p,\GG_m ) \ar[d] & \implies & H^n_{\acute{e}t}(\cM,\GG_m) \ar[d] \\ E^{p,q}_1=H^q_{\acute{e}t}(A'_p,\GG_m) & \implies & H^n_{\acute{e}t}(\cM^\circ,\GG_m) }
\end{equation} whose differentials are compatible with restriction morphisms for every $E^{p,q}_r$ including $r=\infty$.
\end{lem} Here below we will not need the compatibility of the restriction of $E^{p,q}_\infty$ and $E^n$.
\begin{proof} Let  $K \ra I^\bullet$ be an injective resolution in the derived category of bounded below complexes of abelian sheaves on $A_\bullet$. Consider the  double complex
\begin{equation} 
\Gamma(A_p,I^q|_{A_p}).
\end{equation}
Then  $E^{p,q}_1=H^q_{\acute{e}t}(A_p,I^\bullet|_{A_p} ) \implies  H^n_{\acute{e}t}(\cM,K)$ is the spectral sequence associated to this double complex (cf the proof of Theorem \ref{ss} or \cite[Thm 6.11]{conrad}). The restriction morphism $res: Ab(\acute{e}t (A_\bullet)) \ra Ab(\acute{e}t(A'_\bullet))$ admits a left adjoint $L$ defined as follows: given $G \in Ab(A'_\bullet)$ and $u: U \ra A_p \in \acute{e}t(A_\bullet)$, define $L(G)_u$ by extending $G$ restricted to $U \times_{A_p} A'_p$ by zero  outside of $U'$. This is exact. This shows that $res$ sends injectives to injectives. Thus $K_{A'_\bullet} \ra I^\bullet |_{A'_\bullet}$ is an injective resolution. Since we have a morphism of double complexes compatible with the differentials
\begin{equation} 
\Gamma(A_p,I^q|_{A_p}) \ra \Gamma(A'_p, I^q|_{A_p})
\end{equation}
so for every $E^{p,q}_r$, including $r = \infty$ we get morphisms compatible with differentials.
\end{proof}

\begin{lem} The natural restriction map $H^q_{et}(A_p,\GG_m) \rightarrow H^q_{et}(A'_p,\GG_m)$ is an isomorphism for any $p$ and any $q \in \{0,1,2\}$.\end{lem}\begin{proof}For the case $q=0$, this follows by observing that $\cO_{A_p}^\times(A_p) \rightarrow \cO_{A'_p}^\times(A'_p)$ is an isomorphism since $A'_p \subset A_p$ is an open subset of codimension at least two. For $q=1$, this follows by identifying $H^1_{et}(A_p,\GG_m)$  with the group of line bundles on $A_p$  and since $A_p$ is smooth.

For $q=2$, notice that $A_p$ and $A'_p$ are both smooth and quasi-projective. For a smooth variety $Y$, recall that by a theorem of Grothendieck  $H^2_{\acute{e}t}(Y,\GG_m)$ is always torsion (cf \cite{milne}). On the other hand, by a theorem due to Gabber (cf \cite{dj}), for a quasi-projective variety, $H^2_{\acute{e}t}(Y,\GG_m)_{tor}$ coincides with the Brauer group of morita equivalence classes of Azumaya algebras. So $H^2_{\acute{e}t}(A_p,\GG_m)$ identifies with the Brauer group $Br(A_p)$. Thus the restriction map for $q=2$, identifies with the restriction map $Br(A_p) \ra Br(A'_p)$.  Let $Z_0 \hra A_p$ be the closed subscheme whose complement is $A'_p$. Since we are over $\CC$, so if $Z_0$ is not itself smooth then its singular locus $Z_1$ is a closed subvariety of strictly smaller dimension. Inductively define $Z_{n+1}$ as the singular subvariety of $Z_n$ till we reach a $k$ such that $Z_k$ is smooth. Such a $k$ exists, for instance when $dim(Z_k)$ becomes zero. From $A_p$, we remove $Z_k$ and then successively $Z_{n-1} \setminus Z_{n}$ for $n$ from $k$ to $1$. The restriction map $Br(X) \ra Br(X \setminus Z)$ is an isomorphism by \cite[Corollaire 6.2]{grot-brau} (cf \cite[Theorem VI.5.1]{milne}) if $X$ and $Z$ are regular and the codimension of $Z$ is at least two. Applying this result repeatedly, it follows that  $Br(A_p) \ra Br(A'_p)$ is an isomorphism. 
\end{proof}

\begin{lem} \label{twostep} The groups of $H^2_{\acute{e}t}(\cM(\cG),\GG_m)$ and $H^2_{\acute{e}t}(\cM^{rs}(\cG),\GG_m)$ have two-step filtrations whose associated graded are isomorphic.\end{lem} \begin{proof} The differentials $d^{p,q}_1: E^{p,q}_1 \rightarrow E^{p+1,q}_1$ in (\ref{twospec1})  are 
compatible with restriction morphisms which are isomorphisms.
So we have an isomorphism of groups $E^{p,q}_2$  of both spectral sequences for all $(p,q)$ where $q \leq 2$. Again the differentials $d^{p,q}_2: E^{p,q}_2 \rightarrow E^{p+2,q-1}_2$ are compatible with restriction morphisms. By analysing case by case, we shall show that the induced morphism between the $E^{p,q}_\infty$ terms, where $p+q=2$, is an isomorphism. To this end, we first consider the $E^{0,2}_\infty$ term. We have the sequences $E^{-2,3}_2 \rightarrow E^{0,2}_2 \rightarrow E^{2,1}_2 \rightarrow E^{4,0}_2$ and $E^{1,1}_2 \rightarrow E^{3,0}_2 \rightarrow E^{5,-1}_2$. Since $E^{-2,3}_2$ and $E^{5,-1}_2$ is zero, so we conclude that the restriction morphism is again an isomorphism for the sequence $ E^{-3,4}_3 \rightarrow E^{0,2}_3 \rightarrow E^{3,0}_3$. So $E^{0,2}_4$ are isomorphic for both spectral sequences. But now $E^{0,2}_4=E^{0,2}_\infty$. Now we consider $E^{1,1}_\infty$. We have the sequence $E^{-1,2}_2 \rightarrow E^{1,1}_2 \rightarrow E^{3,0}_2$. Reasoning as before, we conclude that   $E^{-2,3}_3 \rightarrow E^{1,1}_3 \rightarrow E^{4,-1}_3$ are isomorphic for both spectral sequences. But now $E^{1,1}_3 =E^{1,1}_\infty$. Now we consider $E^{2,0}_\infty$. We have the sequence $E^{0,1}_2 \rightarrow E^{2,0}_2 \rightarrow E^{4,-1}_2$ with isomorphic terms for both spectral sequences. Thus $E^{-1,2}_3 \rightarrow E^{2,0}_3 \rightarrow E^{5,-2}_3$ are isomorphic for both spectral sequences. Thus $E^{2,0}_3 =E^{2,0}_\infty$. \end{proof} 

Consider the restriction map $H^2(\cM,\GG_m)_{tor} \ra H^2(\cM^{rs},\GG_m)_{tor}$. 
By the so-called $5$-lemma it follows that it is an isomorphism.

\end{proof}

\subsection{Brauer group of $M^{rs}_X(\cG)$}

Recall that an Azumaya algebra of degree $n$ on a scheme $Y$ is a sheaf of $\cO_Y$-algebras that is \'etale locally isomorphic to $M_n(\cO_X)$. We say that two Azumaya algebras are Brauer-equivalent if there are locally free sheaves $V$ and $W$ on $Y$ of strictly positive rank at every point of $Y$ and an isomorphism $A \otimes End(V) \simeq B \otimes End(W)$ of $\cO_Y$-algebras. The Brauer group of $Y$ is the group of equivalence classes of Azumaya algebras on $Y$ with group operation induced by tensor product.

\subsection{Reduction to cokernel of weight homomorphism} \label{redcokerwt}
Let $M_X^{rs}(\cG)$ denote the moduli space of regularly stable $\cG_X$-torsors on $X$. 

Let us denote  the class of the gerbe $ \cM_X^{rs}(\cG) \ra  M_X^{rs}(\cG)$ by $\psi \in H^2_{\acute{e}t}(M_X^{rs}(\cG),\GG_m)$. {\it For a stack $\cX$, following \cite[BHf1]{BHf1} and \cite[Giraud]{giraudbook} by $\Pic(\cX)$ we shall denote the abelian group of isomorphism classes of line bundles on $\cX$}. 
 
Given any line bundle $L$ on $ \cM_X^{rs}(\cG)$, the {\it weight} of $L$ is a character $\chi: Z_G  \ra \GG_m$ defined as follows (cf \cite[BHf1]{BHf1}).  An object $\cE \in \cM(S)$, may be viewed as a $1$-morphism $S \ra \cM$. Let $L_{\cE} \ra S$ denote the pull-back line bundle. A line bundle $L$ on $\cM$ defines, for any $k$-scheme $S$, a functor from $\cM(S)$ to the groupoid of line bundles on $S$. So we have a group homomorphism 
$ \Aut (\cE) \ra \Aut_{S}(L_\cE)$.
Since $Z_G(S) \hra \Aut (\cE)$, we get $ Z_G(S) \ra \Aut_{S}(L_\cE)$ for every $\cE$. When $\cE$ is a closed point, then $L_\cE$ is a one-dimensional vector space. Thus we get $wt(L_{\cE}):Z_G \ra \Aut_S(L_\cE)=\GG_m$. We now suppose that $\cM$ is connected. Then  $wt(L_{\cE})$ is independent of $\cE$. So we get the weight homomorphism 
\begin{equation} \label{wt} wt(L): Z_G \ra \GG_m
\end{equation}

\begin{prop} \label{bgp} We have the exact sequence
\begin{equation} \label{mainseq}
1 \ra \Pic(M_X^{rs}(\cG)) \ra \Pic(\cM_X^{rs}(\cG)) \stackrel{wt}{\ra} \Hom(Z_G, \GG_m) \stackrel{\psi_*}{\ra} Br(M_X^{rs}(\cG)) \ra Br(\cM^{rs}_X(\cG))
\end{equation}

\end{prop}
\begin{proof} We start by recalling a general result for any gerbe $p:\cM \ra M$. For a sheaf $A \in \cM_{\acute{e}t}$, 
we have the Leray spectral sequence $
E^{p,q}_2=H^p_{et}(M,R^qp_*(A)) \implies H^{*}_{\acute{e}t}(\cM,A)$.
 This  gives the so called {\it Brauer sequence} (cf \cite[V.3.1.4.1]{giraudbook}) 
 \begin{eqnarray*}
0 \ra H^1(M,p_* A) \stackrel{\phi}{\ra} H^1(\cM,A) \stackrel{\gamma}{\ra}
H^0(M,R^1p_*(A)) 
\stackrel{d}{\ra} H^2(M,p_*(A)) \stackrel{\psi}{\ra} H^2(\cM,A) 
\end{eqnarray*}

Let us recall the explicit descriptions of the above morphisms. By \cite[Prop III.3.1.3]{giraudbook}, $\phi$ pulls back a $p_*(A)$ torsor and then extends structure group by $p^*p_* A \ra A$. For $\phi$, we temporarily specialize to the following case: $A$ is the sheaf defined by $\GG_m$ on $\acute{e}t(\cM)$. Evaluating over  $U$ we have $p_*(\GG_m)(U)=\GG_m(p^{-1}(U))$ which equals $\GG_m(U)$ since $p:\cM \ra M$ is a gerbe by \cite[Lemme 3.18]{lmb}. So $p_*\GG_m=\GG_m$. Thus in our case, $\phi$ is just the pull-back of line bundles from $M$ to $\cM$. By \cite[V.2.1.1]{giraudbook}, $\gamma$ may be described as follows: consider the presheaf $\cP \in M_{\acute{e}t}$ which on any object $Y \in \acute{e}t(M)$ takes value $H^1(p^{-1}(Y),A)$. Consider the sheafification functor which to a presheaf associates its sheaf. Applying this functor to $\cP$ we have a morphism $\cP \ra R^1p_*(A)$ . Evaluating this map at the open $M$ of the site $\acute{e}t(M)$ gives $\gamma$.
 
We specialize the above situation to $\cM_X^{rs}(\cG) \ra M_X^{rs}(\cG)$ which is a gerbe banded by $Z_G$. Setting $A=\GG_m$, we now want to analyse $\gamma$ and $H^0(M,R^1p_*(\GG_m))$. Let $\mathcal{H}om(Z_G,\GG_m)$ denote the constant sheaf in $M_{\acute{e}t}$. Consider the map $wt: \cP \ra \mathcal{H}om(Z_G,\GG_m)$ given by the weight homomorphism. So the weight homomorphism factors naturally as $\cP \ra R^1p_*(\GG_m) \stackrel{\theta}{\ra} \mathcal{H}om(Z_G,\GG_m)$. Let $m: \Spec(\CC)  \ra M$ be an arbitrary closed point. We have $R^1p_*(\GG_m)_m= \varinjlim_{Y'} H^1(\cM \times_M Y', \GG_m)$ where the colimit is taken over $Y'$ in the dual category of \'etale neighbourhoods of $m \in M$ (\cite[II.6.4]{tamme}). Now since $m$ is separably closed, so it is a final object of this category. Thus the stalk $R^1p_*(\GG_m)_m$ is simply $H^1_{\acute{e}t}(\cM \times_M m,\GG_m)$. By a theorem of Giraud the set of $Z_G$-gerbes on $m$ are classified by $H^2(m,Z_G)$ which is trivial. Thus the gerbe $\cM \times_M m \ra m$ must be $BZ_G$. Thus $H^1_{\acute{e}t}(\cM \times_M m,\GG_m)=H^1_{\acute{e}t}(BZ_G,\GG_m)=H^1(Z_G,\GG_m)=Hom(Z_G,\GG_m)$ which is also the stalk of $\mathcal{H}om(Z_G,\GG_m)$ at $m$. Since $M$ is a scheme of finite type over $\CC$, so by \cite[II.5.6(i)]{tamme}, it follows that $\theta$ is an isomorphism of sheaves. Thus after identifying $H^0(M,R^1p_*(\GG_m))$ with $Hom(Z_G,\GG_m)$, the morphism $\gamma$ is simply $wt$.
 
Summarizing the above, we get the following $5$-term short sequence
\begin{equation}
1 \ra \Pic(M) \ra \Pic(\cM) \stackrel{wt}{\ra} \Hom(Z, \GG_m) \stackrel{\psi_*}{\ra} H^2_{\acute{e}t}(M,\GG_m) \ra H^2_{\acute{e}t}(\cM,\GG_m).
\end{equation} 
 
For a smooth variety $Y$, recall that by a theorem of Grothendieck  $H^2_{\acute{e}t}(Y,\GG_m)$ is always torsion (cf \cite{milne}). On the other hand, by a theorem due to Gabber (cf \cite[Thm 1.1 de Jong]{dj}), for a quasi-compact separated scheme $Y$ endowed with an ample invertible sheaf, $H^2_{\acute{e}t}(Y,\GG_m)_{tor}$ coincides with the Brauer group of morita equivalence classes of Azumaya algebras.
 Taking $Y=M_X^{rs}(\cG)$, since it is contained in the smooth locus of $M_X^{ss}(\cG)$ by Proposition \ref{isagerbe}, and since by \cite[Balaji-Seshadri]{bs} $M_X^{ss}(\cG)$ exists as a quasi-projective variety, so by combining the theorems of Grothendieck and Gabber it follows that the Brauer group $Br(Y)$,  $H^2_{\acute{e}t}(Y,\GG_m)$ and $H^2_{\acute{e}t}(Y,\GG_m)_{tor}$ all coincide.

For any smooth scheme or stack $Y$, the inclusion of sites $i:\acute{e}t(Y) \ra \acute{E}t(Y)$ is a morphism of topologies (cf \cite[Definition I.1.2.2]{tamme}). Further $i$ is full-faithful and we observe that any covering in $\acute{E}t(Y)$ of $u \in \acute{e}t(Y)$ is also a covering in $\acute{e}t(Y)$. Thus by \cite[Theorem I.3.9.2]{tamme}, the functor $i^s:Y_{\acute{E}t} \ra Y_{\acute{e}t}$ defined for $F \in Y_{\acute{E}t}$ by $i^s(F)(u)=F(i(u))=F(u)$
for $u: U \ra \cX \in \acute{e}t(Y)$ is exact. Thus by the Leray spectral sequence \cite[I.3.7.6]{tamme},
$E^{pq}_2=H^p_{\acute{e}t}(Y,R^q i^s(F)) \implies H^{p+q}_{\acute{E}t}(Y,F)$
it follows that cohomology over the big and small \'etale sites agree.  So we may switch to the big-\'etale site in the above cohomology groups. 
 Further the image of $H^2_{\acute{E}t}(M^{rs},\GG_m)$ lands in $H^2_{\acute{E}t}(\cM^{rs},\GG_m)_{tor}$.  Now by  Proposition \ref{codim=2}, we get the desired result.
 
\end{proof}

{\it Thus to compute the Brauer group of $M$, we need to compute the image of $wt:\Pic(\cM) \ra Hom(Z_G, \GG_m) $ homomorphism}. This is carried out in the subsections below by relating it to the case where at the parabolic point $x$ in $\cR$ the parahoric group scheme corresponds to the Iwahori subgroup.

\subsection{The case of one parabolic point $x$ with facet a vertex $\sigma$ of the alcove} \label{onept}

\subsubsection{cokernel of weight in terms of cokernel of evaluation map}
For the convenience of the reader, we give a summary of \cite[Biswas-Hoffmann]{BHf1}. Let us place ourselves in the context $G/\CC$ is any reductive group. Recall that the connected components of the stack $\cM_X(G)$ are parametrized by $\pi_1(G)$. Let $d \in \pi_1(G)$ and let $\cM^{rs,d}_X(G)$ denote the corresponding component of the regularly stable locus. Now \cite[Prop 7.2]{BHf1} gives an alternate description of the cokernel of 
\begin{equation} wt^d:\Pic(\cM^{rs,d}_X(G)) \ra \Hom(Z_G,\GG_m).
\end{equation}  
To explain the alternate description, we need to first introduce some more notation. Let $\Lambda_{T}:= \Hom(\GG_m,T)$ denote the cocharacter group of $T$. Let $\Lambda_{coroots} \subset \Lambda_T$ denote the subgroup generated by the coroots of $G$. The Weyl group $W$ acts trivially on $\Lambda_T/\Lambda_{coroots}$. This quotient upto canonical isomorphism is $\pi_1(G)$. A bilinear form $b$ on $\Lambda_{coroots}$ is said to be even if $b(\lambda \otimes \lambda) \in 2 \ZZ$ for all $\lambda \in \Lambda_{coroots}$. Let $G^{ad}$ be the adjoint group of $G$.  Following \cite[Definition 6.3]{BHf1} let 
\begin{equation} \Psi(G) \subset \Hom(\pi_1(G^{ad}) \otimes \pi_1(G^{ad}), \mathbb{Q}/\ZZ),
\end{equation}
denote the abelian group of all bilinear maps that come from even $W$-invariant symmetric bilinear forms on $\Lambda_{coroots} \times \Lambda_{coroots} \ra \ZZ$. It is determined by the root system of $G$. If $G=G_1 \times G_2$, then $\Psi(G)=\Psi(G_1) \oplus \Psi(G_2)$. We refer the reader to \cite[Table 1]{BHf1} for an explicit description of $\Psi(G)$ and its generator for all simple groups. Let $G'$ denote the derived subgroup of $G$ and $Z^\circ_G$ the connected component of identity in $Z_G$. Following \cite[Definition 6.4]{BHf1}, let $\Psi'(G) \subset \Psi(G)$ be the subgroup of all elements such that 
\begin{equation} b(\pi_1(G/Z_G^\circ) \times \pi_1(G'))=0.
\end{equation}
 Given an element $d \in \pi_1(G)$, with image $\overline{d} \in \pi_1(G/Z_G^\circ)$, the evaluation map
\begin{equation} ev_G^d: \Psi'(G) \rightarrow \Hom(\frac{\pi_1(G^{ad})}{\pi_1(G')}, \mathbb{Q}/\ZZ),
\end{equation}
sends a bilinear form $b$ to $b(\overline{d},\underline{?})$. One can check directly that $\Hom(\frac{\pi_1(G^{ad})}{\pi_1(G')}, \mathbb{Q}/\ZZ)$ equals just $\Hom(Z_G,\GG_m)$ (for instance cf \cite[Remark 6.5]{BHf1}).
We recall 
\begin{prop} \label{coker=} \cite[Prop 7.2]{BHf1} $coker(wt^d) \simeq coker(ev^d_G) $.
\end{prop}

With this preparation, let us prove our main result for the case of one point. Recall that the diagram (\ref{cartesiansquare}) is cartesian. Let $L_\sigma \in \Pic(\cM_X(\cG^{\sigma}))$ correspond to the positive generator of $\Pic(\cF l_{\sigma})$.
\begin{thm} \label{mtonept} 
Let $\sigma=\sigma^\alpha$ for $\alpha \in \mathbf{S}$. Let $\omega_\alpha$ be the fundamental weight of $G/\CC$ for $\alpha \neq \alpha_0$ and set $\omega_{\alpha_0}$ as the trivial weight. Then $L_\sigma$ under $wt: \Pic(\cM_X(\cG^{\sigma})) \ra \Hom(Z_G,\GG_m)$ goes to  $Z_G \hra T \stackrel{\omega_\alpha}{\ra} \GG_m$. 
\end{thm}

\begin{proof} We first consider the case $\alpha=\alpha_0$.  In the following lines we show how to deduce this case from \cite{BHf1}.  Thus $\cG^\sigma$ is the constant group scheme $X \times G$. Now since our group is semi-simple and simply-connected so we have $\Psi'(G)=\Psi(G)$. Since $\pi_1(G)$ is trivial in our case, so we just have to consider $ev^{e}_G$ where $e \in G$ is the identity element. By the explicit description of $\Psi(G)$ in \cite[Table 1]{BHf1}, we see that $b(\overline{e}, \underline{?})$ is the trivial homomorphism in $\Hom(Z_G,\GG_m)$. So $coker(ev^e_G)=\Hom(Z_G,\GG_m)$.  By Proposition \ref{coker=}, $wt$ is the trivial morphism. 

We now consider the case $\alpha \in \mathbf{S} \setminus \alpha_0$. 
Consider the morphism of stacks $\pi: \cM_X(\cG^{\mathbf{a}}) \ra \cM_X(\cG^{\sigma})$ defined  by extension of structure group $\cG^{\mathbf{a}} \ra \cG^{\sigma}$.

\begin{lem} The weight map $wt$ as defined in \S \ref{redcokerwt} factors as follows:
\begin{equation} \label{wtfact}
\xymatrix{
\Pic(\cM_X(\cG^\sigma)) \ar@{^{(}.>}[r] \ar@/^1pc/[rr] & \Pic(\cM_X(\cG^\mathbf{a})) \ar[r]^{wt} & \Hom(Z_G,\GG_m).
}
\end{equation}
\end{lem}
\begin{proof}
To check the claim, setting  $S=\Spec(\CC)$, we take
\begin{enumerate}
\item an arbitrary closed point $\cE: S \ra \cM_X(\cG^{\mathbf{a}})$,
\item a line bundle  $L$ on $\cM_X(\cG^{\sigma})$. 
\end{enumerate} 
Let $\cE \ms \cE^\sigma$ under $\pi: \cM_X(\cG^{\mathbf{a}}) \ra \cM_X(\cG^\sigma)$.
 So we have a group homomorphism $\Aut(\cE) \ra \Aut(\cE^{\sigma})$. Let $L_S \in \Pic(S)$ and $L' \ra \cM_X(\cG^{\mathbf{a}})$ denote the pull-back line bundles $(\cE^\sigma)^*L$ and $\pi^*L$ respectively. Viewing a line bundle on a stack as a sheaf, we have a natural morphism $\alpha: \cE^*L' \ra (\cE^\sigma)^*L = L_S$. It is an isomorphism of line bundles. Let us check that we have a factorization which makes  the diagrams 
\begin{equation}
\xymatrix{
                   & \Aut (\cE^\sigma) \ar[rr]_{wt(L)}    && \Aut_S((\cE^\sigma)^*L) \\
Z_G \ar[r] \ar[ru] &  \Aut (\cE)       \ar[rr]^{wt(L')} \ar@{.>}[u] && \Aut_S(\cE^* L') \ar[u]_{\Aut(\alpha)}   
}
\end{equation}
commute. The triangle commutes because $\pi$ is a morphism of stacks. Considering  $S \stackrel{\cE}{\ra} \cM_X(\cG^{\mathbf{a}}) \stackrel{\pi}{\ra} \cM_X(\cG^{\sigma})$, by functoriality properties of sheaves on a stack applied to $L \ra \cM_X(\cG^\sigma)$ the square also commutes. We have $\Aut_S((\cE^\sigma)^*L)=\Aut_{S}(L_S) =\GG_m(S) $.
Since $\cM_X(\cG^{\mathbf{a}})$ is connected, so this show that $wt$ factors in (\ref{wtfact}).
\end{proof}
So the pull-back $L_{\sigma^{\alpha_0}}(\mathbf{a}) \ra \cM_X(\cG^\mathbf{a})$ of $L_{\sigma^{\alpha_0}} \ra \cM_X(G)$ has trivial weight. From (\ref{decomplineb}) we observe that under $\pi: \cM_X(\cG^{\mathbf{a}}) \ra \cM_X(\cG^{\sigma})$ we have  
\begin{equation} \pi^*L_\sigma \simeq L_{\omega_\alpha} \otimes L^{a_\alpha^\vee}_{\sigma^{\alpha_0}}(\mathbf{a}).
\end{equation} We want to compute the weight of $L_\sigma$. So we are reduced to computing the weight of $L_{\omega_\alpha} \ra \cM_X(\cG^{\mathbf{a}})$. To carry out the computation, we introduce some notation.
 
 Let $\omega=\omega_\alpha$. Let $S=\Spec(\CC)$. For any $\cE \in \cM_X(\cG^{\mathbf{a}})(S)$ viewed as a $1$-morphism $S \ra \cM_X(\cG^{\mathbf{a}})$, let $L_{\omega,S}$ denote the pull-back of $L_\omega$, let $\cE_x$ denote the restriction of $\cE \ra X \times S$ to the closed subscheme $x \times S$ and let $\cE_x(T)$ denote the $T$-bundle obtained by extension of structure group $\cG^{\mathbf{a}}_x \simeq \cG_{\mathbf{a}} \otimes k \ra B \ra T$. We have morphisms of stacks $\cM_X(\cG^{\mathbf{a}}) \ra \mathbb{B}(\cG^{\mathbf{a}}_x) \ra \mathbb{B}(T)$ where $\mathbb{B}(?)$ denotes the classifying stack of a group scheme. By construction $L_\omega$ comes from a line bundle in $\Pic(\mathbb{B}(T))=\XX^*(T)$. Since $\cM_X(\cG^{\mathbf{a}})$ is connected, so to compute $wt(L_\omega)$ by definition (cf (\ref{wt})) we want to compute the composite $$Z_G  \ra \Aut(\cE) \ra \Aut(\cE_x) \ra \Aut(\cE_x(T)) \ra \Aut(L_{\omega,S})=\GG_m(S).$$ To this end, we book-keep some morphisms to show that they are canonical.

\begin{lem} \label{redq=T} We have a natural isomorphism $\cG_{\mathbf{a}}/\cG^u_{\mathbf{a}} \ra T$.
\end{lem}
\begin{proof}  By Proposition \ref{raghu} it follows that $\cG_{\mathbf{a}}/\cG^u_{\mathbf{a}}$ has empty  root-system since $Y_\mathbf{a}$ is empty. For any $\theta \in \mathbf{a}^\circ$, we have $-\lfloor (\theta,r) \rfloor = 1 - \lceil (\theta,r) \rceil$ where $r$ is a root of $G$. By (\ref{defnpara}) and (\ref{defnparau}) it follows that $\cG_{\mathbf{a}}/\cG^u_{\mathbf{a}}$ is isomorphic to $T$. Further, this isomorphism can be made canonical since $\cG_{\mathbf{a}} \otimes k$ surjects {\it naturally} onto $B$.  
\end{proof}

\begin{lem} Let $\phi: Z_G  \ra \Aut(\cE)$ be given by the natural inclusion $Z_G \times X \hra \cG^{\mathbf{a}}$ followed by the right $\cG^{\mathbf{a}}$-action on $\cE$. Then $Z_G \ra \Aut(\cE)$ is given by $\phi$.
 \end{lem}
 \begin{proof} It is well known that for any principal $G$-bundle $F \ra M$, the automorphism group scheme $Aut_M(F)$ is represented by $F \times_{R-act,M,conj} G$, which is the quotient of $F \times G$ by the right action of $G$ on $F$ and action through conjugation on itself. If $f$ denotes a local section of $F$ and $g \in G$, then the class of $(f,g)$ in $Aut_M(F)$ represents the unique (local) $G$-bundle automorphism that sends $f \ms fg$ and therefore $fh \ms fgh= (fh) (h^{-1}gh)$ for $h \in G$. Considering $F \times_{R-act,M,conj} Z_G$, through $Z_G \hra G$ we get $Z_G \times M \hra Aut_M(F)$. Let $z \in  Z_G$. If $f \ms fz$, then $fh \ms fzh$ which equals $fhz$. This means that $Z_G$ action on $F$ through $G$-bundle automorphism coincides with the restriction of the action map $a: F \times G \ra F$ from $G$ to $Z_G$. This describes $Z_G \times M \hra Aut_M(F)$. From now on, let $\Aut(F)$ denote $\Gamma(M,Aut_M(F))$. So we have a description of $Z_G \ra \Aut(F)$, which is an inclusion. 
 
 More generally, let us describe $Z_G \ra \Aut(\cE)$ by $\Gamma$-$G$ bundle theory (cf (\ref{sbs})) as follows. Assume $\cE$ corresponds to a $\Gamma$-$G$ bundle $E \ra Y$ on a cover $Y \ra X$.  The composite $Z_G \stackrel{\phi}{\ra} \Aut(\cE)= \Aut_{\Gamma-G}(E/Y) \hra \Aut_Y(E)$ identifies with the natural inclusion $Z_G \hra \Aut_Y(E)$ for the $G$-bundle case described above. Therefore we have
 $Z_G \ra \Aut(\cE)$ is an inclusion and it is given by $\phi$.
 \end{proof}

 \begin{lem} The composite $Z_G \ra \Aut(\cE) \ra \Aut(\cE_x) \ra \Aut(\cE_x(T))=T$  identifies with the natural inclusion $Z_G \hra T$. 
 \end{lem} 
 \begin{proof}  Therefore $Z_G$-action on $\cE_x$ factors via the {\it natural} inclusion $Z_G \hra \cG^{\mathbf{a}}_x$ and $\cG ^{\mathbf{a}}_x$-action on $\cE_x$. Thus the composite $Z_G \hra \Aut(\cE) \ra \Aut(\cE_x)$ remains injective. Since $Z_G$-action on $\cE_x$ is via $Z_G \hra \cG^{\mathbf{a}}_x= \cG_{\mathbf{a}} \otimes k$, and by Lemma \ref{redq=T} the reductive quotient $\cG_{\mathbf{a}}/\cG^u_{\mathbf{a}}$ is isomorphic to $T$ canonically, so 
 \begin{enumerate}
 \item the homomorphism $Z_G \ra \Aut(\cE_x(T))$ is also an inclusion.
 \item Further, the $Z_G$-action on $\cE_x(T)$ is via the natural inclusion $Z_G \hra T$, coming from $Z_G \hra \cG_{\mathbf{a}} \otimes k \ra T$, followed by the right $T$-action on $\cE_x(T)$. 
 \end{enumerate} Now, $\Aut(\cE_x(T))=T$ because $T$ is abelian. Thus $Z_G  \hra \Aut(\cE) \ra \Aut(\cE_x) \ra \Aut(\cE_x(T))=T$  identifies with the natural inclusion $Z_G \hra T$.
 \end{proof}  
 
 We see that $T=\Aut(\cE_x(T))=\cE_x \times_{\cG_{\mathbf{a}} \otimes k} T \stackrel{(\Id,\omega)}{\lra} \cE_x \times_{\cG_{\mathbf{a}} \otimes k} \GG_m= \Aut(L_{\omega,S})=\GG_m$ is given by $\omega: T \ra \GG_m$ because the $\cG_\mathbf{a} \otimes k$ action is trivial on the second factor.  So $Z_G \hra \Aut(\cE_x(T)) \ra \Aut(L_{\omega,S})$ identifies with $Z_G  \hra T  \stackrel{\omega_\alpha}{\lra} \mathbb{G}_m $. Thus the weight of $L_\omega \ra \cM_X(\cG^{\mathbf{a}})$ is given by  $Z_G \hra T \stackrel{\omega}{\ra} \GG_m$. 
To summarize,  \begin{equation} \label{wtLepsilon} wt(L_{\sigma})=wt(L_{\omega_\alpha})=\omega_\alpha|_{Z_G}: Z_G \hra T \stackrel{\omega_\alpha}{\lra} \GG_m.
\end{equation} 
\end{proof}

\subsection{The General case} Let $\theta$ denote the highest root of $G/\CC$. For simple roots $\alpha$, let $a^\vee_\alpha$ be integers defined by the relation $\theta^\vee=\sum a^\vee_\alpha \alpha^\vee$. Set $a^\vee_{\alpha_0}=1$. For any facet $\sigma_x$ we define $ l_{\sigma_x}= GCD \{ a_\alpha^\vee | \alpha \in Y^{\sigma_x}  \}$. We define $ f =LCM \{l_{\sigma_x} | x \in \cR \}$.

\begin{proof}[Proof of Main Theorem \ref{mt2}] By Proposition \ref{bgp} $\ker(Br(M_X^{rs}(\cG)) \ra Br(\cM_X(\cG))$ is the cokernel of $wt:\Pic(\cM) \ra \Hom(Z_G,\GG_m)$. We place ourselves in the setup of the proof of Theorem \ref{genstacks}. Let $\cQ :=\prod_{x \in \cR} \cF l_{\sigma_x}$. So the following square

\begin{equation*}
\xymatrix{
\Pic(\cM^{\mathbf{a}}) \ar@{^{(}->}[r]^{q^*} & \Pic(\cF l_{\mathbf{a}}^{\times \cR})  \\
\Pic(\cM) \ar@{^{(}->}[r]^{q_0^*} \ar@{^{(}->}[u]^{\pi^*} & \Pic(\cQ)  \ar@{^{(}->}[u]^{p^*}
}
\end{equation*}
is a pull-back square. By Theorem \ref{constantcy}, the image of $q^*$ consists of line bundles with equal central charge in each factor. For $x \in R$, $p^*$ maps to line bundles whose central charge is a multiple of $l_{\sigma_x}$. Therefore the image of $q_0^*$ consists of line bundles whose central charge is a multiple of $f$ in each factor.

Let $m \in \ZZ$.  For $x \in \cR$, let $I^x_m$ denote the set of $|Z^x|$-uple integral solutions $e^x_m:=(\cdots, n^{m,x}_\alpha, \cdots) \in \ZZ^{Z^x}$ to  \begin{equation} \sum_{\alpha \in Z^x} n^{m,x}_\alpha a_\alpha^\vee = mf. \end{equation} Then for each solution $e^x_m \in I^x_m$, we have  a line bundle $L(e^x_m)$ on $  \cF l_{\sigma_x}$ of central charge $mf$ given by the box-tensor product over ${\alpha \in Z^x}$ of $ L_{\epsilon_\alpha}$ raised to the power $n_\alpha^{m,x}$. For $e=(\cdots,e^x_m,\cdots) \in \prod_{x \in \cR} I^x_m$, the line bundle  \begin{equation} L:=\boxtimes_{x \in \cR} L(e^m_x) \ra \cQ \end{equation} descends to $\cM_X(\cG)$. Further we argued above that  all line bundles on $\cM_X(\cG)$ are exactly of this form. For each  $e^x_m:=(\cdots,n^{m,x}_\alpha,\cdots) \in I^x_m$, consider the weight  \begin{equation} \omega(e^m_x)= \sum_{\alpha \in Z^x} n^{m,x}_\alpha \omega_\alpha \end{equation} of $G/\CC$. The weight of $L$ can be computed by the equation (\ref{wtLepsilon}) and it works out to  \begin{equation} Z_G \hra \prod_{x \in \cR} T \stackrel{\prod \omega(e^m_x)}{\lra} \prod \GG_m \stackrel{\prod}{\lra} \GG_m. \end{equation} Now varying over all line bundles on $\cM_X(\cG)$ is equivalent to varying over $e \in \cup_{m \in \ZZ} \prod_{x \in \cR} I^x_m$. To prove the theorem, it suffices to show that we can restrict ourselves to the case $m=1$. To this end, we shall show that  \begin{equation} I^x_m=I^x_1 + \cdots + I^x_1 \end{equation} where by a sum of sets we mean all possible sums in $\ZZ^{Z^x}$ of $m$-many elements from $I^x_1$. By definition of $I^x_1$, it follows that it is non-empty. Consider the functional  \begin{equation} (\cdots, a_{\alpha}^\vee, \cdots): \ZZ^{Z^x} \ra \ZZ \end{equation} and let $K$ be its kernel.  Let $z_1$ be arbitrary such that $z_1 \ms mf$. Pick $z \in I^x_1$. So $z  \ms f$. Thus $ z_1-mz =k \in K$. Now $z+k \in I^x_1$. Thus $z_1=(m-1) z + (z+k)$. \end{proof}

\section{Some computations} \label{somecomp}
In this section, we will use results from \cite[\S 1.2 Tits]{tits}. The notation will also be from (loc. cit) and \S \ref{onept}. Let $\Lambda$ (resp. $\Lambda^r$) denote the weight lattice (resp. root lattice) of $G/\CC$. We have $\Hom(Z_G,\GG_m)=\Lambda/\Lambda^r$. It is called the cocenter. We will denote as $C^*(\Delta)$. For $\alpha \in \mathbf{S}$ let $\omega_{\alpha}$ denote the fundamental weight, where for $\alpha_0$, we take the trivial weight. Let $\overline{\omega}_\alpha$ be the canonical image of $\omega_\alpha$ in $\Hom(Z_G,\GG_m)$ and set $\overline{\omega}_i=\overline{\omega}_{\alpha_i}$  for $ i \geq 0$.
If $\sigma=\sigma^{\alpha_i}$, then set $\cG^{i}=\cG^{\sigma}$. In the third column, we put the fundamental weight corresponding to $\alpha_i$. By $(i,j)$ we mean the GCD of $i$ and $j$. If $\alpha \in \mathbf{S}$ is the $i$-th root, then let $\cG^i:= \cG^{\alpha}$.
\begin{center}
\begin{tabular}{ |c|c|c|c| } 
 \hline
 Type & $C^*(\Delta)$ & weight $\overline{\omega}_i$ & $Br(M_X^{rs}(\cG^{i}))$ \\ 
 \hline
 $A_l$ & $\ZZ_l$ & $\overline{\omega}_i=i \overline{\omega}_1$ & $\ZZ_{(i,l+1)}$ \\ 
 \hline
 $B_l$ & $\ZZ_2$ & $\overline{\omega}_i=2 \overline{\omega}_l, i \neq l$ & $\ZZ_2 \, \text{for} \, i \neq l, $ \\ 
 & & & $e \,\, \text{for} \, \, i=l$ \\
 \hline
 $C_l$ & $\ZZ_2$ & $\overline{\omega}_i= i \overline{\omega_1}$ & $e$ \, $i$ \, \text{odd} \\
 & & & $\ZZ_2$ \, $i$ \, \text{even} \\
 \hline 
 $D_l$ & $\ZZ_4$  & $\overline{\omega}_{l-1}=\ol{\omega}_l^{-1}$ & $e$ \\
 & $l$  \text{odd} & $\ol{\omega}_1=\ol{\omega}_l^2$ & $\ZZ_2$ \\
 && $\ol{\omega}_i=i \ol{\omega}_1$ & $\ZZ_2$ \, \text{for} $i$ \, \text{odd} \\
 && \,\, \text{for} \, \, $i \leq l-2$ & $\ZZ_4$ \, \text{for} $i$ \, \text{even} \\
\hline 
& $\ZZ_2 \times \ZZ_2$ & $\ol{\omega}_{l-1}$, $\ol{\omega}_l$ & $\ZZ_2$ \\
& $l$ \text{even} & $\ol{\omega}_1=\ol{\omega}_{l-1} + \ol{\omega}_l$ & $\ZZ_2$ \\
& & $\ol{\omega}_i=i \ol{\omega}_1$  & $\ZZ_2$ \, \text{for} \, $i$ \, \text{odd} \\
& & \, \text{for} $i \leq l-2$ & $\ZZ_2 \times \ZZ_2$ \, $i$ \, \text{even} \\
\hline
$G_2$ & $e$ & - & $e$ \\
\hline
$F_4$ & $e$ & - & $e$ \\
\hline
$E_6$ & $\ZZ_3$ & $\ol{\omega}_i= i \ol{\omega}_1$ & $\ZZ_{(3,i)}$ \\
\hline 
$E_7$ & $\ZZ_2$ & $\ol{\omega}_i$, \, $i=4,6,7$ & $e$ \\
& & $\ol{\omega}_i$ \, $i \neq 4,6,7$ & $\ZZ_2$ \\
\hline
$E_8$ & $e$ & - & $e$ \\
\hline
\end{tabular}
\end{center}

\section{Cross-checks} \label{crosscheck}
\subsection{With \cite[Biswas-Holla]{bh}}
The case of \cite{bh} is that of principal $G$-bundles, where $G$ is semi-simple. The set $\cR$ is reduced to a single element, say $\{x\} \in X$. The parahoric group scheme is $G_A \ra \Spec(A)$ where $A=\CC[[t]]$. So the facet $\sigma$ is the origin of the apartment $\cA_T$. Only the affine simple root $\alpha_0:=\delta-\theta \in \mathbf{S}$ (cf (\ref{roots})) does not vanish at $\sigma$. So $\sigma = \sigma^{\alpha_0}$ and also $f=l_{\sigma}=a_{\alpha_0}^\vee=1$. Thus $\omega_{\alpha_0}$ is the trivial weight. Thus by Theorem \ref{mtonept}, it follows that  $Br(M)=\Hom(Z_G,\GG_m)=Z_G^\vee$. This agrees with \cite{bh}.

\subsection{With \cite[BBGN]{bbgn}} The case of \cite{bbgn} is that of the moduli space of vector bundles of rank $n$ and degree $d$ of fixed determinant $\xi$. Here below, we proceed to reinterpret this moduli space in the context of this paper which is that of torsors under Bruhat-Tits group scheme. We refer the reader to \cite[\S 3.4]{bbp} for a summary of $\Gamma$-$G$ bundle  theory of \cite{bs}.

In the case $d=0$, by taking a $n$-th root of $\xi$, we may reduce to the previous case of \cite{bh}. So without loss of generality, we may suppose that $-n<  -d < 0$. Let us fix a point $x \in X$. By putting the full flag $F_x=V_x$ at $x$ together with parabolic weight $d/n$, the space $M(n,-d, \xi)$ is equivalently the moduli space $ParM(n,-d)$ of  parabolic vector bundles of {\it fixed parabolic determinant}. Note that the parabolic degree has become zero. By the well-known correspondence between $\Gamma$-$SL_n$ bundles and parabolic vector bundles of fixed parabolic determinant (\cite[\S 3.4]{bbp}), $ParM(n,-d)$ is isomorphic to $M^{\tau}_Y(\Gamma,SL_n)$ which parametrizes $S$-equivalences classes of $\Gamma$-$SL_n$ bundles on a cover $p: Y \ra X$ where for $y \in Y$ lying over $x$, and denoting $\zeta_n$ the $n$-th root of unity, the isotropy representation $\rho_y: \Gamma_y \ra Aut(E_y)$ is given by $\zeta_n \ms \zeta^{d}_n[Id]_{n \times n}$. Since $\zeta_n^{d}=\zeta_n^{d-n}$, we may rewrite $\rho_y$ equivalently as 
\begin{equation}
\rho_y(\zeta_n)=\begin{pmatrix} \zeta_n^d Id_{(n-d) \times (n-d)} & 0 \\ 0 & \zeta_n^{d-n} Id_{d \times d}
\end{pmatrix}.
\end{equation}
As in \cite[Remark 2.2.9]{bs}, let $\omega$ (resp. $z$) be a local parameter around $y$ (resp. $x$). So $z=\omega^n$. Then if $L$ is the quotient field of the completion of the local ring at $y$, we set a one parameter subgroup $\Delta: \GG_{m,L} \ra SL_n$ by 
\begin{equation}
\Delta(\omega)=\begin{pmatrix}
\omega^d Id_{(n-d) \times (n-d)} & 0 \\ 0 & \omega^{d-n} Id_{d \times d}
\end{pmatrix}.
\end{equation}
Thus to $\Delta$ we can correspond a rational 1-PM subgroup $\theta=\frac{1}{n} \Delta$ in the apartment $\cA_T=Y(T) \otimes_\ZZ \RR$. Notice that only the affine simple root $\alpha_d$ of $SL_n$, which is $\epsilon_d-\epsilon_{d+1}$, does not vanish on $\theta$. Thus $\theta$ belongs to the facet $\sigma^Y$ where $Y=\{ \alpha_d\}$.

Set $B=\widehat{\cO}_{Y,y}$ and $N_y=Spec(B)$ and $U_y=Aut_{\Gamma_y-SL_n}(E|_{N_y}) \ra N_y$  the unit group scheme of local automorphisms of a $\Gamma$-$SL_n$-bundle, say $E$. By \cite[Theorem 2.3.1, Proposition 5.1.2]{bs} we have an isomorphism of group schemes
\begin{equation} U_y=\cG_{\sigma^{\alpha_d}}.
\end{equation} Thus in (\ref{gpscheqn}), we see that $p_*^{\Gamma} {}_EG = {\mathcal 
G}$ where $\cG \ra X$ is obtained by gluing $\cG_{\sigma^{\alpha_d}} \ra \Spec(\widehat{\cO}_{X,x})$ with $SL_n \times X \setminus \{x\}$. Now by (\cite[\S 3.4]{bbp}), $M_Y^{\tau}(\Gamma,SL_n)$ may be interpreted as $M(\cG)$, where we take the weight $\theta$ to check semi-stability conditions of parahoric torsors.  We are in the seting of Theorem \ref{mtonept}: since $\sigma=\sigma^{\alpha_d}$, so we consider $Z_G \ra T \stackrel{\omega_d}{\ra} \GG_m$. Thus if $d=0$, then $Br(M(\cG))=\ZZ/n\ZZ$. If $d \neq 0$, then $\omega_d=\epsilon_1+\cdots+\epsilon_d$ where $\epsilon_i:T \ra \GG_m$ on $T$, which is split, is $T \hra \GG_m^n$ followed by projection onto the $i$-th factor. Further $f=l_{\sigma}=a_{\alpha_d}^\vee=1$. By Theorem \ref{mtonept}, $Br(M^{rs})=Hom(Z_G,\GG_m)/<Z_G \ra T\stackrel{\omega_d}{\ra} \GG_m>$. Now $Z_G$ is generated by $\zeta_n Id_{n \times n}$. Thus $\omega_d:Z_G \ra \GG_m$ sends $\zeta_n \ms \zeta_n^d$. So the order of $\omega_d$ restricted to $Z_G$ is $n/(n,d)$. Now the result $Br(M(n,d))=\ZZ/(n,d)\ZZ$ in \cite{bbgn} follows.

\begin{rem} \label{proofstrategydiff} We want to compare the proof strategy of \cite[BBGN]{bbgn} with our paper. After Proposition \ref{bgp}, the Brauer group computation reduces to the computing the weight homomorphism (cf (\ref{wt})). As we remarked in \S \ref{redcokerwt}, $L_{\cE}: Z_G \ra \GG_m$ is independent of $\cE$ in a connected moduli stack. In \cite{bbgn}, this is computed through \cite[Dr\'ezet-Narasimhan, Prop 5.1]{dn} by taking $\cE$ to be a semi-stable vector bundle of maximal automorphism group i.e the direct sum of $(n,d)$-many copies of a stable vector bundle of rank $\frac{n}{(n,d)}$ and degree $\frac{d}{(n,d)}$. In our paper we use instead \cite[Biswas-Hoffmann, Prop 7.2]{BHf1} to compute the weight homomorphism. Its proof uses several ingredients from \cite[Biswas-Hoffmann]{BHf3} which tie together in \cite[ Prop 5.2.11]{BHf3}. Let us make this more precise. Let $\cM_G$ denote the moduli stack of $G$-bundles on $X$. In \cite{BHf3}, the weight homomorphism is computed explicitly for the determinant of cohomology line bundle on $\cM_{\GG_m}$ in Lemma 4.4.1 and then generalized to $\cM_{SL_n}$ in Corollary 4.4.2. It is generalized to the case of any line bundle for any semi-simple simply-connected group $G$ in Prop 4.4.7(iii) and for any reductive $G$ in \cite[Section 5]{BHf3}. More precisely, it is encoded in the description of the maps $\Pic(\cM^d_G) \ra NS(\cM^d_G)$ and $NS(\cM^d_G) \ra NS(\cM^\delta_{T_G})$ where $d \in \pi_1(G)$, $\delta \in \pi_1(T_G)$ and $NS$ denotes the N\'eron-Severi group. After these preparations, \cite[Prop 5.2.11]{BHf3} is deduced. So to compute the weight homomorphism, the proof strategies of \cite{bbgn} and our paper are different.
\end{rem}

\subsection{With \cite[Biswas-Dey]{bd}} In \cite{bd}, one treats the case  parabolic vector bundles of fixed determinant say $\xi$. Let us first consider the case when the parabolic degree is zero, since by taking $n$-root of $\xi$, this case corresponds directly to the case $G=SL_n$ in our paper. For simplicity, we may suppose that there is only one parabolic point $x$. At $x$, suppose that the flags are $0 \neq F_a \subset \cdots \subset F_1 = E_x$, where $dim(F_i)=d_i$. Then the affine roots would be $\{\alpha_{d_i} | 1 \leq i \leq a_x \}$.  Since for $SL_n$, $a^\vee_{\alpha}=1$ for all $\alpha \in \mathbf{S}$, so $f=l_{\sigma}=GCD(\{a_{\alpha_{d_i}}^\vee | 1 \leq i \leq q \})=1$. Thus the weights would be $\{ \omega_{d_i} | d_i \neq 0 \}$. For $SL_n$ the fundamental weight $\omega_d$ equals $\epsilon_1 + \cdots + \epsilon_d$. Viewing $Z_G$ as generated by $\zeta_n Id_{n \times n}$, we have $\omega_d(\zeta_n)=\zeta_n^d$. So by Theorem \ref{mtonept}, $Br= \ZZ_{(n,d_i| 1 \leq i \leq a_x)}.$ This equals $\ZZ_{(n,d_i-d_{i-1})}$ which is the description in \cite{bd} in terms of the flag-type. The case of multiple parabolic points can be treated as in the previous section. The case of a general $d$ can be handled as we have done above for \cite{bbgn}. This changes the answer to $\ZZ_{(n,d,d_i^x| 1 \leq i \leq a_x, x \in \cR)}$.

\section{A possible proof strategy} \label{strategy}
Let us review how the proof strategy of \cite[Biswas-Dey]{bd} may used to prove our result. For simplicity let us consider the case of only one point $p$. Let $\cI$, $G$, $\cG$ denote the group schemes on $X$ obtained by patching  the constant group scheme $ G \times X \setminus \{p\}$ with a local Iwahori group scheme, or the constant group scheme or an arbitrary Bruhat-Tits group scheme on $Spec(\widehat{\cO}_{X,p})$ respectively.
We have a diagram of Hecke modification (cf \cite{bs}) between the associated stacks
\begin{equation}
\xymatrix{
&\cM(\cI) \ar[rd] \ar[ld] & \\
\cM(G) &  & \cM(\cG)
}
\end{equation}
The idea is essentially to first deduce results on moduli  space of $\cI$-torsors from $G$-bundles and then deduce results for general $\cG$-torsors.

For vector bundles with fixed parabolic structure, Thaddeus (cf \cite{thaddeus}) proved that moduli spaces corresponding to different choice of weights have isomorphic blowups along closed subschemes of codimension at least two. It follows then that their Brauer groups are isomorphic. So once the quasi-parabolic structure is fixed, it becomes sufficient to prove the results for small weights. Now let us suppose below that this result is generalized also to the case of moduli of parahoric torsors.

For small weights, we have a forgetful map
$$ \pi_0: \cP \cM_s \ra \ol{\mathcal{N}}$$
from the moduli $\cP \cM_s$ of stable parabolic bundles to the moduli of semi-stable vector bundles by associating the underlying vector bundle. Let $\mathcal{N} \subset \ol{\mathcal{N}}$ denote the stable locus. It is simply connected (cf \cite{bbgn}) and it is known that the moduli of stable $G$-bundles is also simply-connected (cf \cite{blr}). By considering the Leray Spectral sequence for $\pi_0$ of the sheaf defined by $\GG_m$, we deduce the exact sequence \begin{equation} \label{bdey} Pic(Flag) \stackrel{\theta}{\rightarrow} Br(\ol{\mathcal{N}}) \ra Br(\cP \cM_s) \ra 0 ,
\end{equation} from which $Br(\cP \cM_s)$ is deduced. Analogus to $\pi_0$ a similar forgetful map between the moduli spaces $\pi: M_X^{rs}(\cI) \ra M^{ss}_X(G)$ may be estabilished for small weights. By our results, $M_X^{rs}(\cI)$ is the smooth locus of $M^{ss}_X(\cI)$. Let $M \subset M_X^{rs}(\cI)$ be the inverse image of $M^{rs}_X(G)$. Analogus to (\ref{bdey}) a sequence for $\pi: M \ra M^{rs}_X(G)$ may also be established. The group $Br(\ol{\mathcal{N}})$ is cyclic. The role of its generator   will be played by the generator of $Br(M^{rs}_X(G))$ which is the class of the $Z_G$-gerbe $\cM_X^{rs}(G) \ra M_X^{rs}(G)$  in $H^2(M_X^{rs}(G),\GG_m)$ (cf \cite{bh}).

If the complement of $M$ in $M^{rs}_X(\cI)$ has codimension at least two then one may be able to deduce $Br(M^{rs}_X(\cI))$.

By similarly adjusting weights and estimating codimensions, one may establish a map from $\pi': M \ra M^{rs}_X(\cG)$ where $M$ is an open subset $M \subset M^{rs}_X(\cI)$. By considering the descent spectral sequence of $\pi'$, it may be possible to deduce $Br(M^{rs}_X(\cG)$ from $Br(M)$.

\section{Appendix: Cohomological descent}
\subsection{Ind-schemes, morphisms and fiber products} \label{indschemes}
Let $(I,\leq)$ be a partially ordered set. Thus the relation $\leq$ is reflexive and transitive. We will view it as a category. For applications in this paper, $(I,\leq)$ will be the Iwahori Weyl group $\tilde{W}$ with Bruhat order.


Recall an ind-scheme $\mathbb{X}$ indexed by $I$ is a collection of schemes $\{X_i\}$ together with a collection of closed immersions $i_{j,k}:X_j \hra X_{k}$ for $j \leq k$ where $j,k \in I$. We shall assume that each 
$X_n$ is a finite dimensional scheme. Let $\mathbb{X}$ and $\mathbb{Y}$ be two ind-schemes indexed by categories $I$ and $J$. By a morphism $f: \mathbb{X} \ra \mathbb{Y}$ between ind-schemes we mean the following (cf \cite{sk}) :  for
every $n \in I$, there exists $m(n) \in J$ and $f_n: X_n \ra Y_{m(n)}$ a morphism of schemes such that if $i_1 \leq i_2$, then $f_{i_2}$ restricts to $f_{i_1}$. Thus for each non-empty $I$, the category of schemes embeds diagonally into the category of ind-schemes indexed by $I$. For $n \in I$, consider the partially ordered set $I_{\geq n}$ of elements $j \geq n$. Then $X_n$ extends to an ind-scheme on $I_{\geq n}$. Let us denote the ind-scheme again by $X_n$. So we may define a morphism from ind-scheme indexed by $I_{\geq n}$ to ind-scheme indexed by $I$
\begin{equation} \label{xntoxx}
X_n \ra \XX
\end{equation} using the $i_{j,k}$ for $j,k \geq n$.
Let $a: \bX \ra \cX$ and $b: \bY \ra \cX$ be two ind-schemes indexed by $I$ and $J$ respectively. Then the product category $I \times J$ has a partial order defined by: $(i_1,j_1) \leq (i_2,j_2)$ if $i_1 \leq i_2$ and $j_1 \leq j_2$. By the fibered product 
\begin{equation} \label{fpindsch} \bX \times_\cX \bY
\end{equation} we mean the ind-scheme $\{ X_i \times_\cX Y_j\}_{(i,j) \in I \times J}$ together with the natural closed immersions. We do have projection morphisms $\bX \times_\cX \bY \ra \bX$ and to $\bY$ using the projection functors from $I \times J$ to $I$ and $J$. We also have the diagonal morphism 
\begin{equation} \label{diagonal}
\bX \ra \bX \times_\cX \bX,
\end{equation}
by sending $X_i \ra X_i \times_\cX X_i$ for $i \in I$. So given $\bX \ra \cX$ we can make a simplicial ind-scheme augemented by a stack. If $J$ is filtered i.e given $j, k \in J$ there exists $i \in J$ such that $j \leq i$ and $k \leq i$, then given $f: \bX \ra \bY$, by the nerve construction, we can make a simplicial ind-scheme augemented by an ind-scheme.

\subsection{Sites of ind-schemes} \label{sitesofindscheme}
\begin{rem} \label{whybitetale} Let us explain why we have chosen to work with the Big-\'etale sites of ind-schemes $\bX$ and stacks $\cX$ in this section.  Further, we use the fact that Big-\'etale site is functorial \cite{olsson}. In particular, in our context, it is well-adapted to define a morphism of topoi. More precisely, for $a: \acute{E}t(\bX) \ra \acute{E}t(\cX)$ the condition `$a^*$ commutes with finite limits' is easier to check for the big-\'etale site than it is for other sites. Similarly to determine some results for $L_{X^\circ}(G)$ it is easier to pass to the analytic site from big-\'etale than other sites. 
\end{rem}

\begin{rem} \label{whycomor} Note that for a morphism $a: \bX \ra \cX$, and a covering $u: U \ra \cX$, the fiber product $\bX \times_\cX U$ is only an ind-scheme. Further, the notion of a morphism between ind-schemes following \cite{sk} has some ambiguity with respect to the indexing set. So to define a {\it morphism} of small \'etale sites $\acute{e}t(\bX) \ra \acute{e}t(\cX)$ (cf \cite[Definition 2.1]{giraud}) seems to require some foundational work such as defining \'etale and smooth morphisms between ind-schemes.  This is avoided in this paper by using the notion of {\it comorphism} of sites following \cite[(1.1) and Definitions 1.1, 1.2, 2.3]{giraud} recalled below.
\end{rem} 

Let $\cC$ be a category. A seive $R$ of $\cC$ is a subcategory characterised by its set of objects by the relation: if $f: X \ra Y \in Mor(\cC)$ then $Y \in R \implies X \in R$. The set of all seives is denoted $\emptyset(\cC)$. A functor $f: \cC' \ra \cC$ sends a seive $R \in \emptyset(\cC)$ to $R^f:=R \times_{\cC} \cC'$. Suppose that for every $S \in Ob(\cC)$ we are given a non-empty set $J(S)$ of seives of the comma category $\cC/S$ satisfying the condition:  
\begin{enumerate} \label{s1}
\item[(S.1)]  $\forall f:T \ra S, J(S)^f \subset J(T).$
\end{enumerate} One says that a collection $\{J(S) \subset \emptyset(\cC/S)| S \in Ob(\cC)\}$ define a topology on $\cC$ if
\begin{enumerate}
\item[(S.2)] $\forall S \in Ob(\cC)$, $\forall C \in \emptyset(\cC/S)$, if there exists $R \in J(S)$ such that $\forall f:X \ra S \in Ob(R)$, $C^f \in J(X)$, then $C \in J(S)$.
\item[(S.3)] $\forall S \in Ob(\cC)$, if $C \in \emptyset(\cC/S)$ and if there exists $R \in J(S)$ such that $R \subset C$, then $C \in J(S)$.
\end{enumerate}
A site is a category with a topology. The $R \in J(S)$ are called refinings of $S$.

\begin{defi} A comorphism of sites $f:\cE \ra \cE'$ is a functor on the underlying categories $f: E \ra E'$ such that for every object $X \in ob(E)$, and every seive $R'$ in $(\cE' \downarrow X')$ where $X'=f(X)$ in $\cE'$, the seive $R' \times_{(E' \downarrow X')} (E \downarrow X)$ belongs to $\cE$.
\end{defi}
Let us recall how a covering gives rise to a seive.  Let $\cR=\{r_i: R_i \ra S | i \in I\}$ be a family of arrows with the range. They define a seive $R$ of $\cC/S$ whose objects are $f: X \ra S$ such that there exists a $i \in I$ such that $\Hom_S(X,R_i) \neq \emptyset$. {\it So to describe the topology whenever it is convenient, we will specify only the coverings.} The verification of (S.1) will be usually obvious.

Let $C$ be the category of schemes with \'etale topology. Let $\mathbb{X}$ be an ind-scheme. We define the big \'etale site $\acute{E}t(\mathbb{X})$ of $\mathbb{X}$ as follows: 
\begin{enumerate}
 \item[(1)] objects are morphism $u:U \ra \bX$ factoring through $X_n$ for some $n \in I$ (\ref{xntoxx}). Such $u: U \ra \bX$ are called opens of $\bX$.
 \item[(2)] a morphism from $u:U \ra \bX$ to $u':U' \ra \bX$ is a morphism $f:U \ra U'$ of schemes over $\bX$.
 \item[(3)] a covering of $u: U \ra \bX$ is just a covering of $U$ in $C$.
\end{enumerate}

Let $f: \mathbb{X} \ra \mathbb{Y}$ be a morphism of ind-schemes. We obtain a comorphism of sites 
\begin{equation}
\acute{E}t(f): \acute{E}t(\bX) \ra \acute{E}t(\bY),
\end{equation}
where the functor $\acute{E}t(f)$ sends $u: U \ra \bX$ to $f \circ u: U \ra \bY$. 

Let $\tilde{\mathbb{X}}$ (or sometimes $\mathbb{X}_{Et}$) denote the category of sheaves of sets on $\acute{E}t(\mathbb{X})$. It admits the following alternate description. A sheaf of sets $\cF$ on an ind-scheme $\mathbb{X}$ is a collection of sheaves of sets $\cF_n$ on $X_n$, where $n \in I$, together with 
morphisms $\phi_{j,k}^\cF:i_{j,k}^{-1}(\cF_k) \ra \cF_j$, whenever $j \leq k$, satisfying an obvious co-cycle condition: let $j \leq k \leq l$, then
\begin{equation}
\phi_{j,l}=\phi_{j,k} \circ i_{j,k}^{-1}(\phi_{k,l}).
\end{equation}
Let $Ab(\mathbb{X})$ denote the subcategory of abelian group objects in $\tilde{\mathbb{X}}$. It is well known that the category of abelian presheaves on any Grothendieck site is an abelian category with enough injectives (cf \cite[(I.2.1.1)(I.2.1.2)]{tamme}).

\subsubsection{Pull-back and pushforward of (pre)sheaves} Let $\cF \in \tilde{\bY}$ and $f: \bX \ra \bY$ be a morphism. Let $u: U \ra \bX$ be an open of $\cX$. We define {\it pull-back presheaf} by
\begin{equation}
(f^*\cF)_{(U,u)}=\cF_{(U,f \circ u)}.
\end{equation}
The {\it pull-back sheaf} is obtained by sheafifying it following \cite[Prop 2.6, $2^\circ$]{giraud}.

Similarly, let $\cG \in \tilde{\bX}$. We view it as a presheaf. Let $u: U \ra \bY$ be an open of $\bY$. Consider the comma category $(\acute{E}t(f),u)$. It consists of commutative diagrams
\begin{equation} \label{pushforward}
\xymatrix{ 
V \ar[r]_v \ar[d]^f & \bX \ar[d]^f \\
U \ar[r]^u & \bY
} \quad
\xymatrix{
V_2 \ar@/^1pc/[rr]_{v_2} \ar[r]_{\theta} \ar[rd]_{f_2} & V_1 \ar[r]_{v_1} \ar[d]^{f_1} & \bX \ar[d]^{f} \\
& U \ar[r]^{u} & \bY
}
\end{equation}
where objects are $v: V \ra \bX$ lying over $u:U \ra \bY$ and morphisms $\theta:V_2 \ra V_1$ fit compatibly as above. It will be convenient to denote $V \ra U$ also by $f$.
 
For push-forward, one considers $(V_1,v_1)$ to be finer than $(V_2,v_2)$ because $\cG_{v_1} \ra \theta_* \cG_{v_2}$. Taking limit over $(\acute{E}t(f),u)$, one defines the {\it push-forward presheaf}
\begin{equation} \label{pushforwardindschmes} (f_* \cG)_{(U,u)}=\varprojlim f_*\cG_{(V,v)}.
\end{equation}
Actually $f_* \cG$ is already a sheaf if $\cG$ is a sheaf by \cite[Prop 2.6, $2^\circ$]{giraud}. Indeed, this follows because pushforward commutes with arbitrary limits and the sheaf condition is formulated in terms of limits. Indeed, by \cite[Chapitre 0, \S 2]{giraudbook} or \cite[(1.2)]{giraud}, a presheaf $F$ is a sheaf on a site $E$ if for every object $S \in ob(E)$, and every refining $R \in J(S)$ the natural map $F(S) \ra \varprojlim (F|_{R})$ is bijective.
By \cite[page 196]{giraud}, the functor $f^*$ commutes with finite limits. So we have a morphism of topoi of sheaves on ind-schemes
\begin{equation} \label{mortopoi}
(f^*,f_*):\tilde{\bX} \ra \tilde{\bY}.
\end{equation}

\subsection{Comorphism between sites of ind-schemes and stacks}
Let $\cX/S$ be an algebraic stack. Let us recall the big-\'etale site $\acute{E}tale(\cX)$ of $\cX$ defined as follows:
\begin{enumerate} \item[(1)] The objects are representable  $1$-morphisms $u:U \ra \cX$ of $S$-algebraic stacks, where $U$ is a scheme.
 \item[(2)] A morphism from $(U_1,u_1)$ to $(U_2,u_2)$ consists of a pair $(\phi,\alpha)$ where 
 \begin{equation}
  \xymatrix{
  U_1 \ar@/_1pc/[rr]^{u_1} \ar[r]^{\phi} & U_2 \ar[r]^{u_2}  & \cX 
      }
 \end{equation} and $\alpha:u_1 \ra u_2 \circ \phi$ is a $2$-morphism.
 
 \item[(3)] Coverings are families $\{(\phi_i,\alpha_i):(U_i,u_i) \ra (U,u), i \in I\}$ such that the $1$-morphism of $S$-schemes
 $\cup_i \phi_i: \cup_i U_i \ra U$ is \'etale and surjective.
\end{enumerate}
We denote by $\cX_{\acute{E}t}$ the big-\'etale topos of sheaves on  $\cX$.
We shall work with Artin stacks in this paper. So we will tacitly assume the existence of a smooth atlas. Further the diagonal map $\Delta: \cX \ra \cX \times_S \cX$ is representable.

By a morphism $a: \bX \ra \cX$  from an ind-scheme to a stack we mean a sequence of morphisms $a_j: X_j \ra \cX$ such that through  $X_j \hra X_k$, $a_k$ restricts to $a_j$. Given an open $u: U \ra \bX$ in $\acute{E}t(\bX)$, since it factors through some $X_n$, so let agree to define $a \circ u$ as $a_n \circ u$.  We thus obtain a comorphism of sites (cf \cite[D\'efinition 2.3]{giraud})
\begin{equation}
\acute{E}t(a): \acute{E}t(\bX) \ra \acute{E}t(\cX),
\end{equation}
where the functor $\acute{E}t(a)$ sends $u: U \ra \bX$ to $a \circ u: U \ra \cX$.

\subsubsection{Pull-back and push-forward of sheaves and presheaves}
To define pullback of (pre)sheaves it will be notationally convenient below to use $(?)^*$ instead of $(?)^{-1}$ even for abelian sheaves. 
 Let $\cF \in \cX_{\acute{E}t}$. We view it as a presheaf. Let $v: V \ra \bX$ be an open of $\bX$. Let us consider the comma category $(v \downarrow \acute{E}t(a))$. It consists of commutative diagrams
\begin{equation} \label{pullback}
\xymatrix{ 
V \ar[r]_v \ar[d]^a & \bX \ar[d]^a \\
U \ar[r]^u & \cX
} \quad
\xymatrix{
& V \ar[r]_v \ar[d]^{a_1} \ar[ld]_{a_2} & \bX \ar[d]^{a} \\
U_2 \ar@/_1pc/[rr]^{u_2} & U_1 \ar[r]^{u_1} \ar[l]_{\theta} & \cX 
}
\end{equation}
of objects $u:U \ra \cX$ fitting below $v: V \ra \bX$ and morphisms $\theta:U_1 \ra U_2$ fitting compatibly as above. It will be convenient to denote $V \ra U$ also by $a$.

For pull-back, one considers the open $(U_1,u_1)$ as finer than $(U_2,u_2)$ because $\theta^* \cF_{u_2} \ra  \cF_{u_1}$.
Taking colimit over  $(v \downarrow \acute{E}t(a))$, one defines 
\begin{equation}
(a^*\cF)_{(V,v)}=\varinjlim a^*\cF_{(U,u)}
\end{equation}
Let $\phi:(V_1,v_1) \ra (V_2,v_2)$ be a morphism between opens of $\mathbb{X}$. We have a morphism 
\begin{equation}
\phi^* (a^* \cF)_{(V_2,v_2)} \ra (a^* \cF)_{(V_1,v_1)}
\end{equation}
because any open $(U,u)$ under $(V_2,v_2)$ is also under $(V_1,v_1)$.
{\it This defines the pull-back presheaf $a^*\cF$.} We note that pull-back commutes with finite limits, and so is afortiori an exact functor. We remark a simplification in the present case. The category $(v \downarrow \acute{E}t(a))$ has $\acute{E}t(a)(v)=a \circ v: V \ra \cX$ as a final object with $a_1=Id_{V}$. Thus 
\begin{equation}
(a^*\cF)_{(V,v)}=\cF_{(V,a \circ v)}.
\end{equation}
{\it For $\cF \in \bX_{\acute{E}t}$, we define the pull-back sheaf $\cF$
as the sheaf associated to this presheaf.}

Similarly, let $\cG \in \bX_{\acute{E}t}$. We view it as a presheaf. Let $u: U \ra \cX$ be an open of $\cX$. The comma category $(\acute{E}t(a),u)$ consists of commutative diagrams
\begin{equation} \label{pushforward1}
\xymatrix{ 
V \ar[r]_v \ar[d]^a & \bX \ar[d]^a \\
U \ar[r]^u & \cX
} \quad
\xymatrix{
V_2 \ar@/^1pc/[rr]_{v_2} \ar[r]_{\theta} \ar[rd]_{a_2} & V_1 \ar[r]_{v_1} \ar[d]^{a_1} & \bX \ar[d]^{a} \\
& U \ar[r]^{u} & \cX
}
\end{equation}
where objects are $v: V \ra \bX$ lying over $u:U \ra \cX$ and morphisms $\theta:V_2 \ra V_1$ fit compatibly as above. It will be convenient to denote $V \ra U$ also by $a$.
 
For push-forward, one considers $(V_1,v_1)$ to be finer than $(V_2,v_2)$ because $\cG_{v_1} \ra \theta_* \cG_{v_2}$. Taking limit over $(\acute{E}t(a),u)$, one defines 
\begin{equation} \label{pushforwardindschemestack} (a_* \cG)_{(U,u)}=\varprojlim a_*\cG_{(V,v)}.
\end{equation}
Let $(\phi,\alpha):(U_1,u_1) \ra (U_2,u_2)$ be a morphism between opens of $\cX$. We have a morphism 
\begin{equation}
(a_* \cG)_{(U_2,u_2)} \ra \phi_* (a_* \cG)_{(U_1,u_1)}
\end{equation}
because any open $(V,v)$ over $(U_1,u_1)$ is also over $(U_2,u_2)$. This defines the pushforward presheaf $a_* \cG$. Actually $a_*\cG$ is already a sheaf because push-forward commutes with arbitrary limits.

\begin{rem} Taking $\acute{E}t(a): \acute{E}t(\bX) \ra \acute{E}t(\cX)$ as $u: E \ra E'$ in \cite[Lemme 2.5]{giraud}, we see that $a^*=u^\bullet$ and $a_*=u_\bullet$ (cf also \cite[Proposition 2.6 (2)]{giraud} where we set $f: \cE \ra \cE'$ as $\acute{E}t(a)$). So the constructions of $a^*$ and $a_*$ are just specializations to our case of the constructions in \cite{giraud} for the case of comorphisms. In particular, one has the following adjunction. We give below a self-contained proof.
\end{rem}

\begin{prop} \label{adjunction} Let $a: \bX \ra \cX$ be a morphism. Let $\cG \in \cX_{\acute{E}t}$ and $\cF \in \bX_{\acute{E}t}$. Then we have $Mor_{\bX}(a^* \cG,\cF) = Mor_{\cX}(\cG,a_* \cF)$.

\end{prop}
\begin{proof} A morphism in $Mor_{\bX}(a^* \cG,\cF)$ is equivalently a functorial assignment for every $(V,v) \in \acute{E}t(\bX)$ of an element in $Mor_{(V,v)}((a^*\cG)_{(V,v)},\cF_{(V,v)})$. This last set is $\varprojlim Mor((a^*(\cG_{(U,u)}),\cF_{(V,v)})$ where one varies over $(U,u) \in \acute{E}t(\cX)$ sitting under $(V,v)$ and $a$ is the corresponding morphism.
Thus a  morphism in $Mor_{\bX}(a^* \cG,\cF)$ is a functorial assigment for $\acute{E}t(\bX) \times \acute{E}t(\cX)$ such that $(V,v)$ sits on $(U,u)$ of an element in $Mor_{(V,v)}(a^* (\cG_{(U,u)}),\cF_{(V,v)})$. Now this last set equals $Mor_{(U,u)}(\cG_{(U,u)},a_*(\cF_{(V,v)}))$ by the usual adjunction between $(a^*,a_*)$ the case of schemes. This same description is found for $Mor_{\cX}(\cG,a_* \cF)$ by similar reasoning.
\end{proof}
Since $a^*$ commutes with finite limits, so we have a morphism of topoi
\begin{equation} \label{mortopoiindschemestack}
(a^*,a_*): \bX_{Et} \ra \cX_{Et}.
\end{equation}

\subsection{Simplicial Ind-schemes}

A simplicial ind-scheme $\bX_\bullet$ is a simplicial object in the category of ind-schemes. Its $n$-simplices will be denoted by $\bX_n$. 
\subsubsection{\'Etale site of simplicial ind-scheme} \label{essis}
Let $\mathbb{Y}_\bullet$ be a simplicial ind-scheme. 
Let us consider the big \'etale site of $\mathbb{Y}_\bullet$ (cf \cite[Defn 6.1]{conrad}). It is a category whose
\begin{enumerate}
\item objects are maps $U \ra \bY_n$ for some $n$.
\item morphims from $U \ra \bY_n$ to $U' \ra \bY_{n'}$ is a pair $(f,\phi)$ that fit in a diagram 
\begin{equation} \label{mor}
\xymatrix{
U \ar[r]_f \ar[d] & U' \ar[d] \\
\bY_{n} \ar[r]^{Y(\phi)}    & \bY_{n'}
}
\end{equation}
where $\phi:[n'] \ra [n]$ is a non-decreasing map and $f:U \ra U'$ is any morphism of schemes.
\item a covering $\{U_i \ra U\}$ of $U \ra \bY_n$ is a collection $(f_i,id)$ such that the images of $f_i$ cover $U$ set-theoretically and each $f_i$ is \'etale.
\end{enumerate}
Let $\tilde{\mathbb{Y}}_\bullet$ denote the category of sheaves of sets on this site. It may equivalently be defined as follows. A sheaf
of sets on $\mathbb{Y}_\bullet$ is a collection $\cF_n \in \tilde{\bY_n}$ of sheaves on $\bY_n$ together with morphism $[\phi]$ of sheaves on $\bY_n$,
\begin{equation}
 [\phi]: Y(\phi)^{-1} (\cF_{n'}) \ra \cF_n
\end{equation}
for every $\phi \in Mor_{\Delta}([n'],[n])$, satisfying an obvious co-cycle relation. Let $Ab(\tilde{\mathbb{Y}}_\bullet)$ denote the sub-category of abelian group objects in $\tilde{\mathbb{Y}}_\bullet$. 

We have a natural restriction morphism $Res_n: \tilde{\bY}_\bullet \ra \tilde{\bY}_n$. It is exact. It has both left and right adjoint \cite[\S 2.4]{olsson}. Using its right adjoint, we can show that there are enough injective objects in the category of abelian presheaves
on {\it simplicial ind-schemes} since there are enough injective objects in the category of abelian presheaves on ind-schemes.  
Further various operations like sheafification relative to this site, formation of images and quotients by 
equivalence relations, checking epimorphisms and monomorphisms, computing inverse and direct limits may be carried out 
"degree-by-degree".  Let $u_\bullet: \mathbb{X}_\bullet \ra \mathbb{Y}_\bullet$ be a map of simplicial ind-schemes. One defines the functors
$$u_{\bullet,*}: \tilde{\mathbb{X}}_\bullet \ra \tilde{\mathbb{Y}}_\bullet \quad u_\bullet^*: \tilde{\mathbb{Y}}_\bullet \ra \tilde{\mathbb{X}}_\bullet,$$
by term-by-term construction of pushforward and pullback.

Let us consider the special case of augmented simplicial spaces $a: \bY_\bullet \ra \cX$. Let $\cX_\bullet$ denote the constant simplicial space and $a_\bullet: \bY_\bullet \ra \cX_\bullet$ the associated map. We thus obtain a comorphism of sites (cf \cite[D\'efinition 2.3]{giraud})
\begin{equation}
\acute{E}t(a_\bullet): \acute{E}t(\bY_\bullet) \ra \acute{E}t(\cX_\bullet),
\end{equation}
where the functor $\acute{E}t(a_\bullet)$ sends $u: U \ra \bY_n$ to $a_n \circ u: U \ra (\cX_\bullet)_n=\cX$.

The category $\tilde{\cX}_\bullet$ is canonically identified with the category of co-simplicial objects in the category of sheaves of sets on $\cX$. For presheaves, the pull-back functor
\begin{equation} \label{pullbackstackindscheme}
a^*_\bullet: \tilde{\cX} \ra \tilde{\bY}_\bullet, 
\end{equation}
defined as $a^*_\bullet(\cG)^n = a_n^*(\cG)$ in each degree and with natural face and degeneracy relations. Let $\cF \in \tilde{\cX}$ and $u: U \ra \bY_n$, then $a_\bullet^*(\cF)_u$ is simply the restriction $(\cF)_{a_n \circ u}$. So pull-back is an exact functor on presheaves. For sheaves, we take pull-back of $\cG$ as a presheaf and then sheafify it.  It has a right adjoint 
\begin{equation} \label{directimageindschemestack} a_{*}: \tilde{\bY}_\bullet \ra \tilde{\cX},
\end{equation}
given by defining $a_*(\cF)$ as the kernel equalizer of the two "face"-morphisms $ a_{0*} \cF^0 \ra a_{1*} \cF^1$. The proof of the adjunction $(a^*_\bullet,a_*)$ is word-to-word generalisation of the standard proof in the case $\bY_\bullet$ is a simplicial scheme. The only non-obvious ingredient was proven in Prop \ref{adjunction}. So we have a morphism of topoi
\begin{equation}
(a_\bullet^*,a_*):(\bY_\bullet)_{\acute{E}t} \ra \cX_{\acute{E}t}.
\end{equation}

\subsubsection{A spectral sequence}
Let $\mathbb{Y}_\bullet$ be a simplicial ind-scheme. Let $\cF_\bullet \in Ab({\bY}_\bullet)$. We define $\Gamma(\mathbb{Y}_\bullet,\cF_\bullet)$ as the kernel equalizer of the two "face"-morphisms
\begin{equation} \label{globalsecsimpsch} \Gamma(\mathbb{Y}_0,\cF_0) \ra \Gamma(\mathbb{Y}_1,\cF_1).
\end{equation}

The following is a mild generalization of cf \cite[Thm 6.11]{conrad}. 
\begin{thm} \label{ss} Let $\mathbb{Y}_\bullet$ be a simplicial ind-scheme. Let $D_+(\mathbb{Y}_\bullet)$ denote the derived category of bounded below complexes of sheaves of abelian groups on $\mathbb{Y}_\bullet$. For any $K' \in D_+(\mathbb{Y}_\bullet)$, there is a natural spectral sequence
$$E^{p,q}_1=H^q(\mathbb{Y}_p,K'|_{\mathbb{Y}_p}) \implies H^{p+q}(\mathbb{Y}_\bullet,K'),$$
where $d^{\bullet,q}_1$ is induced from the differential complex structure along $\mathbb{Y}_\bullet$. 
\end{thm}
\begin{proof} Let $K' \ra I^\bullet$ be a quasi-isomorphism to a bounded below complex of injectives in $Ab(\mathbb{Y}_\bullet)$. Let us consider the $1$-st quadrant double complex
\begin{equation} \label{doublecomplex}
\Gamma(\mathbb{Y}_p,I^q|_{\mathbb{Y}_p})_{p,q}.
\end{equation}
Let us study page one of the spectral sequence that arizes by filtering (\ref{doublecomplex}) first by rows. If we filter (\ref{doublecomplex}) by rows, for instance taking the $n$-th row, notice that
\begin{equation}
H^0(\Gamma(\mathbb{Y}_0,I^n) \ra \Gamma(\mathbb{Y}_1,I^n) \ra \Gamma(\mathbb{Y}_2,I^n) \ra \cdots)=\Gamma(\mathbb{Y}_\bullet,I^n).
\end{equation}
At page one, these groups fit vertically on the "$0$-th" column to make a complex $\cdots \ra \Gamma(\mathbb{Y}_\bullet,I^n) \ra \Gamma(\mathbb{Y}_\bullet,I^{n+1}) \ra \cdots$ with differentials given by $I^n$. So the $0$-column on page one is isomorphic  to $R\Gamma(\mathbb{Y}_\bullet,K')$ in $D_+(\bY_\bullet)$.

\begin{lem} In (\ref{doublecomplex}) all horizontal $H^0$ away from degree zero vanish.
\end{lem}
\begin{proof}  An injective object in $Ch_{\geq 0}(Ab)$ is necessarily acyclic in positive degrees. It suffices to show that for an injective sheaf $I \in Ab(\mathbb{Y}_\bullet)$ the complex
\begin{equation}
\Gamma(\mathbb{Y}_0,I) \ra \Gamma(\mathbb{Y}_1,I) \ra \Gamma(\mathbb{Y}_2,I) \ra \cdots
\end{equation}
is an injective object in $Ch_{\geq 0}(Ab)$. To this end, it suffices to construct 
an exact left-adjoint to the functor $Ab(\mathbb{Y}_\bullet) \ra Ch_{\geq 0}(Ab)$ which sends $\cF \ms \{ \Gamma(\mathbb{Y}_p,\cF|_{\mathbb{Y}_p})\}_{p \geq 0} $ where the differentials are given by alternating sum of face maps. Let $C_\bullet \in Ch_{\geq 0}(Ab)$. Applying the Dold-Kan functor we get a simplicial abelian group. Let $C_n \in Ab$ be the $n$-simplices. We associate to $C_n$ the sheaf $\tilde{C_n}$ which on every non-empty open $U$ of $\mathbb{Y}_n$ takes value $C_n$. These $\tilde{C_n}$ define the abelian sheaf $\tilde{C}_\bullet$ on $\mathbb{Y}_\bullet$.  Let us call this functor $dk:Ch_{\geq 0}(Ab) \ra Ab(\mathbb{Y}_\bullet)$. It is exact. Further we have the desired adjunction
$Hom_{Ab(\mathbb{Y}_\bullet)}(dk \{ C_\bullet \},\cG^\bullet)=Hom_{Ch_{\geq 0}(Ab)}( C_\bullet, \{ \Gamma(\mathbb{Y}_p,\cG|_{\mathbb{Y}_p}) \}).$

\end{proof} 
In conclusion, at page one only groups in column zero survive. Thus the horizontal spectral sequence abuts to the cohomology groups of $R\Gamma(\mathbb{Y}_\bullet,K')$.

Now we study page one of the spectral sequence that arizes by filtering (\ref{doublecomplex}) by columns. A column consists of $\Gamma(\mathbb{Y}_p,I^\bullet|_{\mathbb{Y}_p})$ for a fixed $p \geq 0$. Upon taking homology, we get a spectral sequence $H^q(\Gamma(\mathbb{Y}_p,I^\bullet|_{\mathbb{Y}_p})) \implies H^*(\mathbb{Y}_\bullet,K')$.
Recall that $K' \ra I^\bullet$ is a quasi-isomorphism of bounded below chain complexes
in $Ab(\mathbb{Y}_\bullet)$. So by exactness of the functor restricting sheaves on $\mathbb{Y}_\bullet$ to $\mathbb{Y}_p$ we  have $K'|_{\mathbb{Y}_p} \ra I^\bullet|_{\mathbb{Y}_p}$ is also exact. Further $I^n|_{\mathbb{Y}_p}$ is an injective sheaf. Indeed, the restriction functor $res_n: Ab(\bY_\bullet) \ra Ab(\bY_n)$ has an exact left adjoint given by 
$$(l_n(G))_k = \oplus_{\rho \in \Hom_{\Delta}([n],[k])} \rho^* G.$$
It is exact because $\rho^*$ is exact for each morphism $Ab(\tilde{\bY_k}) \ra Ab(\tilde{\bY_n})$.
So $K'|_{\mathbb{Y}_p} \ra I^\bullet|_{\mathbb{Y}_p}$ is a bona-fide qis of  $K'|_{\mathbb{Y}_p}$ to an injective in $Ch_{\geq 0}(Ab(\mathbb{Y}_p))$. Thus 
$$H^q(\Gamma(\mathbb{Y}_p,I^\bullet|_{\mathbb{Y}_p})) = H^q(\mathbb{Y}_p,K'|_{\mathbb{Y}_p}).$$
So $H^q(\mathbb{Y}_p,K'|_{\mathbb{Y}_p})$ fit to form the page one of the column spectral sequence. Thus we get the desired spectral sequence.

\end{proof}

\subsection{Cohomological descent for maps from an ind-scheme to a stack} \label{cohdessec}

\subsubsection{Morphism of sites}
Let $\epsilon: \bX_\bullet \ra \cX$ be a simplicial ind-scheme over an Artin stack $\cX$.
Let $\epsilon_n: \bX_n \ra \cX$ 
denote the projection morphisms. We have a comorphism of sites (cf \cite[Definition 2.3]{giraud})
\begin{equation}
\epsilon^{-1}:  \acute{E}t(\bX_\bullet) \ra \acute{E}t(\cX)
\end{equation}
which sends $([n],u:U \ra \bX_n)$ to $(U,u \circ \epsilon_n)$.
This morphism induces a pair of adjoint functors (cf \ref{pullbackstackindscheme}, \ref{directimageindschemestack}, \S \ref{essis})
\begin{equation}
 (\epsilon^*,\epsilon_*):  (\bX_\bullet)_{\acute{E}t} \ra \cX_{\acute{E}t}.
\end{equation}

\subsubsection{Cohomological Descent}
We say that $\epsilon_\bullet:\bX_\bullet \ra \cX$ is a morphism of {\it cohomological descent} if
the natural transformation 
\begin{equation} \label{cohdesdefn} id \ra R\epsilon_* \circ \epsilon^*,
\end{equation}
is an isomorphism on the derived category of bounded below complexes of abelian sheaves $D_+(\cX)$ on $\cX$.  We say $\epsilon_\bullet$ is universally of cohomological descent it it is remains so under arbitrary base-change. We shall say that a morphism $\epsilon: \bX \ra \cX$ from an ind-scheme to an algebraic stack is a morphism of cohomological descent if $\epsilon_\bullet:=cosk_0(\epsilon)$ is so (cf \ref{fpindsch} for fibered products of ind-schemes).

\begin{thm} \label{abutmentucd} Suppose that $a: \mathbb{Y}_\bullet \ra \cX$ is augmented. Let $K' \in D_+(\mathbb{Y}_\bullet)$. We have a canonical spectral sequence 
\begin{equation}
E^{p,q}_1=R^q a_{p,*}(K'|_{\mathbb{Y}_p}) \implies R^{p+q} a_* (K'),
\end{equation} which is functorial in $a$. Suppose that $a$ is universally of cohomological descent. Let $K \in D_+(\cX)$ be a complex and set $K'=a^*K$. The  spectral sequence of Theorem \ref{ss} abuts to $H^{p+q}(\cX,K)$. 
\end{thm}
\begin{proof} This spectral sequence is a sheafified version of the one in Theorem \ref{ss}. It is constructed by working with $a_{p,*}$ and $a_*$ instead of $\Gamma(\mathbb{Y}_p,\bullet)$ and $\Gamma(\mathbb{Y}_\bullet,\bullet)$ of (\ref{globalsecsimpsch}) respectively. For $K \in D_+(\cX)$ we have the following isomorphism $$R\Gamma(\cX,K) \stackrel{coh des}{\lra} R\Gamma(\cX, Ra_* \circ a^*K) \simeq R(\Gamma(\cX, \bullet) \circ a_*) (a^*K)\stackrel{Def \, \ref{globalsecsimpsch}}{=} R\Gamma(\mathbb{Y}_\bullet,a^*K),$$ which is functorial in $a$. Thus we get 
$$H^i(\cX,K) \simeq H^i(\mathbb{Y}_\bullet,a^*K),$$
which is also natural in $a$. This shows the equality of abutment terms.
\end{proof}

The following is a mild generalization of \cite[Thm 7.2]{conrad} to the case of a morphism $\epsilon:\bX \ra \cX$ from an ind-scheme to an algebraic stack.

\begin{defi} We say that $\epsilon$ admits a section $s: \cX \ra \bX$ if there is a morphism $s: \cX \ra X_n$ for some $n \in I$ such that $\epsilon_n \circ s = Id_{\cX}$.
\end{defi}
\begin{defi} \label{etlocalsection} We say that $\epsilon$ admits \'etale local sections if any neighbourhood $u: U \ra \cX$, admits an \'etale surjective map $\theta: U' \ra U$ such that making make fiber product of ind-scheme by (\ref{fpindsch}) $\epsilon_{U'}$ admits a section 
\begin{equation} \label{bc}
\xymatrix{
\bX_{U'} \ar[r]_{\tilde{\theta}} \ar[d]^{\epsilon_{U'}} & \bX_U \ar[d]^{\epsilon_U} \\
U' \ar[r]^{\theta} & U
}
\end{equation}
\end{defi}

\begin{thm} \label{et-sec-indsch-stack} Suppose that $\epsilon$ admits sections \'etale locally on $\cX$. Then $\epsilon$ is universally a morphism of cohomological descent.
\end{thm}
\begin{proof}
The universality follows simply because our hypothesis is preserved under base-change. To check that $1 \ra R \epsilon_{U*} \circ \epsilon^*_U$ is an isomorphism of functors on $D_+(U)$, by successively truncating any complex $K \in D_+(U)$ from above, it suffices to consider the case $K=\cF$ for a $\cF \in Ab(U_{\acute{E}t})$.

Let us take a local neighbourhood $u: U \ra \cX$ of $\cX$. Let $\theta: U' \ra U$ be an \'etale and surjective. Since $\theta^*:Ab(U_{\acute{E}t}) \ra Ab(U'_{\acute{E}t})$ is exact, so it preserves kernels and cokernels. Thus it suffices to check $\theta^* \ra \theta^* R \epsilon_{U*}  \epsilon^*_U $ is an isomorphism of functors. Inserting the adjunction morphism $1 \ra \tilde{\theta}_* \tilde{\theta}^*$ on $Ab(\bX_{U,\acute{E}t})$ (cf (\ref{mortopoi})) we get
$$1 \ra R\epsilon_{U*} \epsilon^*_U \ra R\epsilon_{U*}  \tilde{\theta_*}  \tilde{\theta^*} \epsilon^*_U =R(\epsilon_{U}  \tilde{\theta})_*  (\epsilon_U \tilde{\theta})^* = R (\theta  \epsilon_{U'})_*  (\theta  \epsilon_{U'})^* =\theta_*  R \epsilon_{U'*}   \epsilon_{U'}^*\theta^*,$$
i.e $1 \ra R\epsilon_{U*} \epsilon^*_U \ra \theta_*  R \epsilon_{U'*}   \epsilon_{U'}^*\theta^*$. Here we remark that the composed natural transformation is also given by natural adjunction. This  by adjuction gives 
$ \theta^* \ra \theta^* R \epsilon_{U*} \epsilon_U^* \ra     R \epsilon_{U'*}   \epsilon_{U'}^*  \theta^*$. We have thus shown a natural factorization of natural transformations
\begin{equation}
\xymatrix{
\theta^* \ar@/^1pc/[rr] \ar[r]  & \theta^* R \epsilon_{U*} \epsilon_U^*  \ar[r] & R \epsilon_{U'*} \circ  \epsilon_{U'}^* \theta^*,
}
\end{equation}
where arrows from $\theta^*$ are obtained from adjunction of $\epsilon_U$ and $\epsilon_{U'}$. This shows that if in (\ref{bc}) $\epsilon_{U'}$ is of cohomological descent then so is $\epsilon_U$.  So we may suppose that $\bX_U \ra U$ itself has  a section $s: U \ra \bX_U$. Let us abbreviate $\bX_U$ to  $\bX$ and $\epsilon_U$ to $\epsilon$.

 Set $K'=\epsilon^* \cF \in D_+(\bX_\bullet)$. By Theorem \ref{abutmentucd}, we have a spectral sequence
\begin{equation}
E^{p,q}_1=R^q\epsilon_{p*}(K'|_{\bX_p}) \implies R^{p+q}\epsilon_* (K').
\end{equation}
Here the $q$-th row has differential given by the simplicial structure on $\bX_\bullet$.
Let us abbreviate $\bX$ simply by $X$ from now.

Note $E^{\bullet,q}_1$ makes sense for $p \geq -1$, where $\epsilon_{-1}=id_{\cX}$. Thus $E^{-1,q}_1=0$ for $q \geq 1$ since push-forward along the identity map has vanishing higher derived functors. For $p \geq 1$, consider
\begin{equation}
h_p= s \times id_{\bX_p}: \bX_p = \bX^{\times_{\cX} (p+1)} \ra \bX^{\times_{\cX} (p+2)}=\bX_{p+1},
\end{equation}
defined by inserting the section $s$ along the $0$-th coordinate of $\bX_{p+1}$. We have thus
\begin{eqnarray} \label{hreln}
\bX(\partial^0_p)  h_p = id_{\bX_p} & p \geq 0 \\
h_{p-1}  \bX(\partial^j_{p-1}) = \bX(\partial^{j+1}_p) h_p & \forall p \geq 1, 0 \leq j \leq p-1.
\end{eqnarray}
When $K'=\epsilon_*\cF$, by $K'|_{\bX_p}=h_p^*(K'|_{\bX_{p+1}})$, we have
$$K'|_{\bX_{p+1}} \ra Rh_{p*} h_p^*(K'|_{\bX_{p+1}})=Rh_{p*}(K'|_{\bX_p}).$$
Now applying $R\epsilon_{p+1,*}$ we get
$$R\epsilon_{p+1,*}(K'|_{\bX_{p+1}}) \ra R\epsilon_{p+1,*}Rh_{p*} h_p^*(K'|_{\bX_{p+1}}) = R\epsilon_{p+1,*}Rh_{p*}(K'|_{\bX_p})=R a_{p,*}(K'|_{\bX_p}).$$
So we get maps $E^{p+1,q}_1 \ra E^{p,q}_1$.

\begin{lem} These maps form a homotopy between the identity and zero maps on the augmented differential complexes $E^{\bullet,q}_1$ (in degrees $\geq -1$).
\end{lem}
\begin{proof} Let $p \geq 0$. Let us compute $
h_p^* d^{p,q}_1 + d^{p,q}_1 h_{p-1}^* $
\begin{eqnarray}
= h_p^* (\sum_{j=0}^{p+1} (-1)^j \bX(\partial_p^j)^*) + (\sum_{k=0}^{p-1+1} (-1)^k \bX(\partial_{p-1}^k)^*) h_{p-1}^* \\
=h_p^* \bX(\partial^0_p)^* + \sum_{k=0}^{p-1+1} (-1)^k( - h_p^* \bX(\partial^j_p)^* + \bX(\partial^k_{p-1})^* h_{p-1}^*) \\
\stackrel{(\ref{hreln})}{=} id +0 =id.
\end{eqnarray}
For $p=-1$, we have $d^{-1,q}_1=\bX(\partial^0_{-1})^*$. Now $h_{-1}^* d_1^{-1,q}=h_{-1}^*\bX(\partial_{-1}^0)^*=(\bX(\partial_{-1}^0) h_{-1})^*$ which by $(\ref{hreln})$ equals $id^*$. Hence the augmented complex is acyclic.

\end{proof}
Now let us consider the terms of the spectral sequence restricted to the $1$st  quadrant. Since $E^{-1,q}_1=0$ for $q \geq 1$, thus the rows for $q \geq 1$ remain exact upon restriction to the $1$st quadrant. For $q=0$, we have $E^{p,0}_1=\epsilon_{p,*} \epsilon_p^* \cF$. The kernel of $\epsilon_{0,*}\epsilon_0^* \cF = E^{0,0}_1 \ra E^{1,0}_1 = \epsilon_{1,*}\epsilon_{1}^* \cF$ identifies with $\cF$ via the natural inclusion $\cF \ra \epsilon_{0,*}\epsilon_0^* \cF$. So at the $E_2$ stage, the spectral sequence collapses to just the $(0,0)$-position. Hence the total complex has vanishing homology in higher degrees and homology in degree zero is given by $\cF$. Thus $R^n\epsilon_* \epsilon^*\cF=0$ for $n \geq 1$ and the natural map $\cF=E^{0,0}_2 \ra R^0\epsilon_* \epsilon^*\cF$ is an isomorphism. Thus $\cF \ra R \epsilon_* \epsilon^* \cF$ is an isomorphism of complexes of abelian sheaves on $\cX$. Thus $\epsilon$ is a morphism of cohomological descent.
\end{proof}

\bibliographystyle{plain}
\bibliography{BrauerOct18}

\end{document}